\documentclass[12pt, oneside, reqno]{amsart}

\usepackage{amsfonts,color}
\usepackage{mathabx}
\usepackage{amssymb,amscd}
\usepackage{footnote}
\usepackage[utf8]{inputenc}
\usepackage{mathtools}
\usepackage{amsthm,wasysym}
\usepackage[english]{babel}
\usepackage{bbm}
\usepackage{colonequals}
\usepackage{cancel}
\usepackage{units}
\usepackage{dsfont}
\usepackage{adjustbox}
\usepackage{eucal}
\usepackage{graphicx}
\usepackage{bm}
\usepackage{amsfonts}
\usepackage{hyperref}
\usepackage{wrapfig}
\usepackage{subcaption}
\usepackage{amsmath}
\usepackage{float}
\usepackage{stmaryrd}
\usepackage{accents}
\usepackage{pdfsync}
\usepackage{todonotes}

\usepackage{algorithm}
\usepackage{algpseudocode}

\usepackage{siunitx}
\newcolumntype{N}{c@{}S}

\usepackage{pgf,tikz,pgfplots}
\pgfplotsset{compat=1.15}
\usepackage{mathrsfs}
\usetikzlibrary{arrows}
\usetikzlibrary {arrows.meta} 
\usetikzlibrary{calc}
\def\centerarc[#1](#2)(#3:#4:#5)
{ \draw[#1] ($(#2)+({#5*cos(#3)},{#5*sin(#3)})$) arc (#3:#4:#5); }

\usepackage{chngcntr}
\counterwithin{equation}{section}

\newtheorem{theorem}{Theorem}[section]
\newtheorem{lemma}[theorem]{Lemma}
\newtheorem{prop}[theorem]{Proposition}

\newtheorem{corollary}[theorem]{Corollary}
\newtheorem{proposition}[theorem]{Proposition}
\theoremstyle{definition}
\newtheorem{definition}[theorem]{Definition}
\theoremstyle{remark}
\newtheorem{remark}[theorem]{Remark}

\allowdisplaybreaks[1]

\newcommand{\da}{\mathrm{da}} 
\newcommand{\dl}{\mathrm{dl}} 

\newcommand{\act}[1]{\langle{#1}\rangle}
\newcommand{\vG}{\varGamma}
\newcommand{\VV}{{\mathcal{V}}}
\newcommand{\Vo}{\mathring{\VV}}
\newcommand{\trans}{\prime}
\newcommand{\om}{\varOmega}
\renewcommand{\Omega}{\om}
\def\d{\partial}
\newcommand{\agl}{\sphericalangle}
\newcommand{\vphi}{\varphi}
\newcommand{\veps}{\varepsilon}
\newcommand{\Xm}[1]{\mathfrak{X}({#1})}
\newcommand{\Xmo}[1]{\mathfrak{X}_c({#1})}
\renewcommand{\forall}{\text{ for all }}

\newcommand{\xtt}{\tilde{x}}
\newcommand{\Gamtt}{\tilde{\Gamma}}
\newcommand{\dtt}{\tilde{\partial}}
\newcommand{\Rtt}{\tilde{R}}
\newcommand{\gtt}{\tilde{g}}
\newcommand{\sigmatt}{\tilde{\sigma}}
\newcommand{\startt}{\tilde{\star}}

\newcommand{\Moo}{{\breve{M}}}
\newcommand{\Yoo}{{\breve{Y}}}
\newcommand{\Poo}{{\breve{\Phi}}}

\newcommand{\xoo}{\breve{x}}
\newcommand{\doo}{\breve{\partial}}

\newcommand{\Sc}{\ensuremath{\mathcal{S}}}
\newcommand{\E}{\ensuremath{\mathscr{E}}}
\newcommand{\Eint}{\ensuremath{\mathscr{E}^{\mathrm{int}}}}

\newcommand{\V}{\ensuremath{\mathscr{V}}}
\newcommand{\Vint}{\ensuremath{\mathscr{V}^{\mathrm{int}}}}
\newcommand{\Vbnd}{\ensuremath{\mathscr{V}^{\mathrm{bnd}}}}

\newcommand{\dhat}[1]{{\bar{{#1}}}}
\newcommand{\Gex}{{\dhat{G}}}  
\newcommand{\gex}{{\dhat{g}}}  
\newcommand{\nabex}{\dhat{\nabla}}  
\newcommand{\Rmex}{\dhat{R}}

\newcommand{\Kex}{\dhat{K}}

\newcommand{\gh}{{g}}
\newcommand{\vh}{{v}}
\newcommand{\sh}{{\sigma}}
\newcommand{\sigmaappr}{{\sigma_h}}

\renewcommand{\div}{\mathrm{div}}
\DeclareMathOperator{\curl}{\mathrm{curl}}
\DeclareMathOperator{\rot}{\mathrm{rot}}
\DeclareMathOperator{\inc}{\mathrm{inc}}

\DeclareMathOperator{\grad}{\mathrm{grad}}  
\newcommand{\W}{\mathord{\adjustbox{valign=B,totalheight=.6\baselineskip}{$\bigwedge$}}}
\newcommand{\TT}{\mathcal{T}}
\newcommand{\og}{\omega}
\newcommand{\mt}[1]{[{#1}]}

\newcommand{\EuclBasis}{E}
\newcommand{\GBasis}{e}

\newcommand{\jump}[2][]{\ensuremath{\ifthenelse{\equal{#1}{}}{\llbracket #2 \rrbracket}{\llbracket #2 \rrbracket}_{#1}}}
\newcommand{\jmp}[1]{\ensuremath{\llbracket #1 \rrbracket}}

\newcommand{\vol}[1]{\mathrm{vol}_{{#1},g}}
\newcommand{\mat}[1]{\ensuremath{\begin{pmatrix}#1\end{pmatrix}}}	
\newcommand{\diag}[1]{\text{diag}(#1)}								

\newcommand{\N}{\ensuremath{\mathbb{N}}}

\newcommand{\R}{\ensuremath{\mathbb{R}}}

\newcommand{\T}{\ensuremath{\mathscr{T}}} 
\newcommand{\Teuc}{\tilde{T}}

\newcommand{\Ltwo}[1][]{\ensuremath{L^2\ifthenelse{\equal{#1}{}}{}{(#1)}}}
\newcommand{\Linf}[1][]{\ensuremath{L^{\infty}\ifthenelse{\equal{#1}{}}{}{(#1)}}}
\newcommand{\Lp}[1][]{\ensuremath{L^{p}\ifthenelse{\equal{#1}{}}{}{(#1)}}}
\newcommand{\Winf}[1][]{\ensuremath{W^{1,\infty}\ifthenelse{\equal{#1}{}}{}{(#1)}}}
\newcommand{\Winfh}[1][]{\ensuremath{W_h^{1,\infty}\ifthenelse{\equal{#1}{}}{}{(#1)}}}
\newcommand{\Wsp}[1][]{\ensuremath{W^{s,p}\ifthenelse{\equal{#1}{}}{}{(#1)}}}
\newcommand{\Wsph}[1][]{\ensuremath{W_h^{s,p}\ifthenelse{\equal{#1}{}}{}{(#1)}}}
\newcommand{\Hone}[1][]{\ensuremath{H^1\ifthenelse{\equal{#1}{}}{}{(#1)}}}
\newcommand{\Honeh}[1][]{\ensuremath{H^1_h\ifthenelse{\equal{#1}{}}{}{(#1)}}}
\newcommand{\Honez}[1][]{\ensuremath{H^1_0\ifthenelse{\equal{#1}{}}{}{(#1)}}}
\newcommand{\Hmone}[1][]{\ensuremath{H^{-1}\ifthenelse{\equal{#1}{}}{}{(#1)}}}
\newcommand{\Hsh}[1][]{\ensuremath{H^s_h\ifthenelse{\equal{#1}{}}{}{(#1)}}}
\newcommand{\Htwo}[1][]{\ensuremath{H^2\ifthenelse{\equal{#1}{}}{}{(#1)}}}
\newcommand{\Htwoz}[1][]{\ensuremath{H^2_0\ifthenelse{\equal{#1}{}}{}{(#1)}}}
\newcommand{\Hmtwo}[1][]{\ensuremath{H^{-2}\ifthenelse{\equal{#1}{}}{}{(#1)}}}
\newcommand{\HDiv}[1][]{\ensuremath{H(\mathrm{div}\ifthenelse{\equal{#1}{}}{}{,#1})}}
\newcommand{\HDivz}[1][]{\ensuremath{H_0(\mathrm{div}\ifthenelse{\equal{#1}{}}{}{,#1})}}
\newcommand{\HCurl}[1][]{\ensuremath{H(\mathrm{curl}\ifthenelse{\equal{#1}{}}{}{,#1})}}
\newcommand{\HDivDiv}[1][]{\ensuremath{H(\mathrm{divdiv}\ifthenelse{\equal{#1}{}}{}{,#1})}}
\newcommand{\HCurlCurl}[1][]{\ensuremath{H( \mathrm{curl curl}\ifthenelse{\equal{#1}{}}{}{,#1})}}
\newcommand{\HCurlDiv}[1][]{\ensuremath{H( \mathrm{curl div}\ifthenelse{\equal{#1}{}}{}{,#1})}}
\newcommand{\Hinc}[1][]{\ensuremath{H(\inc\ifthenelse{\equal{#1}{}}{}{,#1})}}
\newcommand{\Cone}[1][]{\ensuremath{C^1\ifthenelse{\equal{#1}{}}{}{(#1)}}}
\newcommand{\Czero}[1][]{\ensuremath{C^0\ifthenelse{\equal{#1}{}}{}{(#1)}}}
\newcommand{\Ctwo}[1][]{\ensuremath{C^2\ifthenelse{\equal{#1}{}}{}{(#1)}}}
\newcommand{\Ck}[1][]{\ensuremath{C^k\ifthenelse{\equal{#1}{}}{}{(#1)}}}
\newcommand{\TF}[2][]{\ensuremath{C_0^{\infty}(#2\ifthenelse{\equal{#1}{}}{}{,#1})}}

\newcommand{\Cinf}[1][]{\ensuremath{C^{\infty}\ifthenelse{\equal{#1}{}}{}{(#1)}}}
\newcommand{\DF}[1][]{\ensuremath{C_0^{\infty,*}\ifthenelse{\equal{#1}{}}{}{(#1)}}}

\newcommand{\Gauss}{K}

\newcommand{\cn}{\varpi} 
\newcommand{\OneForm}{\cn}  

\newcommand{\nv}{\nu}
\newcommand{\nut}{{\tilde{\nu}}}   
\newcommand{\tv}{{{\tau}}}       
\newcommand{\gn}{{{\hat{\nu}}}}  
\newcommand{\gt}{{{\hat{\tau}}}} 
\newcommand{\tin}{\tau_-}
\newcommand{\tout}{\tau_+}
\newcommand{\tinout}{\tau_\pm}
\newcommand{\htin}{\hat{\tau}_-}
\newcommand{\htout}{\hat{\tau}_+}
\newcommand{\htinout}{\hat{\tau}_\pm}
\newcommand{\hni}{\hat{\nu}_-}
\newcommand{\hno}{\hat{\nu}_+}
\newcommand{\hnio}{\hat{\nu}_\pm}

\newcommand{\WW}{{\mathcal{W}}}
\newcommand{\Wo}{\mathring{\mathcal{W}}}
\newcommand{\Regge}[1][]{\ensuremath{\mathcal{R}\ifthenelse{\equal{#1}{}}{}{(#1)}}}
\newcommand{\RR}{\Regge}

\newcommand{\idop}{\ensuremath{I}}

\newcommand{\DD}{\mathcal{D}}

\newcommand{\pder}[2]{\ensuremath{\partial_{#2}{#1}}}

\DeclareMathOperator{\tro}{\mathrm{tr}}
\newcommand{\tr}[1]{\ensuremath{\,\mathrm{tr}(#1)}}

\newcommand{\cof}[1]{\ensuremath{\,\mathrm{cof}(#1)}}

\newcommand{\norm}[2][]{\ensuremath{\ifthenelse{\equal{#1}{}}{\left\|#2\right\|}{\left\|#2\right\|_{#1}}}}
\newcommand{\lnorm}[2][]{\ensuremath{\left\|#2\right\|_{L^2\ifthenelse{\equal{#1}{}}{}{(#1)}}}}
\newcommand{\hnorm}[2][]{\ensuremath{\left\|#2\right\|_{H^1\ifthenelse{\equal{#1}{}}{}{(#1)}}}}

\newcommand{\BDM}{BDM}
\newcommand{\RT}{RT}

\newcommand{\Pol}{\ensuremath{\mathcal{P}}}
\newcommand{\RegInt}[1][]{\mathcal{I}^{\Regge}_{\ifthenelse{\equal{#1}{}}{}{#1}}}
\newcommand{\HoneInt}[1][]{\mathcal{I}_{h}^{\VV,\ifthenelse{\equal{#1}{}}{}{#1}}}

\newcommand{\Chrtwo}[1][]{\ensuremath{\Gamma\ifthenelse{\equal{#1}{}}{_{\bullet\bullet}^{\,\,\,\,\,\bullet}}{_{#1 \bullet}^{\hphantom{#1 \bullet}\bullet}}}}

\newcommand{\Eucl}{\delta}

\newcommand{\VDer}[2]{\ensuremath{D_{#2}(#1)}}

\newcommand{\nrm}[1]{\left\vvvert{{#1}}\right\vvvert}

\title[Covariant incompatibility \& curvature]{Analysis of curvature approximations via covariant curl and incompatibility for Regge metrics}

\author[J.~Gopalakrishnan]{Jay~Gopalakrishnan}
\address{Portland State University, PO Box 751, Portland OR 97207,USA }
\email{gjay@pdx.edu}

\author[M.~Neunteufel]{Michael~Neunteufel}
\address{Institute for Analysis and Scientific Computing, TU Wien, Wiedner Hauptstr. 8-10, 1040 Wien, Austria}
\email{michael.neunteufel@tuwien.ac.at}
\author[J.~Sch\"oberl]{Joachim~Sch\"oberl}
\address{Institute for Analysis and Scientific Computing, TU Wien, Wiedner Hauptstr. 8-10, 1040 Wien, Austria}
\email{joachim.schoeberl@tuwien.ac.at}
\author[M.~Wardetzky]{Max~Wardetzky}
\address{Institute of Numerical and Applied Mathematics, University of G\"ottingen, Lotzestr. 16-18, 37083 G\"ottingen, Germany}
\email{email: wardetzky@math.uni-goettingen.de}

\begin{document}

\maketitle

\begin{abstract}
  The metric tensor of a Riemannian manifold can be approximated using
  Regge finite elements and such approximations can be used to compute
  approximations to the Gauss curvature and the Levi-Civita connection
  of the manifold. It is shown that certain Regge approximations yield
  curvature and connection approximations that converge at a higher
  rate than previously known.  The analysis is based on covariant
  (distributional) curl and incompatibility operators which can be
  applied to piecewise smooth matrix fields whose
  tangential-tangential component is continuous across element
  interfaces. Using the properties of the canonical interpolant of the
  Regge space, we obtain superconvergence of approximations of these
  covariant operators. Numerical experiments further illustrate the
  results from the error analysis.
	\\
	\vspace*{0.25cm}
	
\end{abstract}

\section{Introduction}
\label{sec:intro}


This paper is concerned with the finite element approximation of the
Gauss curvature $\Gauss$ of a two-dimensional Riemannian manifold.  As
shown by Gauss's Theorema Egregium, $K$ is an intrinsic quantity of
the manifold. It can be computed solely using the metric tensor of the
manifold. Therefore, when a finite element approximation of the metric
tensor is given, it is natural to ask if an approximation to $K$ can
be computed. The answer was given in the affirmative by the recent
work of~\cite{Gaw20}, assuming that the metric is
approximated using Regge finite elements, and  further improved by~\cite{BKG21}.
The convergence theorems of
this paper are heavily based on these works. We prove that the
resulting curvature and connection approximations converge at a higher
rate than previously known for the approximation given by the
canonical Regge interpolant. Our method of analysis is different and
new. In particular, we show how covariant curl and incompatibility can
be approximated using appropriate finite element spaces, given a
nonsmooth Regge metric. These operators arise
in a myriad of other applications, so our
intermediate results regarding them are of independent interest.
Notions of curvature while gluing together piecewise smooth metrics
has been a preoccupation in varied fields far away from computing, as
early as \cite{Israe66} to recent years~\cite{Stric20}, so we note at
the outset that we approach the topic with
numerical computation in mind.

The Regge finite element takes its name from Regge
calculus, originally developed for solving Einstein field equations in
general relativity. It  discretizes the metric tensor through
edge-length specifications, allowing the curvature to be approximated by
means of angle deficits~\cite{Regge61}.
Regge calculus was established in 
theoretical and numerical physics and routinely finds applications
in relativity and quantum mechanics. In \cite{williams92,RW00,BOW18} a comprehensive overview of the development of Regge calculus over the last fifty years can be found.
Just as Whitney forms  \cite{whitney57} can be interpreted as finite elements, 
it was observed that Regge's approach of prescribing {quantities on edges}
is equivalent to defining a piecewise constant metric tensor whose  tangential-tangential components are continuous across element interfaces \cite{Sorkin75}. The first rigorous proof of convergence of Regge's
angle deficits to the scalar curvature, for a sequence of appropriate triangulations in the sense of measures, was accomplished in \cite{Cheeger84}.
Later, it was also shown \cite{christiansen15}
that for a given metric in the lowest order Regge finite element space,  a sequence of mollified metrics converges to the angle deficit in the sense of measures.
Methods based on angle deficits for approximating the Gauss curvature on triangulations consisting of piecewise flat triangles are well-established in discrete differential geometry and computer graphics. 
{On specific triangulations satisfying certain conditions, convergence  in the $\Linf$-norm up to quadratic order was proven, but for a general irregular grid  there is no reason to expect convergence~\cite{BCM03,Xu06,XX09}.}
In \cite{LX07},  Regge's concept of angle deficits has been extended to quadrilateral meshes. Notable
among the results applicable for higher dimensional manifolds
is the proof of convergence for approximated Ricci curvatures of isometrically embedded hypersurfaces $\subset\R^{n+1}$ presented 
in \cite{Fritz13}, and used later for Ricci flows \cite{Fritz15}.

Another natural perspective to place the modern 
developments on the Regge finite element is within the 
emergence of \emph{finite element exterior calculus} (FEEC) \cite{arnold06, arnold10}. Finite element structures for Regge calculus were developed in \cite{christiansen04,christiansen11} and  the resulting elements became popular in FEEC under the name Regge finite elements~\cite{li18}.
Regge elements approximating metric and strain tensors were extended to arbitrary polynomial order on triangles, tetrahedra, and higher dimensional simplices in \cite{li18}, and for quadrilaterals, hexahedra, and prisms in \cite{Neun21}. 
The utility of Regge elements when
discretizing parts of the  Kr\"oner complex,
or the elasticity complex,  was considered in 
\cite{AH21, christiansen11,Hauret13}.
Properties of Regge elements were  exploited to construct a method avoiding membrane locking for general triangular shell elements~\cite{NS21}.

In this backdrop, the recent work of \cite{Gaw20} provides an  interesting
application of Regge elements by developing a
high-order Gauss curvature approximation formula based on higher degree 
Regge elements. (It was applied to Ricci and Ricci-DeTurck flow
\cite{Gawlik19}.) The key is an  integral formulation of the angle deficit, extendable to higher orders. Using it, rigorous proofs of convergence at specific rates were proved in~\cite{Gaw20}. Even more recently, in \cite{BKG21}, this approach has been reformulated in terms of a nonlinear distributional curvature and connection 1-form (Levi-Civita connection), using the element-wise Gauss curvature, jump of geodesic curvature at edges, and angle deficits at vertices as sources of curvature \cite{Sullivan08,Str2020}. The authors show that
$\Ltwo$-convergence of the approximated curvature can be obtained if
Regge elements using polynomials
of degree at least two are used to approximate the metric.
This is in line with the rule of thumb that a second order
differential operator approximated using polynomials of degree $k$
leads to convergence rates of order~$k-2$.
Nonetheless, convergence rates better than this rule of thumb have often
been observed in compatible discretizations in FEEC.
One of our goals in this paper is to establish such an improved rate for the curvature and connection 
approximations, as well as for the intermediate covariant operators
arising in our analysis, such as the curl and incompatibility.

In a later section, we extend the ideas in \cite{Gaw20,BKG21} by
exploiting certain orthogonality properties for the error in the
canonical interpolation by Regge finite elements to obtain one extra
order of convergence for the curvature approximation.
This extra order  is comparable to super-convergence properties of mixed methods \cite{BBF13,Com89} and has been observed for the Hellan--Herrmann--Johnson method for the biharmonic plate and shell equation~\cite{AW20,Walker21}.

Another difference in our analysis, in comparison  to
\cite{Gaw20,BKG21}, is the use of the intrinsic (or covariant)
incompatibility operator (which we define  using the covariant curl
 on the manifold).  It is now well known that linearizing the
curvature operator around the Euclidean metric gives a first order
term involving the incompatibility operator~\cite{christiansen11} and
we exploit this relationship in the analysis of the curvature
approximation.  On Euclidean manifolds, the incompatibility operator
is well known to be the natural differential operator for Regge
elements in any dimension. By showing that the curvature
approximation can be analyzed via the incompatibility operator, we
hope to generate new ideas for computing and analyzing approximations
of the intrinsic curvature tensor of higher dimensional manifolds.  The
incompatibility operator also arises in modeling elastic materials
with dislocations~\cite{AG16,AG20}, another potential area of application. 
The key insight on which we base our definition of these covariant
operators for Regge metrics is revealed by the essential role played
by a glued smooth structure (described in
\S\ref{sec:covariant-curl-regge}).  Since coordinates in this glued
smooth structure are generally inaccessible for computations, we
detail how to compute these operators in the coordinates in which the
Regge metric is given as input.

This paper can be read linearly, but we have structured it so a
numerical analyst can also directly start with the error analysis in  Section~\ref{sec:num_ana}---where the main convergence theorems appear in \S\ref{subsec:statement_theorem}---referring back to the previous sections as needed.
(Only coordinate-based formulas are used in \S\ref{sec:num_ana};
their derivations from intrinsic geometry are in the previous sections.)
The next section (\S\ref{sec:notation}) establishes 
notation and introduces geometric preliminaries and finite element spaces. Section~\ref{sec:curv} defines the curvature approximation formula and details coordinate formulas we use for numerical computations.
In \S\ref{sec:cov}, covariant differential operators on piecewise smooth
metric tensors are derived, concentrating on the covariant curl and incompatibility operator, and how they arise from linearization of curvature.
Section~\ref{sec:connection_one_form_approx} is devoted to the
approximation of the connection 1-form.
Section~\ref{sec:num_ana} is devoted to  the numerical analysis of the errors in the method. The analysis is performed by first  proving  optimal convergence rates for the distributional covariant curl and inc, and then for the approximations of the  Gauss curvature and connection 1-form. Numerical examples illustrating  the theoretical results are presented in \S\ref{sec:num_examples}.

\section{Notation and preliminaries}
\label{sec:notation}

This section provides definitions that we use throughout.  We give
intrinsic definitions of quantities on a manifold, but in view of our
computational goals, we also make extensive use of coordinate
expressions. We use the Einstein summation convention, by which a term
where the same integer index appears twice, as both an upper and a
lower subscript, is tacitly assumed to be summed over the values of
that index in $\{ 1, 2\}$. Summation convention does not apply when a
repeated index is not an integer (such as when a subscript or a
superscript represents a vector field or other non-integer
quantities).

\subsection{Spaces on the manifold.}   \label{ssec:spaces-manifold}

Let $M$ denote a two-dimensional oriented manifold with or
without boundary. Endowed with a smooth metric $\gex$, let $(M, \gex)$
be a Riemannian manifold.
Let the 
unique Levi-Civita connection generated by $\gex$ be denoted by~$\nabex$.
Let $\Xm M$,
$\W^k(M)$, and $\TT^k_l(M)$ denote, respectively, the sets of smooth
vector fields on $M$, $k$-form fields on $M$, and $(k, l)$-tensor
fields on $M$. The value of a tensor $\rho \in \TT^k_l(M)$ acting on
$k$ vectors $X_i \in \Xm M$ and $l$ covectors $\mu_j\in \W^1(M)$ is
denoted by $\rho(\mu_1, \ldots, \mu_l, X_1, \ldots, X_k)$.  Note that
$\W^1(M) = \TT^1_0(M)$ and $\Xm M = \TT^0_1(M)$. Note also that it is
standard to extend the Levi-Civita connection $\nabex$ from vector
fields to tensor fields (see e.g., \cite[Lemma~4.6]{Lee97}) so that
Leibniz rule holds.

For coordinate computations, we use a chart to move locally to a
Euclidean domain with coordinates $x^1, x^2$. Let the accompanying
coordinate frame and coframe be denoted by $\d_i$ and $dx^i$.
We assume these coordinates preserve orientation, so 
the orientation of $M$ is given by the ordering $(\d_1, \d_2)$.
Let $\Sc(M) = \{ \sigma \in \TT_0^2(M): \sigma(X, Y) = \sigma(Y, X)$ for $X, Y \in \Xm M \}$ and $\Sc^+(M) = \{ \sigma \in \Sc(M): \sigma(X, X) > 0$ for $0\neq X \in \Xm M \}$.
They represent the subspace of symmetric tensors in $\TT_0^2(M)$, whose
elements $\sigma$ can be expressed in coordinates as
$\sigma = \sigma_{ij} dx^i \otimes dx^j$ with smoothly varying
coefficients $\sigma_{ij}$ satisfying $\sigma_{ij} = \sigma_{ji}$ and are additionally positive definite, respectively.


We will use standard operations on 2-manifold spaces such as the Hodge star
$\star: \W^k(M) \to \W^{2-k}(M)$, the exterior derivative
$d^k: \W^k(M) \to \W^{k+1}(M)$, the tangent to cotangent isomorphism
$\flat: \Xm M \to \W^1(M),$ and the reverse operation
$\sharp: \W^1(M) \to \Xm M$. Their definitions can be found in
standard texts~\cite{Lee97, Peter16, Tu17}.

\subsection{Curvature.}   \label{ssec:curvature}

The exact metric $\gex$ is an element of $\Sc^+(M)$. 
We define the {\em Riemann curvature tensor} $\Rmex \in \TT^4_0(M)$ of the manifold following~\cite{Lee97},
\begin{equation}
  \label{eq:RiemannCurvTensor}
  \Rmex(X, Y, Z, W) =
    \gex (\nabex_X \nabex_Y Z - \nabex_Y \nabex_X Z - \nabex_{[X,Y]}Z, W), \qquad X, Y, Z, W \in \Xm M.
\end{equation}
If $X$ and $Y$ are linearly independent,
the {\em Gauss curvature} of $M$ can be expressed
by
\begin{equation}
  \label{eq:Kexact}
    K (\gex) = 
  \frac{\Rmex(X, Y, Y, X)}{ \gex(X, X) \gex(Y, Y) - \gex(X, Y)^2},  
\end{equation}
whose value is well known to be independent of the choice of the basis
(see, e.g., \cite[p.~144]{Lee97} or  \cite[Ch.~4, Proposition~3.1]{Carmo1992}).

We will also need the {\em geodesic curvature} along a curve $\varGamma$ in the
manifold $(M, \gex)$. To recall its standard definition
(see~\cite[p.~140]{Tu17} or~\cite{Lee97}), we let $0 < s < a$ be the
$\gex$-arclength parameter so that $\varGamma$ is described by $\mu(s)$ for
some smooth $\mu: [0, a] \to M$ and its $\gex$-unit tangent vector is
$\textsc{t}(s) = d\mu/ds$. Let $\textsc{n}(s)$ be such that
$(\textsc{t}(s), \textsc{n}(s))$ is a 
$\gex$-orthonormal set of two vectors in the tangent space whose
orientation is the same as that of $M$, i.e.,
$dx^1 \wedge dx^2(\textsc{t}(s), \textsc{n}(s)) >0$. Then
\begin{equation}
  \label{eq:geodesic-curv-defn}
  \kappa(\gex) = \gex ( \nabex_{\textsc{t}(s)} \textsc{t}(s), \textsc{n}(s))
\end{equation}
gives the geodesic curvature at the point $\mu(s)$ of $\varGamma$.

\subsection{Approximate metric.}   \label{ssec:approximate-metric}

We are interested in approximating $\Kex(\gex)$ when the metric is
given only approximately.  We assume that $M$ has been subdivided into
a geometrically conforming triangulation $\T$. The edges of $\T$ may be curved, but do not necessarily consist of geodesics.

On each element $T \in \T$, we are given an approximation
$g|_T \in \Sc(T)$ of $\gex|_T$.  When the approximation is
sufficiently good, $g$ will also be positive definite since $\gex$
is. Then each $T \in \T$ can be considered to be a
Riemannian manifold $(T, g|_T)$ with $g|_T$ as its metric.  
Since $g|_T$ is smooth within each element $T$
(not across $\d T$), we use the unique Levi-Civita connection $\nabla$
generated by $g|_T$ to compute covariant derivatives within $T$.
(Constraints on $g$ across element boundaries are
clarified below in \eqref{eq:ttspace}).
We drop the
accent $\dhat{\;}$  in any
previous definition to indicate that it pertains to the manifold
$(T, g|_T)$ instead of $(M, \gex)$, e.g., $R$ refers to the Riemann curvature tensor computed using $g$ and $\nabla$ in place of $\gex$ and $\nabex$ in \eqref{eq:RiemannCurvTensor}.

A point $p \in T$ can be viewed either as a point in the manifold $M$
or as a point in the manifold $T$. Irrespective of the two viewpoints,
the meanings of coordinate frame $\d_i$, coframe $dx^i$, and the
tangent space $T_pM$ at $p$ are unchanged. In coordinates,
\begin{equation}
  \label{eq:gij}
  g_{ij} = g(\d_i, \d_j), \qquad g^{ij} = g^{-1}(dx^i, dx^j)
\end{equation}
may be viewed as entries of symmetric positive definite matrices.
\emph{Christoffel symbols of the first kind} ($\Gamma_{ijk}$) and the
\emph{second kind} ($\Gamma_{ij}^k$) are defined by
\begin{subequations}
  \label{eq:Christoffels}
\begin{equation}
  \label{eq:Chris12}
  \Gamma_{ijk} = g( \nabla_{\d_i} \d_j, \d_k), \qquad
  \nabla_{\d_i} \d_j = \Gamma_{ij}^k \d_k.
\end{equation}
They can alternately be expressed, using~\eqref{eq:gij}, as
\begin{equation}
  \label{eq:Chris12-coords}
  \Gamma_{ijl} = \frac 1 2 (\d_i g_{jl} + \d_j g_{li} - \d_l g_{ij}), \quad
  \Gamma_{ij}^k = g^{kl} \Gamma_{ijl}.
\end{equation}
\end{subequations}
Later, we will also use 
$\Gamma_{ijl}(\sigma)$ to denote $ \frac 1 2 (\d_i \sigma_{jl} + \d_j \sigma_{li} - \d_l \sigma_{ij})$ for other tensors $\sigma$ in $\TT^2_0(M)$.

\subsection{Tangents and normals on element boundaries.}
\label{ssec:tang-norm-elem}

%

Throughout this paper, we use $\tv$ to denote a tangent vector (not
$g$-normalized; cf.~\eqref{eq:tau-normalization} later) along an element boundary $\d T$ for any
$T \in \T$. The orientation of $\tv$ is aligned with the boundary
orientation of $\d T$ (inherited from the orientation of $T$, which is
the same as the orientation of $M$). 
For any $p \in \d T,$ define
$\nut \in T_pM$ by
\begin{equation}
  \label{eq:gnu}
  g(\nut, X) = (dx^1 \wedge dx^2)
  {(\tv,X)}, \qquad\text{ for all } X \in T_pM.
\end{equation}
It is easy to see from~\eqref{eq:gnu} that the
ordered basis
{$(\tv, \nut)$} has the same orientation as $(\d_1, \d_2)$
since $(dx^1 \wedge dx^2){(\tv, \nut)}>0$, and moreover, 
\begin{equation}
  \label{eq:nt-properties}
  g(\nut, \tv) = 0,   \quad\text{ and } \qquad  g(\nut, \nut) \det(g) = g(\tv, \tv).
\end{equation}
In particular, defining 
\begin{equation}
  \label{eq:normalized-nu-tau}
  \gn = \frac{\nut}{\sqrt{g_{\nut \nut}}}, \qquad
  \gt = \frac{\tv}{\sqrt{g_{\tv \tv}}},
\end{equation}
we obtain a $g$-orthonormal basis 
{$(\gt, \gn)$} of normal and tangent
vectors along every element boundary $\d T$, whose orientation matches
the manifold's orientation. 
E.g., if $M$ is the unit disc in $\R^2$
with the Euclidean metric $g=\Eucl$ and the standard orientation, then
$\gt$ is oriented counterclockwise and $\gn$ points inward.

In \eqref{eq:normalized-nu-tau} and
throughout, we use $\sigma_{uv}$ to denote $\sigma(u, v)$ for any vectors $u, v \in T_pM$ and $\sigma \in \Sc(M)$. Note that $g_{uv}$ is not to be confused with the $g_{ij}$ introduced 
in~\eqref{eq:gij} where the indices are integers (which trigger the
summation convention) rather than vectors.

To write $\nut$ in coordinates, it is useful to introduce the
alternating symbol $\veps^{ij}$ whose value is $1$, $-1,$ or $0$
according to whether $(i,j)$ is an even permutation, odd permutation,
or not a permutation of $(1,2)$, respectively. The value of the
symbols $\veps_{ij}$, $\veps_i^{\phantom{,i}j}$, and
$\veps^i_{\phantom{i,}j}$ equal $\veps^{ij}$. It is easy to see
that~\eqref{eq:gnu} implies
\begin{equation}
  \label{eq:nu-tau-coords}
  {\nut^k = -g^{kj} \veps_{ji} \tv^i}
\end{equation}
and furthermore (e.g., using~\eqref{eq:nu-tau-coords},
\eqref{eq:nt-properties}, and the Hodge star in coordinates, 
{$\star dx^i = \sqrt{\det g}g^{ij}\varepsilon_{jk}dx^k$}) that
\begin{equation}
  \label{eq:star-rotation}
  {(\star \og)(\gn) =  \og(\gt),\qquad (\star \og)(\gt)= -\og(\gn)},
\end{equation}
for any $\og \in \W^1(M)$.

\subsection{Finite element spaces.}  \label{ssec:finite-elem-spac}

Let $C^\infty(\T)$ denote the space of {\em piecewise smooth}
  functions on $\T$, by which we mean functions that are infinitely
smooth within each mesh element and continuous up to (including) the
boundary of each mesh element $T$.
A 
notation in \S\ref{ssec:spaces-manifold}
with $\T$ in place of $M$ indicates the piecewise smooth
analogue, e.g.,
\[
  \Xm \T = \{ X= X^i\d_i: X^i \in C^\infty(\T)\},
  \qquad  
  \Sc(\T) =\{ \sigma_{ij} dx^i \otimes dx^j: \sigma_{ij}=\sigma_{ji} \in C^\infty(\T)\},
\]
etc.
Note that a $\sigma \in \Sc(\T)$
need not  be continuous across the element
interfaces. Let $E = \d T_+ \cap \d T_-$ denote an interior mesh edge
(possibly curved) shared between elements $T_\pm \in \T$.  Let $T_pE$
denote the one-dimensional tangent space of the curve $E$ at any one
of its points $p$. (Note that the tangent space at $p$ from either
element $T_\pm$ coincides with $T_pM$ and $T_p E \subset T_pM$.)  We
say that a $\sigma \in \Sc(\T)$ has ``tangential-tangential
continuity'' or that $\sigma$ is {\em $tt$-continuous} if
\begin{equation}
  \label{eq:tt-cts}
\sigma|_{T_+}(X, X) =
\sigma|_{T_-}(X, X),
\qquad \text{ for all } X \in T_pE,  
\end{equation}
at all $p \in E$, and for every interior mesh edge $E$.  Let
\begin{subequations}
  \label{eq:ttspace}
  \begin{align}
    \Regge(\T)
    & =
    \{
    \sigma \in \Sc(\T): \sigma \text{ is $tt$-continuous}
    \},
    \\
    \RR^+(\T)
    & = \{ \sigma \in \Regge(\T): \sigma(X, X)>0 \text{ for all } X \in T_pM\}.
  \end{align}
\end{subequations}
The approximate metric $g$ is assumed to be in $\RR^+(\T)$. For the
numerical analysis later, we will additionally assume that it is in
$\RR_h^k$ defined below.

In finite element computations, we use a reference element $\hat T,$
the {unit triangle}, and the space $\Pol^k(\hat T)$ of polynomials of
degree at most~$k$ on $\hat T$. Let $\Teuc$ denote a Euclidean
triangle with possibly curved edges that is diffeomorphic to
$\hat T$ via $\hat{\Phi}: \hat T \to \Teuc$. For finite element computations on manifolds,  we need charts so that each whole
element $T \in \T$ of the manifold is covered by a single chart giving the coordinates $x^i$ on $T$.  The chart identifies the parameter domain of $T$ as the 
(possibly curved) Euclidean triangle   $\Teuc$    diffeomorphic to  $T$.
Let $\Phi: T \to \Teuc $ denote
the diffeomorphism. Then $\Phi_T = \hat \Phi^{-1} \circ \Phi : T \to \hat T$ maps diffeomorphically to the reference element where $\Pol^k(\hat T)$ is defined. 
We use its pullback $\Phi_T^*$ below.

Define the
{\em Regge finite element space} of degree $k$ on the manifold $M$ by
\begin{equation}
  \label{eq:ReggeFE}  
  \begin{aligned}
    \Regge_h^k =  \{ \sigma \in \Regge(\T):
    &
    \text{ for all }  T \in \T,\;
    \sigma|_T = \sigma_{ij} dx^i \otimes dx^j
    \\
    & \text{ with } \sigma_{ij} = \Phi_T^*\hat q \text{ for some }
    \hat q \in \Pol^k(\hat T ) \}.
  \end{aligned}
\end{equation}
The subscript $h$ indicates a mesh size parameter, e.g., on meshes
whose elements are close to straight-edged triangles, one may set
$h = \max_{T \in \T} \text{diam}(T).$ Let 
\begin{equation*}
  \begin{gathered}
    \VV(\T) = \{ u \in \W^0 (\T) :
    \; \text{$u$ is continuous on } M\},
    \\
    \Vo_\vG(\T)
     = \{ u \in \VV(\T): u|_{\vG} = 0\},
  \end{gathered}
\end{equation*}
where $\vG$ denote a subset of the boundary $\d M$ of positive
boundary measure. 

The {\em
  Lagrange finite element space} on $M$ and its subspaces with essential boundary conditions are defined by
\begin{equation}
\label{eq:LagrangeFEspace}
\begin{gathered}
  \VV_h^k = \{ u \in \VV(\T):  \text{ for all } T \in \T,
  \; u|_T=\Phi_T^* \hat u \text{ for some }
  \hat u \in \Pol^k(\hat T) \}, 
  \\
  \Vo_{h, \vG}^k = \{ u \in \VV_h^k: u|_{\vG} = 0 \} \quad\text{ and } \quad
  \Vo_h^k = \Vo_{h, \d M}^k.
\end{gathered}
\end{equation}

The previous 
definitions in this subsection were independent of the
metric. We will now introduce a metric-dependent space of 
normal-continuous vector fields. First,
we introduce the following notation surrounding an
interior mesh edge $E$ shared by two adjacent elements in $\T$,
\begin{subequations}
	\label{eq:jump-def}
	\begin{equation}
	E = \d T_- \cap \d T_+, \qquad T_\pm \in \T.   
	\end{equation} 
	In this context, the $g$-orthonormal tangent and normal vectors
	introduced above along $\d T_{\pm}$ are denoted by $\gt_\pm$ and
	$\gn_\pm$, respectively.  For a collection of scalar functions, 
        $\{f_{\d T}(\gn): T \in \T \}$,   each depending on the normal $\gn$ at an  element boundary,
        we define the jump  on $E$ by
	\begin{equation}
	\label{eq:jump-def-b}
	\jmp{f( \gn)} = f_{\d T_+}(\gn_+) + f_{\d T_-}(\gn_-).   
	\end{equation}
\end{subequations}
Thus the jump function $\jmp{f(\gn)}$ is well defined (and
single-valued) on the union of all interior mesh edges, excluding the
mesh vertices. The jump of an element boundary function dependent on $\gt$
is defined similarly.

We say that  a piecewise smooth vector field $W \in \Xm  \T$
 has {\em ``$g$-normal continuity''} across element
interfaces if  $\jump{ g(W, \gn)}=0$.  Define
\begin{equation}
\label{eq:Ngspaces}
\begin{gathered}
\WW_g(\T) = \{ W \in \Xm \T : \jmp{g(W, \gn)} = 0\},
\\
\Wo_{g,\vG}(\T)
= \{ W \in \WW_g(\T): g(W|_{\vG}, \gn) = 0\}, \quad
\Wo_{g}(\T) = \Wo_{g,\d M}(\T).
\end{gathered}
\end{equation}
Also define their polynomial subspaces
\begin{equation*}
\begin{gathered}
  \begin{aligned}
    \WW^k_{g,h} = \{ W \in \WW_{g}(\T) : \text{ for all } T \in \T,
    & \; W|_T=\Phi_T^* \hat W \\
    & \text{ for some }
    \hat W \in \Pol^k(\hat T,\R^2)\},    
  \end{aligned}
\\
\Wo^k_{g,h,\vG}
= \{ W \in \WW^k_{g,h}: g(W|_{\vG}, \gn) = 0\}, \quad
\Wo^k_{g,h} = \Wo^k_{g,h, \d\om}.
\end{gathered}
\end{equation*}

\subsection{Integrals over the manifold's triangulation.} \label{ssec:integr-over-manif}

On every element $T \in \T$, in order to integrate a scalar function
$f \in \W^0(T)$,
adopting the notation of \cite{Lee12b},
we tacitly use the unique Riemannian
volume form $\vol T$ to convert it to a 2-form and then pullback to integrate
over the Euclidean parameter domain $\Teuc$, i.e.,
\begin{equation}
  \label{eq:intTf}
  \int_{(T, g)} f \equiv
  \int_{T} f \,\vol T
  =
  \int_{\Teuc} (\Phi^{-1})^*(f \,\vol T)
  = 
  \int_{\Teuc}f \circ \Phi^{-1} \frac{\sqrt{\det g}}{\det (D\Phi)} \; \da,
\end{equation}
where $\det(D\Phi)$ denotes the Jacobian determinant of $\Phi$, we have used the
standard extension of pullback to forms, and we have appended an area
measure notation ``$\da$'' to emphasize that the right most integral
is a standard Lebesgue integral over the Euclidean domain $\Teuc $.
For $v, w \in \W^0(\T)$, set
  \begin{equation*}
  \int_\T w = \sum_{T\in \T} \int_{(T, g)} w, 
  \qquad
  (v, w)_\T = \sum_{T\in \T} \int_{(T, g)}  v \, w,  
\end{equation*}
with the understanding that the right hand sides above must be
evaluated using~\eqref{eq:intTf}.  In order to integrate along the boundary
curve $\d T$, we use the one-dimensional analogue of the formula
in~\eqref{eq:intTf} to compute on the Euclidean domain
$\d \Teuc = \Phi( \d T)$, namely
\begin{equation}
  \label{eq:bdrintg}
  \int_{(\d T, g)}  f =
  \int_{\d \Teuc} (\Phi^{-1})^*( f \, \vol{\d T})
  = \int_{\d \Teuc} f\circ \Phi^{-1}\; \sqrt{ g_{tt}} \;\dl, 
\end{equation}
where $t$ is a tangent vector along $\d \Teuc$ of unit Euclidean
length---and to emphasize that the last integral is a standard
Euclidean integral, we have appended the length measure ``$\dl$''. We use
\begin{equation*}
  \int_{\d \T} w\; = \sum_{T\in \T} \int_{(\d T, g)} w
\end{equation*}
to simplify notation for sum of integrals over element boundaries.

\section{Curvature approximation}
\label{sec:curv}

In this section we give the curvature approximation formula and
discuss a few nontrivial computational details on curved elements. In order to
approximate the Gauss curvature $K(\gex)$, one may consider computing
$K(g|_T)$ on each element $T \in \T$ using the given approximation $g$
of the exact metric $\gex$. However, this alone cannot generally be a
good approximation to $K(\gex)$ because discontinuities of $g$ across
elements generate additional sources of curvature on the edges and vertices
of the mesh.  Below we provide a curvature approximation incorporating
these extra sources. Since it coincides with the formula given in a
recent work~\cite{BKG21} for a specific case, we opt for a brief description,
expanding only on aspects complementary to that work.

\subsection{A finite element curvature approximation}
\label{subsec:curv_approx_formula}

Given a metric $g \in \RR^+(\T)$ approximating $\gex$
%
we identify
three sources of curvature, modeled after similar terms in the
Gauss-Bonnet formula, and define them as the following linear
functionals acting on $\vphi \in \VV(\T)$:
\begin{equation}
  \label{eq:variousK}
  \begin{gathered}
    \act{K_g^T, \,\vphi}_{\VV(\T)}
    =   \int_{(T, g)} K(g)\, \vphi,
    \qquad 
    \act{K^T_{E, g}, \vphi}_{\VV(\T)}
    = \int_{(E, g)} \kappa(g) \, \vphi,
    \\
    \act{ K_{V, g},  \vphi}_{\VV(\T)}
    = \left(2\pi -\sum_{T \in \T_V} \agl_V^Tg\right) \, \vphi(V),
  \end{gathered}
\end{equation}
where $\T_V$ denotes the set of all elements of $\T$ which have $V$ as a vertex.
Here $K(g)$ and $\kappa(g) $ are defined by \eqref{eq:Kexact}
and~\eqref{eq:geodesic-curv-defn} after replacing $\gex$ by $g$, 
and $\agl_V^T(\cdot)$ denotes the
interior angle at a vertex $V$ of $T$ determined using the metric in
its argument (computable using \eqref{eq:2} below).
Throughout, we use
$\act{f, \vphi}_H$ to denote duality pairing on a vector space $H$
that gives the action of a linear functional $f \in H'$ acting on a $\vphi \in H$.
Also, $\V$ and $\E$ denote the set of mesh vertices and edges,
respectively (so in~\eqref{eq:variousK},  $V \in \V$ and $E \in \E$).
Define 
  $K_g \in \VV(\T)'$ by
  \begin{equation}
    \label{eq:Kg}
    K_g =
    \sum_{T \in \T} \left(
      K_g^T + \sum_{E \in \E_T} K_{E, g}^T \right)  + \sum_{V \in \V} K_{V, g}.      
  \end{equation}
Here $ \E_T$ denotes the set of three edges of
$\d T$.

  In addition to $g$, 
  suppose that we are also given boundary curvature data
  in essential (Dirichlet) or natural (Neumann) forms for manifolds with boundary.
  The former
  type of data arises when we know that $M$ is a submanifold of a larger
  manifold whose Gauss curvature is known outside of $M$. To accommodate
  such information only on a part of the boundary of $M$,
  we split $\d M$ into two non-overlapping parts $\Gamma_D$ and $\Gamma_N$.
  One of these can be empty. In case none of them is empty, both must have positive length.  On $\Gamma_D$, we assume that we are given $\Gauss^D = \Gauss(\gex)|_{\Gamma_D}$ and that $\Gauss^D$ is in the trace of the Lagrange finite element space $\VV_{h}^k$. E.g., when a manifold is flat around $\Gamma_D$ (i.e., $\Gauss(\gex)$ vanishes in a neighborhood of $\Gamma_D$), we may set homogeneous Dirichlet data $K^D = 0$ on $\Gamma_D$.
  The other type of boundary data, in the form of a natural (or
  Neumann) boundary condition, is motivated by the Gauss--Bonnet theorem, and provides geodesic curvature data at the boundary.
  These natural Neumann-type boundary data is given in the form of
  a data functional  $\kappa^N \in \VV(\T)'$, 
\begin{align}
\label{eq:Neumann_data_functional}
  \act{\kappa^N, \vphi}_{\VV(\T)}  = \int_{(\Gamma_N, \gex)} \kappa(\gex)\, \vphi
  \;- \sum_{V \in \V^N} \tilde{\agl}_V^N(\gex) 
  \,\vphi(V),
\end{align}
for $\vphi \in \VV(\T)$, where $\V^N \subset \V\cap\Gamma_N$ is the subset of the manifold's vertices contained in the interior of $\Gamma_N$ and
$\tilde{\agl}_V^N$ denotes the exterior angle measured by the edges of
$\Gamma_N$ at $V$. (If $V$ is part of a smooth boundary, such an angle amounts to $\pi$,
whereas at kinks of the boundary the angle has to be provided as input data.)
The action of functional \eqref{eq:Neumann_data_functional} on the
finite-dimensional subspace $\VV_h^k \subset \VV(\T)$ is computable if
we are given the exact metric $\gex$ on and near $\Gamma_N$. 
For manifolds without boundary, there is no need to provide any boundary data.

\begin{definition}
  \label{def:lift_distr_curv_mf}
  Let $g \in \RR^+(\T)$ and $k \ge 0$ be an integer. 
  The finite element curvature approximation $K_h(g)$ of degree $k+1$
  is the unique function in
  $\VV_{h}^{k+1}$ determined by requiring that
  $K_h(g)|_{\Gamma_D} = K^D$ on $\Gamma_D$ and for all
  $u_h\in \Vo_{h,\Gamma_D}^{k+1}$,
  \begin{equation}
    \label{eq:1}
    \begin{aligned}    
      \int_\T K_h(g) \,u_h = \act{K_g, u_h}_{\VV(\T)}
      -\act{\kappa^N,u_h}_{\VV(\T)}.
    \end{aligned}
  \end{equation}
\end{definition}

\subsection{Implementation issues}

We now discuss how to numerically compute the
quantities in~\eqref{eq:1} in the given computational coordinates~$x^i$. Recall (from \S\ref{ssec:tang-norm-elem}) that
the tangent vector $\tv$ along the boundary $\d T$ is aligned with the
boundary orientation of $\d T$.
  Let  $\V_T$ denotes the set of three vertices of an
  element~$T \in \T$. 
  At any vertex $V \in \V_T$, the
tangent $\tv$ undergoes a change in direction, between   an incoming
and an outgoing tangent vector, which we denote by $\tin$ and $\tout$,
respectively. The angle at $V$ with respect to the metric $g$ is then
computed by
\begin{equation}
  \label{eq:2}
  \agl_{V}^Tg = \arccos \left(
    \frac{g( -\tin, \tout)}
    {\sqrt{ g(\tin, \tin)}\sqrt{g(\tout, \tout)}}
    \right).
\end{equation}
This is what we use to calculate the angle deficit
functional $K^T_{V, g}$ \eqref{eq:variousK}.

Next, consider the interior source term $K^T_g$, defined using $K(g)$,
and related to the Riemann curvature by~\eqref{eq:Kexact}.  By~\eqref{eq:RiemannCurvTensor},
$R_{ijkl} = R(\d_i, \d_j, \d_k, \d_l)$ simplifies to
\begin{align}
  \nonumber 
  R_{ijkl} & =
             (\d_i\Gamma_{jk}^p+\Gamma_{jk}^q\Gamma_{iq}^p-
             \d_j\Gamma_{ik}^p-\Gamma_{ik}^q\Gamma_{jq}^p)g_{pl}
  \\
  \label{eq:Rijkl}    
           & =
             \d_i \Gamma_{jkl} - \d_j \Gamma_{ikl} - \Gamma_{ilp}\Gamma_{jk}^p
             + \Gamma_{jlp} \Gamma_{ik}^p,  
\end{align}
where  $\Gamma_{ij}^k$ and $\Gamma_{ijk}$ are as in \eqref{eq:Christoffels}.
For two-dimensional manifolds the Gauss curvature can be expressed~\cite{Carmo1992} by $K(g) = R_{1221} / \det g.$ 
Hence by~\eqref{eq:intTf},
\begin{equation}
  \label{eq:KT}
  \act{K_g^T, \vphi}_{\VV(\T)}
  = \int_T K(g) \,\vphi\, \sqrt{\det g}\; dx^1 \wedge dx^2
  = \int_{\Teuc}
  \frac{\Phi_*(R_{1221} \,\vphi)}{\sqrt{\det g}\; \det (D\Phi)}\; \da.
\end{equation}

It thus remains to discuss the computation of the edge sources
$K_{E, g}^T$ using the definition of the  $\kappa(g)$
in \eqref{eq:geodesic-curv-defn}.  In finite element computations,
we usually do not have ready access to the $g$-arclength parameter $s$ used
there. 
But $\kappa(g)$ can be computed  using the
readily accessible $\tv$ and $\nut$ of \S\ref{ssec:tang-norm-elem},
as shown below. %
%
Let $\gamma(t)$ be an orientation-preserving parametrization that
gives an oriented mesh edge $E \subset \d T$ as
$E = \{ \gamma(t): t_0 \le t \le t_1 \}$.  Parametrizing scalar functions $a$
on $E$  by $t$, we
abbreviate $d a /d t$ to $\dot{a}$.  Note that the components of
$\tv = \tv^k \d_k$ are given by $\tv^k = \dot\gamma^k$ and their derivatives by 
$ d^2 \gamma^k / d t^2 = \dot \tv^k$.
Let
\begin{equation}
  \label{eq:22}
  \dot \tv = \dot{\tv}^k \d_k,  
  \quad
  G^w_{uv} = g(u^iv^i\Gamma_{ij}^k \d_k, w),  
\end{equation}
where $u = u^i\d_i, v=v^i \d_i, w \in \Xm T$ and recall that $\nut$ was defined in \eqref{eq:gnu}.

\begin{lemma}
  \label{lem:geosdesiccurv}
  The geodesic curvature along each edge of an element boundary
  $\d T$ is given by
  \begin{equation}
  \label{eq:9}
  \kappa(g) =
  \frac{\sqrt{\det g}}{ g_{\tv\tv}^{3/2}}
  \left(
    g_{\dot \tv \nut}
    +
    G_{\tv\tv}^\nut
  \right).  
\end{equation} 
\end{lemma}
\begin{proof}
  By~\eqref{eq:geodesic-curv-defn} and our definitions of $\gt, \gn$ in
  \S\ref{ssec:tang-norm-elem},
  \begin{equation}
    \label{eq:3}
    \kappa(g) = g(\nabla_{\gt(s)} \gt(s), \gn(s))
  \end{equation}
  where $s$ is  the $g$-arclength parameter.  Inverting 
  $s(t) = \int_0^t g(\tv(\alpha), \tv(\alpha))^{1/2}\, d\alpha$ to write $t$ as a
  function of $s$, applying the chain rule to $\mu(s) = \gamma(t(s))$,
  and using $dt/ds = g_{\tv\tv}^{-1/2}$,
  \[
    \frac{d  \gt^k}{ds}=  \frac{d^2\mu^k}{d s^2}
    = \frac{d^2 \gamma^k}{d t^2} \left(\frac{dt}{ds}\right)^2 +
    \frac{d \gamma^k}{d t} \frac{d^2t}{ds^2}
    = \frac{\dot \tv^k }{g_{\tv\tv}} + \tv^k \frac{d^2t}{ds^2}.
  \]
  Using the properties of the connection $\nabla$ (see e.g. \cite{Lee97}), 
  \begin{equation}
    \label{eq:t-d-t}
    \nabla_{ \gt} \gt =
    \frac{d  \gt^k}{ds}\, \d_k +
    \gt^i  \gt^j \Gamma_{ij}^k \d_k
    =
    \frac{\dot \tv }{g_{\tv\tv}} + \tv \frac{d^2t}{ds^2}
    +
    \gt^i  \gt^j \Gamma_{ij}^k \d_k.
  \end{equation}
  Hence \eqref{eq:3}, \eqref{eq:t-d-t}, and the $g$-orthogonality of
  $\tv$ with $\gn$, implies that at the point $\gamma(t)$,
\begin{equation}
  \label{eq:4}
  \kappa(g) =
  \frac{g(\dot \tv, \gn)}{g_{\tv\tv}} +
  \frac{g ( \tv ^i  \tv ^j \Gamma_{ij}^k \d_k ,\gn)}{g_{\tv\tv}}.
\end{equation}
Now, by \eqref{eq:nt-properties} and \eqref{eq:normalized-nu-tau},
$\gn  = \nut (\det g/ g_{\tv\tv})^{1/2}$, so~\eqref{eq:4} implies
\begin{equation}
  \label{eq:5}
  \kappa(g) =
  \frac{(\det g)^{1/2}}{ g_{\tv\tv}^{3/2}}
  \left[
    g(\dot \tv, \nut)
    +
    g (\tv^i \tv^j \Gamma_{ij}^k \d_k , \nut)
  \right]
\end{equation}
which proves~\eqref{eq:9}.
\end{proof}

Returning to the edge source term, using
Lemma~\ref{lem:geosdesiccurv} and \eqref{eq:bdrintg},
\begin{align}
  \label{eq:20}
  K_{E, g}^T(\vphi, g)
  & =  \int_E \kappa(g) \,\vphi \sqrt{g_{\tv\tv}} 
   = \int_{\Phi(E)}
    \Phi_*
    \left(\frac{\sqrt{\det g}}{ g_{\tv\tv}}
    \left(
    g_{\dot \tv \nut}
    +
    G_{\tv\tv}^\nut
    \right) \vphi \right)\, \dl.
\end{align}
Thus, through~\eqref{eq:2}, \eqref{eq:KT}, and~\eqref{eq:20}, we have
shown it is possible to easily compute all the terms in the curvature
approximation~\eqref{eq:1} using standard finite element tools.

\subsection{A model case for further analysis}
\label{ssec:model}

Having shown how curvature of general manifolds can be computed, we
now focus our analysis on the following model case for the
remainder of the paper.
We assume that the manifold $M$, as a set, equals
$\om \subset \R^2$, a bounded open connected domain,
that $\Gamma_N = \emptyset,$ and that $\gex$ is
compactly supported in~$\om$ (so Dirichlet boundary data is zero).
The set $\om$ forms
the full parameter domain of $M$ for the single trivial chart
$\Phi: M \to \om$ with $\Phi$ equaling the identity map. The
triangulation $\T$ of $M$ is now a conforming finite element mesh in
the planar domain $\om$ and its elements are (possibly curved)
triangles.

In this setting, an element $T \in \T$ can be considered as either the
Euclidean manifold $(T, \delta)$, equipped with the identity metric
$\Eucl$, or the Riemannian manifold $(T, g)$.  Let us reconsider the
tangent vector $\tv$ on an element boundary $\d T$ with the
orientation described in \S\ref{ssec:tang-norm-elem}, previously left
un-normalized. Henceforth, we  assume that
\begin{equation}
  \label{eq:tau-normalization}
  1 = \delta(\tv, \tv). 
\end{equation}
We computed the normal vector $\nut$ from $\tv$
by \eqref{eq:gnu} for the manifold $(T, g)$. For the manifold
$(T, \Eucl)$, a generally different Euclidean normal vector arises at
any $p$ in $\d T$ and we denote it by $\nv$.
It can be computed by simply
replacing $g$ with $\Eucl$ in \eqref{eq:gnu}, i.e., $\nv \in T_pM$ satisfies
\begin{equation}
  \label{eq:gnu-eucl}
  \Eucl(\nv, X) = 
  {dx^1 \wedge dx^2 (\tv,X)} \qquad\text{ for all } X \in T_pM.  
\end{equation}
Analogous to~\eqref{eq:nu-tau-coords}, we now have the accompanying
coordinate expression,
\begin{equation}
  \label{eq:Eucl-normal-comp}
  {\nv^k = -\delta^{kj} \veps_{ji} \tv^i = -\veps^k_{\phantom{k,}i}\, \tv^i.  }
\end{equation}
Note that~\eqref{eq:tau-normalization} implies that
$\delta(\nv, \nv)=1$.
These identities, together with \eqref{eq:normalized-nu-tau}
guide us move between the
$\delta$-orthonormal tangent and normal vectors ($\tv$ and $\nv$) and
{$g$-orthonormal tangent and normal vectors ($\gt$ and $\gn$)}, while
preserving the orientation.
Jumps of functions of $\nv$ and $\tv$
on the Euclidean element boundaries $(\d T, \delta)$, $T \in \T$,
are defined in analogy to \eqref{eq:jump-def-b}.

Recall the parameterization $\gamma(t)$ along an edge $E$ and
accompanying notation considered before~\eqref{eq:22}.
We now claim that
\begin{equation}
  \label{eq:6}
  \dot{\nv}=-\delta(\nu, \dot \tau)\tv, \qquad
  \dot{\tv}=\delta(\nu, \dot{\tv})\nv.
\end{equation}
This follows by differentiating the equation $\delta(\nv, \nv) =1$
with respect to $\tv$, yielding $\delta(\dot \nv, \nv)=0$. Hence there
must be a scalar $\alpha$ such that $\dot \nv = \alpha \tv$. Now
differentiating $\delta(\nv, \tv) =0$, we find that the $\alpha$ must
satisfy $\delta(\alpha \tv, \tv) + \delta(\nv, \dot \tv)=0$, so
$\alpha = -\delta(\nv,\dot \tv),$ thus proving the first identity
in~\eqref{eq:6}.  A similar argument using $\delta(\tv, \tv) = 1 $
proves the second identity in~\eqref{eq:6}.  Later, we will have
occasion to consider variations of
$\sigma_{\tv\nv} = \sigma(\tv, \nv)$ along an edge for some tensor
$\sigma\in \TT_0^2(M)$. By chain rule,
\begin{align*}
  \frac{d}{d t} (\sigma_{\nv\tv}(\gamma(t))
  & = \dot \nv^i \tv^j \sigma_{ij}(\gamma(t))
    + \nv^i \dot  \tv^j \sigma_{ij}(\gamma(t))
    + \nv^i \tv^j\frac{d}{dt}(\sigma_{ij}(\gamma(t)).
\end{align*}
Hence~\eqref{eq:6}, together with $\dot\gamma = \tau$,  yields
\begin{equation}
  \label{eq:10}
  \d_\tv ( \sigma_{\nv\tv}) = (\d_\tv \mt\sigma)_{\nv\tv}
  + 
  ( \sigma_{\nv\nv} - \sigma_{\tv\tv}) \,\delta(\nv,  \dot \tv),
\end{equation}
on a curved edge, where $\mt\sigma$ is the matrix whose $(i,j)$th entry is $\sigma_{ij}$.

We proceed to display the curvature approximation formula in
coordinates for this model case using the Euclidean
normal ($\nv$) and tangent ($\tv$).

\begin{lemma}
  \label{lem:geodesic-M-Om}
  The geodesic curvature along any mesh edge $E$ is given by
  \[
    \kappa(g) = \frac{\sqrt{\det g} }{g_{\tv\tv}^{3/2}} (  \dot{\tv}^\nv + \Gamma_{\tv\tv}^\nv)
  \]
  where
  $\Gamma_{\tv\tv}^\nv = \tv^i \tv^j \Gamma_{ij}^k \Eucl_{kl}\nv^l$
  and $\dot{\tv}^\nv = \delta(\dot\tv, \nv)$.
\end{lemma}
\begin{proof}
  Comparing~\eqref{eq:gnu-eucl} with~\eqref{eq:gnu},
\begin{equation}
  \label{eq:g-delta-equiv}
  g(\nut, X) = \Eucl(\nv, X) 
\end{equation}
for any $X \in T_pM.$ Using~\eqref{eq:g-delta-equiv} in
Lemma~\ref{lem:geosdesiccurv}, the result follows.
\end{proof}
In this model case of a single parameter domain,
  since the identity metric $\delta$ is well defined throughout, 
  the angle deficit
  in \eqref{eq:variousK}
  can be reformulated using $\delta$ as a sum of element contributions at each
  vertex. Define   
  \[
    \act{K_{V, g}^T, u_h}_{\VV(\T)} = \big(\agl_V^T\delta - \agl_V^Tg\big)\,u_h(V).
  \]
  When summing $\agl_V^T\delta$  over all $V \in \V_T$, we clearly obtain $2\pi$. Hence
  \begin{align}
    \label{eq:7}
	\sum_{V \in\V}\act{ K_{V, g},  \vphi}_{\VV(\T)} 
	= \sum_{T \in \T}\sum_{V \in \V_T}\act{K_{V, g}^T, u_h}_{\VV(\T)}.
	\end{align}
\begin{proposition}
  \label{prop:modelK}
  In this model case, equation \eqref{eq:1} implies that for all
  $u_h\in \Vo_h^k$,
  \begin{equation}
    \label{eq:8}
    \begin{aligned}
      \int_\T
       K_h(g) \, u_h
       & =
       \sum_{T \in \T} 
       \sum_{V \in \V_T}
      \big(\agl_V^T\delta - \agl_V^Tg\big)\,u_h(V)
      \\
      &
      + \sum_{T \in \T} \bigg( 
      \int_{\d T }
      \frac{\sqrt{\det g}}{g_{\tv\tv}} (\dot{\tv}^{\nv} + \Gamma_{\tv\tv}^\nv)\, u_h
      \;\dl + 
            \int_T K(g)\, u_h \sqrt{\det g}\; \da
            \bigg).
    \end{aligned}
  \end{equation}
  If $g \in \Sc^+(M)$, then all terms vanish except the last.
\end{proposition}
\begin{proof}
  Equation~\eqref{eq:8} follows from \eqref{eq:7},  \eqref{eq:KT}, \eqref{eq:20}, and
  Lemma~\ref{lem:geodesic-M-Om}, once $\Phi$ is set to the identity.
  To prove the last statement, divide the set of mesh vertices $\V$
  into the set of boundary and interior vertices
  $
  \Vbnd= \V \cap \d M$
  and
  $ \Vint = \V \setminus \Vbnd.
  $
  At every interior vertex $V\in \Vint$, a rearrangement gives
  \begin{align*}
    \sum_{ T\in \T}\sum_{V \in \Vint_T} \act{K_{V, g}^T, u_h}_{\VV(\T)}
    = \sum_{V \in \Vint}\sum_{T \in \T_V}
    \big(\agl_V^T\Eucl - \agl_V^Tg\big) \, u_h(V),
  \end{align*}
  which  vanishes because the smoothness of $g$ implies
  $ \sum_{T \in \T_V} \agl_V^Tg = 2\pi = \sum_{T \in \T_V}
  \agl_V^T\Eucl.  $
  The element boundary integrals can be rewritten using the
  set of interior mesh edges $\Eint$, as 
  \begin{align}\label{eq:21}
      \sum_{T \in \T}
      \int_{\d T}
      \frac{\sqrt{\det g}}{g_{\tv\tv}} (\dot{\tv}^{\nv} + \Gamma_{\tv\tv}^\nv)\, u_h
      \;\dl
     = \sum_{E \in \Eint}
      \int_{E}
      \frac{\sqrt{\det g}}{g_{\tv\tv}} \jmp{\dot{\tv}^{\nv} + \Gamma_{\tv\tv}^\nv}
      \, u_h
      \;\dl,
  \end{align}    
  since $u_h=0$ on $\Gamma_D$ and that the trace of
  $\sqrt{\det g} / g_{\tv\tv}$ is well defined (single valued) on $E$
  due to the given smoothness of $g \in \Sc^+(M)$.  It is easy to see
  that $\dot{\tv}$ has the same value from adjacent elements of $E$
  while $\tv$ and $\nv$ changes sign, so $\jmp{\dot{\tv}^\nv} = 0 $
  and $ \jmp{\Gamma_{\tv\tv}^\nv} = 0$.  Hence~\eqref{eq:21} vanishes.
\end{proof}

\section{Covariant derivatives using the nonsmooth metric}
\label{sec:cov}

The objective of this section is to formulate a covariant
incompatibility operator that can be applied to our situation with
piecewise smooth metrics. To this end, we first define several
covariant derivatives in the smooth case, restricting ourselves to a
single element, i.e., the smooth Riemannian manifold $(T, g|_T)$. Then
we proceed to consider the changes needed due to the jumps of the metric
across element interfaces.

\subsection{Covariant curl and incompatibility for smooth metric}

For a $\mu \in \W^1(T)$, the exterior derivative $d^1 \mu \in \W^2(T)$ is characterized in terms of the connection~\cite{Peter16} by
\begin{equation}
  \label{eq:d1og}
  (d^1 \mu)(X, Y) = (\nabla_X \mu)(Y) - (\nabla_Y \mu)(X),  
\end{equation}
for any $X, Y \in \Xm T.$ Next, for a $\sigma \in \TT^2_0(T)$ writing
$\sigma_Z(Y) = \sigma(Z, Y)$ for any $X, Y, Z \in \Xm T,$ we 
define an operation analogous to~\eqref{eq:d1og} on $\sigma$ by
\begin{equation}
  \label{eq:d1sigma}
  (d^1 \sigma_Z)(X, Y)
  =  (\nabla_X \sigma)(Z,Y) - (\nabla_Y \sigma)(Z, X).  
\end{equation}
Since the expressions in~\eqref{eq:d1og} and~\eqref{eq:d1sigma}--holding $Z$
fixed--are
skew-symmetric in $X$ and $Y$, they may be viewed as elements of $\W^2(T)$.
Then we may use the Hodge star $\star$ operation to convert them to
0-forms, since $T$ is two dimensional. Doing so, define 
\begin{subequations}
  \label{eq:curl_inc_group}
  \begin{align}
    \label{eq:curl_inc_group-1}
  \curl_g \mu & = \star (d^1\mu),
                   && \mu \in \TT^1_0(T) = \W^1 (T),
  \\ \label{eq:curl_inc_group-2}
  (\curl_g \sigma)(Z) & = \star (d^1\sigma_{Z}),
                   && \sigma \in \TT^2_0(T), \; Z \in \Xm T.
  \end{align}
  The latter, due to  linearity in $Z$, is in~$\W^1(T)$, while the former
  $\curl_g \mu$ is in $\W^0(T).$
Combining these operations in succession, we define the {\em covariant incompatibility},
\begin{align}
  \label{eq:incg}
  \inc_g \sigma & =  \curl_g\curl_g \sigma, && \sigma \in \TT^2_0(T),
\end{align}
\end{subequations}
on two-dimensional manifolds. Clearly $\inc_g \sigma$ is  in $\W^0(T)$.

We will now quickly write down coordinate expressions of these tensors.
Expanding
the right hand side of~\eqref{eq:d1sigma} using the Leibniz rule $
(\nabla_X \sigma)(Z,Y) = X\sigma(Z, Y) - \sigma(\nabla_XZ, Y) -
\sigma(Z, \nabla_XY)$ twice, substituting $Z= \d_i, X=\d_j, Y=\d_k$,
and simplifying using~\eqref{eq:Chris12}, 
\begin{equation}
  \label{eq:dsigma-ijk}
  (d^1 \sigma_{\d_i})(\d_j, \d_k)
  = (\pder{\sigma_{ik}}{j} - \Gamma_{ji}^m \sigma_{mk})
  - (\pder{\sigma_{ij}}{k} - \Gamma_{ki}^m \sigma_{mj})  
\end{equation}
for $\sigma = \sigma_{jk} dx^j \otimes dx^k \in \TT_0^2(T)$.  Note that equation~\eqref{eq:dsigma-ijk}
can be rewritten as
$ d^1 \sigma_{\d_i} = \veps^{jk} (\pder{\sigma_{ik}}{j} - \Gamma_{ji}^m
\sigma_{mk}) \,dx^1 \wedge dx^2.$
Next, since~\eqref{eq:curl_inc_group-2} implies $\curl_g \sigma  = (\star d^1\sigma_{\d_i}) dx^i,$
recalling that
$\star (f dx^1 \wedge dx^2) = f / \sqrt{\det g}$ for any
scalar field $f \in \W^0(T)$, we arrive at
\begin{subequations}
  \label{eq:coords-curlogsigma}
\begin{align}
  \label{eq:curlgsigma-coord}
   \curl_g (\sigma_{jk} dx^j \otimes dx^k)
  & =
  \frac{1}{\sqrt{\det g}}
    \veps^{jk} (\pder{\sigma_{ik}}{j} - \Gamma_{ji}^m \sigma_{mk}) dx^i.
  \intertext{Similarly, one obtains coordinate expressions for the remaining
  operators in \eqref{eq:curl_inc_group}, namely}
  \label{eq:curlgog-coord}
  \curl_g (\mu_i dx^i)
  & =
    \frac{1}{\sqrt{\det g}} \veps^{ij} \pder{\mu_j}{i},
  \\ \nonumber 
  \inc_g(\sigma_{jk} dx^j \otimes dx^k)
  & =  \frac{1}{\det g}
    \Big( \veps^{qi} \veps^{jk} \pder{\sigma_{ik}}{jq}
    -  \veps^{qi} \veps^{jk} \pder{( \Gamma_{ji}^m \sigma_{mk})}{q}
  \\ 
  & \hspace{2.5cm} \label{eq:incg-coord}
    - \Gamma_{lq}^l \veps^{qi} \veps^{jk}
    (\pder{\sigma_{ik}}{j} - \Gamma_{ji}^m \sigma_{mk})
    \Big). 
\end{align}
\end{subequations}
In the derivation of \eqref{eq:incg-coord}, we have employed the
useful formula
\begin{equation}
  \label{eq:der-det}
  \Gamma^l_{lq} = \frac{\pder{(\det g)}{q}}{ 2\det g}.  
\end{equation}

It is useful to contrast the expressions
in~\eqref{eq:coords-curlogsigma} with the standard Euclidean curl and
inc.  To this end we use matrix and vector proxies,
$\mt \sigma \in \R^{2 \times 2}$ and $\mt \mu \in \R^2$, of
$\mu \in \W^1(T)$ and $\sigma \in \TT_0^2(T)$, respectively~\cite{Arnol18}.
These proxies are made up of coefficients in the coordinate basis expansion,
specifically $\mt \sigma$ is the matrix whose $(i,j)$th entry is
$\sigma_{ij} = \sigma(\d_i, \d_j)$, and $\mt \mu$ is the Euclidean
vector whose $i$th component, denoted by $\mt{\mu}_i$, equals
$\mu_i = \mu(\d_i)$.  Then the standard two-dimensional curl operator
applied to the vector $\mt \mu$ gives
$\curl \mt \mu = \veps^{ij} \pder{\mu_j}{i}$. When this operator is
repeated row-wise on a matrix, we get the standard row-wise matrix
curl, namely $\mt{\curl \mt \sigma}_i = \veps^{jk} \d_j
\sigma_{ik}$. The standard Euclidean incompatibility operator
\cite{ArnolAwanoWinth08, AG16, ChrisGopalGuzma21} in two dimensions is
$\inc \mt \sigma = \veps^{qi} \veps^{jk}
\pder{\sigma_{ik}}{jq}$. Using these, we can rewrite the formulas in
\eqref{eq:coords-curlogsigma} as
\begin{subequations}
  \begin{align}
    \mt{\curl_g \mu}
    & =
      \frac{1}{\sqrt{\det g}} \curl \mt\mu,
    \\
    \label{eq:curlgsigma-coord-matvec}
    \mt{\curl_g \sigma}_i 
    & =
      \frac{1}{\sqrt{\det g}}
      (\mt{\curl \mt\sigma}_i - \veps^{jk}\Gamma_{ji}^m \sigma_{mk}), 
    \\ \label{eq:incg-coord-matvec}
    \mt{\inc_g\sigma} 
    & =  \frac{1}{\det g}
      \Big( \inc\mt\sigma
      -  \veps^{qi} \veps^{jk} \pder{ (\Gamma_{ji}^m \sigma_{mk})}{q}   
      - \Gamma_{lq}^l \veps^{qi}
      (\curl \mt\sigma _i - \veps^{jk} \Gamma_{ji}^m \sigma_{mk})
      \Big). 
  \end{align}
\end{subequations}
Note how the expressions for covariant curl and covariant inc
contain, but differ from their Euclidean analogues.

Other useful covariant operators include
\begin{subequations}
  \begin{align}
  \label{eq:rotg-scalar}
  \rot_g f & = -(\star d^0 f)^\sharp, && f \in \W^0(T),
  \\
  (\rot_g X)(\mu,\eta)
            & = \mu(\nabla_{(\star \eta)^\sharp} X),
                                   && \mu, \eta \in \W^1(T), \;X \in \Xm T.
\end{align}
\end{subequations}
Clearly, $\rot_g f$ is in $\Xm T$, while $ \rot_g X \in \TT^0_2(T).$
Their coordinate expressions are
\begin{subequations}
  \begin{align}
  \rot_g f & = \frac{\veps^{kp} \d_p f }{ \sqrt{\det g}} \d_k =
             \frac{[\rot f]^k }{ \sqrt{\det g}} \d_k,
  \\ \label{eq:rotg-X}
  \rot_g X & = \frac{\veps^{pi} (\d_i X^m + \Gamma^m_{ik} X^k)}{\sqrt{\det g}}
              \d_m \otimes \d_p
             = \frac{[\rot [X]]^{mp} + \veps^{pi} \Gamma^m_{ik} X^k}{\sqrt{\det g}}
              \d_m \otimes \d_p,
\end{align}
\end{subequations}
where, in the latter expressions, the proxy notation has been extended
to $\Xm T$ and $\TT_2^0(T)$ in an obvious fashion to use the Euclidean
matrix and vector $\rot.$

It is easy to see
that the following integration by parts formula
\begin{equation}
  \label{eq:curg-rotg-intg-by-parts}
  \int_{(T, g)} (\curl_g \sigma)(Z) = \int_{(T, g)} \sigma( \rot_g Z)
  + \int_{(\d T, g)} \sigma(Z, \gt)
\end{equation}
holds for all  $\sigma \in \TT^2_0(T)$ and $Z \in \Xm T$.
(This can be seen, e.g., using the coordinate
expressions~\eqref{eq:rotg-X} and \eqref{eq:curlgsigma-coord-matvec}
and standard integration by parts on the Euclidean parameter domain.)
Here $\gt$ is the unit tangent defined in
\eqref{eq:normalized-nu-tau}, the integrals are computed as indicated
in \eqref{eq:intTf}, and $ \sigma( \rot_g Z)$ denotes the result
obtained by acting the $\TT^2_0(T)$-tensor $\sigma$ on the
$\TT^0_2(T)$-tensor $\rot_g
Z$. Equation~\eqref{eq:curg-rotg-intg-by-parts} shows that $\rot_g$
can be interpreted as the adjoint of $\curl_g$. A similar integration by parts formula for $\phi \in \W^0(T)$ and $\mu\in \W^1(T)$ 
\begin{equation}
  \label{eq:curg-rotg-intg-by-parts-simpler}
  \int_{(T, g)} \phi\, \curl_g \mu = \int_{(T, g)} \mu(\rot_g \phi)
  + \int_{(\d T, g)} \phi \,\mu(\gt) 
\end{equation}
connects the other $\curl_g$ and $\rot_g$ defined in
\eqref{eq:curl_inc_group-1} and \eqref{eq:rotg-scalar}.

\subsection{Covariant curl in the Regge metric}
\label{sec:covariant-curl-regge}

We proceed to extend the definitions of the covariant operators to the
case when the metric $g$ is only $tt$-continuous (see
\eqref{eq:tt-cts}) across element interfaces.  Let $\Moo$ denote the
open set obtained from $M$ by removing all the mesh vertices (after
its triangulation by~$\T$). The topological manifold $\Moo$ can be
endowed with a natural {\em glued smooth structure} based on the
$tt$-continuity of $g$, as alluded to  in the 
literature~\cite{Cheeger84, ClarkDray87, Kosov02, Kosov04}, \cite[Theorem~3.3]{li18} and~\cite{Warde06}.  This glued smooth structure is different from that
given by the coordinates $x^i$  (see
\S\ref{ssec:model}) in which we plan to conduct all computations.  A 
striking difference is that while $g$ is only $tt$-continuous in
$x^i$, it is fully continuous in the natural glued smooth structure.

The glued smooth structure can be understood using the following coordinate
construction around any interior mesh edge $E$. Let $z \in E$ and let
$U_z$ denote an open neighborhood of $z$ not intersecting any other
mesh edge or mesh vertex.  Let $d_g(\cdot, \cdot)$ denote the distance
function generated by $g$ on the manifold $\Moo$. For any $p \in U_z$,
let $\pi(p) = \arg\min_{q \in E} d_g(q, p).$ We use $T_\pm, \gn_\pm, \gt_\pm$
introduced in~\eqref{eq:jump-def}.  Let $E_p$ denote the submanifold
of $E$ connecting $z$ to $\pi(p)$ oriented in the $\gt_+$~direction.
Then for any $p \in U_z$, define new coordinates
\begin{equation}
  \label{eq:Kosovskii-coords}
  \xoo^1(p) = \pm d_g(\pi(p), p) \quad\text{if } p \in T_\pm,
  \qquad
  \xoo^2(p) = \int_{(E_p, g)} 1.  
\end{equation}
Denote the coordinate frame of $\xoo^i$ by $\doo_i$. It can be
shown~\cite{Kosov02} that $\doo_1 = \gn_+$ and $\doo_2 = \gt_+$ at
points in $U_z \cap E$, and that $g(\doo_i,\doo_j)$ is continuous
across $U_z \cap E$ for all $i,j$. Augmenting the set $\Moo$ with the
maximal atlas  giving such coordinates, we obtain a manifold with
the glued smooth structure, which we continue to denote by $\Moo.$
Moreover, $(\Moo, g)$ is a Riemannian manifold with piecewise smooth
and {\em globally continuous metric} $g$.

For the next result, we need the subspace of smooth vector fields
\[
  \Xmo\Moo= \{ X \in \Xm \Moo: \text{  support of $X$ is
  relatively compact in $\Moo$}  \}.
\]
Because the transformations
$
\d_i = (\d \xoo^j / \d x^i ) \doo_j
$ and 
$  dx^i = (\d x^i / \d \xoo^j ) d\xoo^j
$
are smooth within mesh elements, previously defined piecewise smooth
spaces like $\RR(\T)$ carry over to the glued smooth structure.

\begin{prop}
  \label{prop:curlg-ext}
  For all $\sigma \in \Regge(\T)$ and $\vphi \in \Xmo \Moo$, we have
  \begin{equation}
    \label{eq:curlg-ext}
    \int_{(\Moo, g)} \sigma (\rot_g \vphi)
    =
    \int_\T (\curl_g\sigma)(\vphi) - \int_{\d \T} g(\vphi, \gn) \sigma(\gn, \gt).
  \end{equation}
\end{prop}
\begin{proof}
  The integral on the left hand side of \eqref{eq:curlg-ext} may equally well be
  written as $\int_\T \sigma(\rot_g \vphi),$ since the set of vertices
  excluded in $\Moo$ is of measure zero.  Then, integrating by parts, 
  element by element, using \eqref{eq:curg-rotg-intg-by-parts},
  \[
    \int_{\T}\sigma (\rot_g \vphi)
    = \int_\T (\curl_g \sigma) (\vphi)
    - \int_{\d\T} \sigma(\vphi, \gt).
  \]
  Now, using the $g$-orthogonal decomposition
  $\vphi = g(\vphi, \gt)\gt + g(\vphi, \gn) \gn$, we have
  \begin{align*}
    \int_{\d\T} \sigma(\vphi, \gt)
    = \int_{\d\T} g(\vphi, \gn) \sigma(\gn, \gt) + \int_{\d\T}
    g(\vphi, \gt) \sigma(\gt, \gt).
  \end{align*}
  Boundary mesh edges do not contribute to the last integral since
  $\vphi$ is compactly supported. Across an interior mesh edge, since
  $g$ is continuous in the glued smooth structure of $\Moo$, and since
  $\sigma$ is $tt$-continuous, the contributions to the last integral
  from adjacent elements cancel each other. Hence  \eqref{eq:curlg-ext} follows.
\end{proof}

We use $(\Moo, g)$ to extend the definition of covariant curl.  Recall
that the adjoint of $\curl_g$ is $\rot_g$, as shown
by~\eqref{eq:curg-rotg-intg-by-parts}. Hence, as in the theory of Schwartz distributions, a natural extension
would be to consider $\curl_g\sigma$, for $\sigma \in \Regge(\T)$, as
a linear functional on $\Xmo \Moo$ defined by
\begin{align}
  \nonumber
  \act{\curl_g\sigma, \vphi}_{\Xmo \Moo}
  & = \int_{(\Moo, g)} \sigma (\rot_g \vphi)
  \\ \label{eq:dist-curlg-firstdef}
  & =
    \int_\T (\curl_g\sigma)(\vphi) - \int_{\d \T} g(\vphi, \gn) \sigma(\gn, \gt)
\end{align}
for all $\vphi \in \Xmo \Moo$, where we have used
Proposition~\ref{prop:curlg-ext} in the second equality.  The next key
observation is that we may extend the above functional to act on a
piecewise smooth vector field $W \in \Xm \T$ instead of the smooth $\vphi$, {\em
  provided $W$ is $g$-normal continuous} (an intrinsically verifiable property on the manifold).

\begin{definition}
  \label{def:curlg-extended}
  For any $\sigma \in \Regge(\T)$, define $\curl_g \sigma$ as a linear
  functional on $\Wo_g(\T)$, the space defined in \eqref{eq:Ngspaces},  by
  \begin{equation}
    \label{eq:curlg-extended-2}
    \act{\curl_g\sigma, W}_{\Wo_g(\T)}
    =
    \int_\T (\curl_g\sigma)(W) - \int_{\d \T} g(W, \gn) \sigma(\gn, \gt)    
  \end{equation}
  for all $W \in \Wo_g(\T)$, where  the first term on the right hand side
  is evaluated
  using the smooth case in~\eqref{eq:curl_inc_group-2}.  
  The rationale for this definition is
  that there are functions $\vphi$ in $\Xmo \Moo$ approaching
  $W \in \Wo_g(\T)$ in such a way that the right hand side
  of~\eqref{eq:dist-curlg-firstdef} converges to that
  of~\eqref{eq:curlg-extended-2}: see
  Proposition~\ref{prop:dens}
  in Appendix~\ref{sec:rationale-g-normal}. Furthermore, because of the $g$-normal
  continuity of $W$, the last term in~\eqref{eq:curlg-extended-2}
  vanishes when $\sigma$ is globally smooth,
  so~\eqref{eq:curlg-extended-2} indeed extends $\curl_g$ on smooth
  functions. 
\end{definition}

\subsection{Implementation issues in computing covariant curl}

We now develop a formula for computing the extended covariant curl
in the given computational  coordinates  $x^i$ (not $\xoo^i$). We use the
Euclidean parameter domain $(\om, \delta)$ and the Euclidean
$\delta$-orthonormal tangent and normal ($\tv$ and $\nv$) on element
boundaries 
(see~\S\ref{ssec:model}).  Abbreviating $\delta(w, X)$ to
$w^X$, consider \eqref{eq:Ngspaces} for $g=\Eucl$, namely
\begin{align*}
 \WW(\om) &= \left\{ w \equiv [w^1, w^2] :
   \om \to \R^2 \;\big| \; w^i \in C^\infty(\T),
   \; \jump{ w^\nv} = 0\right\},\\
   \Wo(\om) &= \left\{ w \in \WW(\om) \;\big| \; w^{\nv}|_{\d\om} = 0 \right\}.
\end{align*}
Their finite element subspaces of interest are 
\begin{equation}
\label{eq:BDMFEspace}
\begin{gathered}
\WW_{h}^k = \{ w \in \WW(\om):  \text{ for all } T \in \T,
\; w|_T=\Phi_T^* \hat w \text{ for some }
\hat w \in \Pol^k(\hat T,\R^2) \}, 
\\
\Wo_{h}^k = \{ w \in \WW_{h}^k: w^{\nv}|_{\d \om} = 0 \}, 
\end{gathered}
\end{equation}
where the pull-back $\Phi_T^* \hat w$ is the Piola transformation. For $k>0$, \eqref{eq:BDMFEspace} coincides with the {\em Brezzi--Douglas--Marini finite element space} ($\BDM$) \cite{BDM85} on the parameter domain. 
%
In practice, it is more convenient to work with the $\BDM$ space
than~\eqref{eq:Ngspaces}. For any $w = [w^1, w^2] \in \Wo(\om)$,
let 
\begin{align}
\label{eq:def_Q_op}
 Q_g w = \frac{w^1(x^1, x^2) \d_1 + w^2(x^1, x^2) \d_2}{ \sqrt{\det g}}.
\end{align}

\begin{prop}[Extended covariant curl in computational coordinates] \hfill
  \label{prop:cov-curl-coordinates}
  \begin{enumerate}
  \item \label{item:cov-curl-coordinates-1}
    A vector field $w$ on $\om$ is in
    $\Wo(\om)$ if and only if $Q_g w$ is in
  $\Wo_g(\T)$.
  \item  \label{item:cov-curl-coordinates-2}
  For any $\sigma \in \Regge(\T)$ and $w \in \Wo(\om)$, 
  \begin{subequations}
    \label{eq:cov_distr_curl}
    \begin{align}
      \nonumber 
      \lefteqn{\act{ \curl_g\sigma, Q_gw}_{\Wo_g(\T)}}
      \\ \label{eq:cov_distr_curl1}
      &= \sum_{T\in\T} \bigg(
        \int_T
        \frac{
        \mt{\curl \mt\sigma}_i {{w}}^i
        -
        \sigma_{mk}\veps^{jk}\Gamma_{ji}^m
        w^i
        }{\sqrt{\det g}}        \,\da      
        - 
        \int_{\partial T}
        \frac{g_{\tv\tv}\sigma_{\nv\tv}-g_{\nv\tv}\sigma_{\tv\tv}}
        {g_{\tv\tv}\sqrt{\det g}}{{w}}^{\nv}\,\dl \bigg)
      \\
      &= \sum_{T\in\T}
        \bigg(\int_T
        \sigma_{mk}\frac{\mt{\rot \mt{{w}}}^{mk}
        -\veps^{kj}
        (\Gamma_{lj}^l{{w}}^m
        -\Gamma_{ji}^m{{w}}^i)}
        {\sqrt{\det g}}\, \da
        + 
        \int_{\partial T}
        \frac{\sigma_{\tv\tv} g_{i\tv}w^i}{g_{\tv\tv}\sqrt{\det g}}
        \, \dl
        \bigg).
        \label{eq:cov_distr_curl2}
    \end{align}
  \end{subequations}    
  \end{enumerate}
\end{prop}
\begin{proof}
 By \eqref{eq:g-delta-equiv} and \eqref{eq:nt-properties}, 
 \begin{equation}
   \label{eq:d-g-w-Qw}
   w^\nv = \delta(w, \nv) = g(w, \nut) = g(w, \gn\sqrt{g_{\nut\nut}})
   = g(w, \gn) \frac{ \sqrt{g_{\tv\tv}}}{ \sqrt{\det g }}
   = { g(Q_gw, \gn)}{ \sqrt{g_{\tv \tv}}}.    
 \end{equation}
 Hence the continuity of $g_{\tv \tv}$ across element interfaces
 implies that
 \[
   \jmp{g(Q_g w, \gn)} = 0
   \quad \text { if and only if }\quad 
   \jmp{w^\nv}=0.
 \]
 Similarly $w^{\nv}|_{\d\om}=0$ if and only if $g(Q_gw, \gn)=0$ vanishes
 on the boundary. Hence $Q_gw \in \Wo_g(\T)$ if and only if $w \in \Wo(\om).$

  To prove~\eqref{eq:cov_distr_curl1}, we write
  \eqref{eq:curlg-extended-2} in coordinates.  Observe that the
  coordinate expression for covariant curl
  in~\eqref{eq:curlgsigma-coord-matvec} and the integration
  formula~\eqref{eq:intTf} imply
  \[
    \int_\T (\curl_g\sigma)(Q_gw) = \sum_{T\in \T} \int_T
    \frac{
      \mt{\curl \mt\sigma}_i
      {\mt{Q_g w}}^i-\sigma_{mk}\veps^{jk}\Gamma_{ji}^m{\mt{Q_g w}}^i
    }{\sqrt{\det g}}  \sqrt{\det g}
    \,\da,      
  \]
  which coincides with the sum of integrals over elements $T$
  in~\eqref{eq:cov_distr_curl1}. For the element boundary integrals of
  $g(Q_gw, \gn) \sigma(\gn, \gt)$ contributing to
  $\act{\curl_g\sigma, Q_gw}_{\Wo_g(\T)}$, first note
  that~\eqref{eq:d-g-w-Qw} implies
  \begin{align}
    \label{eq:300}
    g(Q_gw, \gn) \sigma(\gn, \gt)
    & = \frac{w^{\nv}}{\sqrt{g_{\tv\tv}}}
      \frac
      {\sigma( \delta(\gn, \nv)\nv + \delta(\gn, \tv) \tv, \tv)}
      {\sqrt{g_{\tv\tv}}}.
  \end{align}
  We simplify by chasing the definitions of $\gn$ and $\tv$. (We detail the
  argument this once and will not expand on later similar occasions.)
  \begin{align*}
    \delta(\gn, \nv)
    & = g(\gn, \nut) = g(\gn, \gn \sqrt{g_{\nut\nut}})
    && \text{ by~\eqref{eq:g-delta-equiv} and~\eqref{eq:normalized-nu-tau}}
    \\
    & = \frac{ \sqrt{g_{\tv\tv}}}{\sqrt{\det g}}
    && \text{ by \eqref{eq:nt-properties},}
    \\
    \delta(\gn, \tv)
    & = \frac{\delta(\nut, \tv)}{\sqrt{g_{\nut\nut}}}
      = \frac{ \sqrt{\det g}}{\sqrt{g_{\tv\tv}}}\, \delta(\nut, \tv)
    && \text{ by~\eqref{eq:normalized-nu-tau} and \eqref{eq:nt-properties},}
    \\
    & = -\frac{ \sqrt{\det g}}{\sqrt{g_{\tv\tv}}}\,
      g^{kj} \veps_{j i} \tv^i \delta_{km} \tv^m
    && \text{ by \eqref{eq:nu-tau-coords},}
    \\
    & = \frac{-g(\nv, \tv) }{ \sqrt{g_{\tv\tv} \det g}},
  \end{align*}
  where, in the last step, we have simplified using the cofactor
  expansion of $g^{kj}$ and \eqref{eq:Eucl-normal-comp}.
  Hence~\eqref{eq:300} implies
  \[
    g(Q_gw, \gn) \sigma(\gn, \gt) =
    \frac{w^{\nv}}{g_{\tv\tv}}
    \left( \frac{\sqrt{g_{\tv\tv}}}{\sqrt{\det g}} \sigma_{\nv\tv} 
      - \frac{g_{\nv\tv}}{\sqrt{g_{\tv\tv}\det g}} \sigma_{\tv\tv} 
    \right)
    =
    \frac{w^{\nv}(\sigma_{\nv\tv}g_{\tv\tv} - g_{\nv \tv} \sigma_{\tv\tv})}
    {g_{\tv\tv}\sqrt{ g_{\tv\tv} \det g}}.    
  \]
  Integrating this over each element boundary using the measure
  $\sqrt{g_{\tv \tv}} \, \dl$ and summing over elements,
  the right hand sides of \eqref{eq:curlg-extended-2} 
  and~\eqref{eq:cov_distr_curl1} are seen to be the same.

  To prove the second identity~\eqref{eq:cov_distr_curl2}, consider
  any $W \in \Wo_g(\T)$.  We start by applying the integration by
  parts formula \eqref{eq:curg-rotg-intg-by-parts} to the first term
  on the right hand side of the definition
  \eqref{eq:curlg-extended-2}:
  \begin{align*}
    \act{\curl_g\sigma, W}_{\Wo_g(\T)}
    & = \int_\T \sigma(\rot_g W) + \int_{\d T} \sigma(W, \gt)
      -\int_{\d T} g(W, \gt)\, \sigma(\gn,\gt)
    \\
    & = \int_\T \sigma(\rot_g W) + \int_{\d \T} g(W, \gt) \, \sigma_{\gt\gt}
  \end{align*}
  after simplifying using $g$-orthogonal decomposition
  $W = g(W, \gt)\gt + g(W, \gn) \gn$ on element
  boundaries. Substituting $W = Q_gw$, applying the quotient rule to
  compute $\rot_g(Q_gw)$ using \eqref{eq:der-det}, and expressing the
  result in $x^i$ coordinates, \eqref{eq:cov_distr_curl2} follows.
\end{proof}

In analogy with the finite element curvature approximation, we may now
also lift the functional $\curl_g \sigma$ to a finite element space to
get a computable representative of the covariant curl
on the parameter domain $\om$.  Using the BDM space in
\eqref{eq:BDMFEspace}, we define $\curl_{g, h} \sigma$, for any
$\sigma \in \Regge(\T)$, as the unique element in $\Wo_{h}^k$
satisfying
\begin{equation}
  \label{eq:cov-curl-lift}
  \int_\om \delta( \curl_{g, h} \sigma, w_h) \,\da
  = \act{ \curl_g\sigma, Q_gw_h}_{\Wo_g(\T)}
  \quad \text{ for all } w_h \in \Wo_h^k,
\end{equation}
where the right hand side can be evaluated using either of the
formulas in \eqref{eq:cov_distr_curl}.

\subsection{Covariant incompatibility in the Regge metric}
\label{sec:covar-incomp}

To extend the smooth covariant incompatibility defined in
\eqref{eq:incg}, we use the space $\VV(\T)$ of piecewise smooth and
globally continuous functions on~$M$. For any $u \in \VV(\T)$, the
vector field $\rot_g u$, by definition \eqref{eq:rotg-scalar}, satisfies
$
g(\rot_g u, \gn) = g(-(\star d^0 u)^\sharp, \gn) = -(\star d^0 u)(\gn).
$
Hence \eqref{eq:star-rotation} implies
\begin{equation}
  \label{eq:rotgu-du}
  g(\rot_g u, \gn) = -(d^0u)(\gt),
\end{equation}
so in particular, $\jmp{ g(\rot_g u, \gn)} = 0$ due to the continuity
of $u$.  Also note that in the formula \eqref{eq:incg} for the
smooth case,  
$\inc_g \sigma = \curl_g \curl_g \sigma$,  
the outer $\curl_g$'s adjoint is the $\rot_g$ appearing in
\eqref{eq:curg-rotg-intg-by-parts-simpler}. These facts motivate us to
use Definition~\ref{def:curlg-extended} to extend $\inc_g$ to
$tt$-continuous $\sigma$ as follows.

\begin{definition}
  \label{def:incg-extended}
  For any $\sigma \in \Regge(\T)$, extend $\inc_g \sigma$ as a linear
  functional on $\Vo(\T)$ by
  \begin{equation}
    \label{eq:def_distr_cov_inc}
    \act{\inc_g \sigma, u}_{\Vo(\T)} =
    \act{\curl_g \sigma, \rot_g u}_{\Wo_g(\T)}
  \end{equation}
  for all $u \in \Vo(\T)$. Note that $\rot_g u$ is an allowable
  argument in the right hand side pairing since it is $g$-normal
  continuous (and hence in $ \Wo_g(\T)$) by~\eqref{eq:rotgu-du}.  The
  next result shows that~\eqref{eq:def_distr_cov_inc} indeed extends
  the smooth case.
\end{definition}

\begin{figure}
	\centering
	\begin{tikzpicture}[scale=1.5]
        \draw[draw=black] (0,0) --(3,0)--(1.5,1.5)--cycle;

	\node (A) at (1.5, 0.55) [] {$T$};
	\node (V0) at (-0.15, -0.3) [] {$V_0$};
	\node (V1) at (3.4, -0.0) [] {$V_1$};
	\node (V2) at (1.5, 1.8) [] {$V_2$};
        
	\node (A) at (2.7, 0.75) [] {$E_0$};
        \node (A) at (0.7, 1.3) [] {$E_1$};
	\node (A) at (1.5, -0.4) [] {$E_2$};
	
	\node (A) at (3.5, -0.65) [] {\textcolor{teal}{$\jmp{\sigma_{\gn\gt}}_{V_1}^T$}};
	\fill[teal](3,0) circle (2pt);
	
%
	
	\draw[very thick,->,color=blue](0.8,0) to  node[above]{$\gt$} (1.5,0);
	\draw[very thick,->,color=red](0.8,0) to  node[above left]{$\gn$} (0.8,0.7);
	
	\draw[very thick,->,color=blue](3-0.2-0.6,0.6+0.2) to  node[above right]{$\gt$} (3-0.2-0.6-0.45,0.6+0.45+0.2);
	\draw[very thick,->,color=red](3-0.2-0.6,0.6+0.2) to  node[below]{$\gn$} (3-0.2-0.6-0.45,0.6-0.45+0.2);
	
	\centerarc[teal,thick,latex-](3,0)(135:180:0.9);
	\centerarc[teal,thick,latex-](3,0)(135:180:0.7);
	\centerarc[teal,thick,latex-](3,0)(135:180:0.5);
	\centerarc[teal,thick,latex-](3,0)(135:180:0.3);
	\end{tikzpicture}	
	\caption{Illustration of the vertex jump defined in \eqref{eq:vertexxjump}.}
	\label{fig:vertex_jump}
\end{figure}
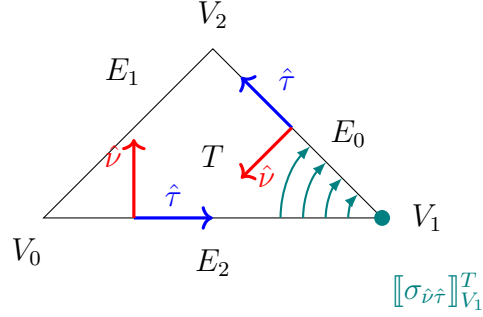

The ``vertex jump'' of $\sigma$ at a vertex $V \in \V_T$ of an element
$T \in \T$ (cf.~\cite{christiansen11,Hauret13} and see
Figure~\ref{fig:vertex_jump}), denoted by
 $\jmp{\sigma_{\gn \gt}}_{V}^T$,
represents the jump in the value of $\sigma(\gn, \gt)$ across the
vertex $V$ when traversing $\d T$ in the $\gt$-direction. Alternately,
enumerating the vertices of $\V_T$ as $V_0, V_1, V_2$ so that the
indices increase while moving in the $\gt$-direction, naming the edge
opposite to $V_i$ as $E_i$, and calculating the indices mod~3, we put
\begin{align}
  \label{eq:vertexxjump}
\jmp{\sigma_{\gn\gt}}_{V_i}^T = \left(\sigma(\gn, \gt)|_{E_{i-1}}-\sigma(\gn, \gt)|_{E_{i+1}}\right)(V_i).
\end{align}

\begin{prop}
  \label{prop:cov_distr_inc}        
  For any $\sigma \in \Regge(\T)$ and $u \in \Vo(\T)$,
  \[
    \act{\inc_g \sigma, u}_{\Vo(\T)} =
    \int_{\T} u \;\inc_g \sigma
    - \int_{\d \T} u\; (\curl_g\sigma + d^0\sigma_{\gn \gt})(\gt)
    -\sum_{T \in \T}\sum_{V \in \V_T}\jump[V]{\sigma_{\gn \gt}}^Tu(V).
  \]
  When $\sigma$ is globally smooth, all terms on the right hand side
  vanish except the first.
\end{prop}
\begin{proof}
  By \eqref{eq:def_distr_cov_inc}, \eqref{eq:curlg-extended-2},
  \eqref{eq:curg-rotg-intg-by-parts-simpler}, and \eqref{eq:rotgu-du},
  \begin{align*}
    \act{\inc_g \sigma, u}_{\Vo(\T)}
    & =
    \int_\T (\curl_g \sigma)(\rot_g u)
    - \int_{\d \T} \sigma_{\gn \gt} \;g( \rot_g u, \gn)
    \\
    & =
      \int_\T u \; \inc_g \sigma - \int_{\d \T} u\; (\curl_g\sigma)(\gt)
      +\int_{\d \T} \sigma_{\gn \gt} \; (d^0u)(\gt).
  \end{align*}
  Now, on the edge $E_i$ oriented from end point $V_{i-1}$ to $V_i$
  (indexed mod~3; see Figure~\ref{fig:vertex_jump}), by the
  one-dimensional integration by parts formula
  \[
    \int_{(E_i, g)} \sigma_{\gn \gt} \; (d^0u)(\gt) 
    =
    \sigma_{\gn \gt} (V_i) \, u \big(V_i)
    - \sigma_{\gn \gt} (V_{i-1}) \, u \big(V_{i-1})    
    -\int_{(E_i, g)} (d^0\sigma_{\gn \gt})(\gt)\; u.     
  \]
  When summing over the three edges $E_i \subset \d T$, the above
  vertex values of $\sigma_{\gn\gt}$ yield vertex jumps. Hence the
  first statement  follows by summing over all $T$ in $\T$.

  To prove the second statement, note that the edge integrals from adjacent
  triangles cancel each other. To show that the last term with vertex
  contributions also vanish, let $\E_V$ denote the set of vertices
  connected to  $V \in \V$ by an edge and let $\gt_V$ denote the $g$-unit
  tangent vector along an $E \in \E_V$ pointing away from~$V$.  Then the jump
  $\jmp{\sigma (\gn, \gt_V)}$ on any edge $E \in \E_V$
  is defined as before, using  \eqref{eq:jump-def}.
  Its limit as we approach a vertex $V$ along any edge $E \in \E_V$,
  is denoted by $\jmp{\sigma}^V_{\gn \gt}$.  Using it, the
  last sum can be rearranged to
  \[
     \sum_{T \in \T}\sum_{V \in \V_T}\jump[V]{\sigma_{\gn \gt}}^T\; u(V)
    = -\sum_{V \in \Vint} \sum_{E \in \E_V} \jmp{\sigma}^V_{\gn\gt}  \, u(V).
  \]
  where $\Vint \subset \V$ is the subset of mesh vertices in the
  interior of the domain.  
  Each summand on the right hand side vanishes when $\sigma$ is smooth.
\end{proof}

As before, one may now lift the incompatibility functional 
into a finite element space to get a computable
representative. Namely, for any
$\sigma \in \Regge(\T)$ and $g \in \Regge^+(\T)$, let
$\inc_{g, h}\sigma$ be the unique function in $\Vo_h^{k+1}$ satisfying
\begin{equation}
  \label{eq:incg-h}
  \int_\om (\inc_{g, h} \sigma) \,u_h \; \da
  = \act{ \inc_g \sigma, u_h}_{\Vo(\T)}, \qquad
  \text{ for all } u_h \in \Vo_h^{k+1},
\end{equation}
where the right hand side can be computed using the formula in
Proposition~\ref{prop:cov_distr_inc}.

\subsection{Linearization of curvature}
\label{subsec:connection_curv_inc}

Linearization of curvature was discussed in various forms by many
previous authors. For example, the linearization of our vertex curvature sources
can be guessed from the three-dimensional case presented in
\cite{christiansen11}, while that of the edge and interior curvature
sources were derived in~\cite{Gaw20} (making use of
\cite{FischMarsd75}) in terms of the covariant divergence operator.
Here we
revisit the topic to derive the linearization of edge and interior
curvature sources directly in terms of the covariant incompatibility and curl
operators. While doing so, we also present different and elementary
proofs.

The variational derivative of a scalar function $f:\Sc(T) \to\R$
in the direction of a
$\sigma \in \Sc(T)$ is given by 
\begin{align*}  
  D_{\sigma}(f(\rho)) =
  \lim\limits_{t\to 0} \frac{f(\rho+t\sigma)-f(\rho)}{t}
\end{align*}
when the limit exists. We use this exclusively for scalar functions of the metric $g$ (i.e., with $\rho=g$ above). Note that the changes in $\rho(p)$ and
$\sigma(p)$ as $p$ varies in $M$ are immaterial in the above
definition. Hence, we may use Riemann normal coordinates~\cite{Lee97}
to prove the pointwise identities of tensorial quantities
in the next result. Let $\xtt^i$
denote the Riemann normal coordinates given by a chart covering a
point $p \in T,$ $\dtt_i$ denote the corresponding coordinate frame,
$\sigmatt_{ij} = \sigma(\dtt_i, \dtt_j)$ for any
$\sigma \in \TT_0^2(T)$,
$\Rtt_{ijkl} = R(\dtt_i, \dtt_j, \dtt_k, \dtt_l),$ and let
$\Gamtt_{ijk}, \Gamtt_{ij}^k$ be defined by \eqref{eq:Christoffels} with
$\d_i$ replaced by $\dtt_i$.  Then, by \cite[Proposition 5.11]{Lee97},
\begin{equation}
  \label{eq:normalmagic}
  \gtt_{ij}|_p = \delta_{ij},
  \quad
  \dtt_k \gtt_{ij} |_p = 0,
  \quad
  \Gamtt_{ijk}|_p = 0,
\end{equation}
which greatly simplify calculations. As an example, consider
the expression for covariant incompatibility of a $\sigma \in \Sc(T)$
given by \eqref{eq:incg-coord-matvec} in
coordinate proxies.
It simplifies in normal coordinates,
by virtue of~\eqref{eq:normalmagic}, to
$ \mt{\inc_g \sigma} =\inc \mt \sigmatt
-\veps^{ij}\veps^{kl}\sigmatt_{ml}\dtt_i\Gamtt^m_{jk}.  $ Expanding
out the last term in terms of $\gtt_{ij}$, using
\eqref{eq:Chris12-coords}, and simplifying,
\begin{align}
  [\inc_g \sigma] & =
    \inc \mt \sigmatt
    -
    \frac{1}{2}\tro\mt\sigmatt\, \inc \mt\gtt.
    \label{eq:cov_inc_normal_coord}
\end{align}

\begin{lemma}[Variations of curvature terms]
	\label{lem:variation}
	Consider an element manifold  $(T, g)$ for $T \in \T$. Let
        $p$ be an arbitrary point in $T$ and let
        $X, Y \in T_pM$. Let
        $q$ be any point in one of the edges $E$ of $\d T$ 
        and let $\tau \in T_qE$. Then
        \begin{subequations}
	\begin{align}
          \label{eq:variation-1}
          &\VDer{K(g) \vol T (X, Y)}{\sigma}
            = -\frac{1}{2}\vol T (X, Y) \,\inc_g \sigma, 
          &&  \text{ at the point } p \in T, 
          \\           \label{eq:variation-2}
          &\VDer{\kappa(g) \vol {E} (\tau)}{\sigma}
            =\frac{1}{2}(\curl_g\sigma+d^0\sigma_{\gt\gn})(\tau),
          &&  \text{ at the point } q \in E, 
          \\           \label{eq:variation-3}
          & 
            \VDer{
            \agl_V^T g}{\sigma} =
            -\frac{1}{2}\jmp{\sigma_{\gn\gt}}_V^T,
          && \text{ at every vertex $V$ of } T.
	\end{align}          
        \end{subequations}
\end{lemma}
\begin{proof}
  We prove~\eqref{eq:variation-1},  
  using the coordinate formula
  $\Gauss(g)= R_{1221}/\det g$ and the Jacobi formula, which implies
  $    \VDer{\sqrt{\det g} }{\sigma} = \frac{1}{2}\sqrt{\det g}
  \,\tr{g^{-1}\sigma}.
  $
  Together they give
  \begin{align*}
    \VDer{K(g)\vol T (\d_1, \d_2)}{\sigma} =
    \frac{1}{\sqrt{\det g}}
    \Big(\VDer{R_{1221}}{\sigma}-\frac{1}{2}\tr{g^{-1}\sigma}R_{1221}\Big)
  \end{align*}
  in any coordinate frame.
  Specializing to Riemann normal coordinates,
  since~\eqref{eq:normalmagic} implies that the last two terms
  of \eqref{eq:Rijkl} vanish, 
  $\Rtt_{1221} = \dtt_1 \Gamtt_{221} - \dtt_2 \Gamtt_{211}  = -\frac 1 2
  \inc \mt\gtt$, so 
  \begin{align*}
    \VDer{K(g)\vol T (\dtt_1, \dtt_2)}{\sigma}
    &
      = -\frac{1}{2}\inc \mt\sigmatt + \frac{1}{4}\tro\mt\sigmatt \,\inc \mt\gtt.
  \end{align*}
  Hence~\eqref{eq:variation-1} follows from~\eqref{eq:cov_inc_normal_coord}.

  To prove~\eqref{eq:variation-2}, we start with its right hand side.
  By \eqref{eq:curlgsigma-coord-matvec},
  \begin{align*}
    (\curl_g \sigma + d^0 \sigma_{\gt\gn})(\tv)
    & =
      \frac{1}{\sqrt{\det g}}
      (\mt{\curl \mt\sigma}_i - \veps^{jk}\Gamma_{ji}^m \sigma_{mk})\tv^i
      + \tv^i \d_i  \sigma_{\gt\gn}.
  \end{align*}
  At any point on the
  edge $E$, without loss of generality, we may choose a Riemann
  normal coordinate system so that $\dtt_1 = \tilde\tv = \gt$ and
  $\dtt_2 = \tilde\nv = \gn$. Then, using~\eqref{eq:normalmagic}, the
  above expression becomes
  \begin{align}
    (\curl_g \sigma + d^0 \sigma_{\gt\gn})(\tilde\tv)
    & =
      \mt{\curl \mt\sigmatt}_i\tilde\tv^i
      +  \d_{\tilde\tv}  \sigma_{\tilde\tv\tilde\nv}
      \nonumber 
    \\ \label{eq:11}
    &
      = \mt{\curl \mt\sigmatt}_i\tilde\tv^i  +
      (\dtt_{\tilde \tv}\mt\sigmatt)_{\tilde\nv\tilde\tv}
      +
      (\sigma_{\tilde\nv \tilde\nv}
      -\sigma_{\tilde\tv \tilde\tv})\dot{\tilde\tv}_{\tilde\nv},
  \end{align}
  where we used \eqref{eq:10} to get the last equality. Now we work on
  the left hand side of \eqref{eq:variation-2}.  Differentiating the
  expression for geodesic curvature from Lemma~\ref{lem:geodesic-M-Om}, 
  we get
  \begin{align*}
    & \VDer{\kappa(g)\vol {E} (\tau)}{\sigma}
      =
    \VDer{g_{\tau\tau}^{-1} \sqrt{\det g}    
    (\dot{\tau_\nu} + \Gamma_{\tv\tv}^\nv)
    }
    {\sigma}
    \\
    & =
    \frac{\sqrt{\det g}}{g_{\tv\tv}}
    \bigg[
     \left( \frac{1}{2}\mathrm{tr}(g^{-1}\sigma)
      -\frac{\sigma_{\tv\tv}}{g_{\tv\tv}}\right)
      (\Gamma_{\tv\tv}^{\nv}+\dot{\tv}^{\nv})
      + \tv^i\tv^j \nv^l\delta_{kl}(g^{km} \Gamma_{ijm}(\sigma)
      - g^{ka}\sigma_{ab}g^{bm} \Gamma_{ijm})
      \bigg]
  \end{align*}
  in general coordinates. Specializing to the previously used Riemann
  normal coordinates, applying~\eqref{eq:normalmagic}
  and simplifying $\Gamma_{ijm}(\sigma)$, 
  \[
    \VDer{\kappa_g(g)\vol {E} (\tilde\tau)}{\sigma}
    =
    \Big(\frac{1}{2}\tro \mt\sigmatt-\mt{\sigmatt}_{\tv\tv}\Big)\dot{\tilde\tv}^{\nv}
    +
    (\dtt_{\tilde\tv}\mt\sigmatt)_{\tilde\nv\tilde\tv}
    -\frac{1}{2}(\dtt_{\tilde\nv}\mt\sigmatt)_{\tilde\tv\tilde\tv}.
  \]
  This coincides with the expression in~\eqref{eq:11}, so
  \eqref{eq:variation-2} is proved.

  To prove \eqref{eq:variation-3}, let $\theta = \agl_V^T g$ denote the angle in
  \eqref{eq:2} and let $\htinout = \tinout/\sqrt{g_{\tinout\tinout}}$
  denote the $g$-normalized incoming and outgoing tangents at~$V$.
  Since $\tau_\pm$ does not vary with~$g$, 
  \begin{align}
    \nonumber 
    -D_\sigma \cos \theta
    & = D_\sigma
      \Big(\frac{g_{\tin\tout}}{\sqrt{g_{\tin \tin}g_{\tout \tout}}}\Big)
     =
      \frac{ D_\sigma {g_{\tin \tout}} }{\sqrt{g_{\tin \tin}g_{\tout \tout}}}      
      +
      g_{\tin\tout}
      D_\sigma \frac{1}{\sqrt{g_{\tin \tin}g_{\tout \tout}}}
    \\ \label{eq:12}
    & = \sigma_{\htin \htout} - \frac 1 2
      g_{\htin\htout} ( \sigma_{\htin \htin} + \sigma_{\htout\htout}).
  \end{align}
  Let $\hnio$ be such that $\htinout, \hnio$ form an ordered
  orthonormal basis matching the orientation of $T$ under
  consideration. Then
  $g_{\htout\hni} = \sin\theta = - g_{\htin \hno}$.  Substituting
  $\sigma_{\htin \htout} =
  \frac 1 2 \sigma(\htin, g_{\htout\htin}\htin + g_{\htout \hni}\hni) +
  \frac 1 2 \sigma( g_{\htin\htout} \htout + g_{\htin \hno}\hno, \htout)$
  into~\eqref{eq:12},
  \[
    -D_\sigma \cos \theta = \frac 1 2
    \big(
    g_{\htout\hni}\sigma_{\htin \hni} +
    g_{\htin\hno} \sigma_{\hno\htout} \big)
    = \frac{\sin\theta}{2}
    \big(
    \sigma_{\htin \hni} 
    -
    \sigma_{\hno\htout} \big) =
    -\frac{\sin\theta}{2} \jmp{\sigma_{\gn\gt}}_{V}^T.
  \]
  Since $
  D_\sigma \theta = -(D_\sigma \cos\theta) / \sin\theta,
  $ the result is proved.
\end{proof}

\section{Connection approximation}
\label{sec:connection_one_form_approx}

In this section we approximate the Levi-Civita 
connection when only an approximation to the true metric
is given, namely 
$g\in \Regge^+(\T)$.
To do so, we assume we are given a $g$-orthonormal frame $(e_1, e_2)$ in
each $T \in\T$.
Then, the connection is fully determined by
a single {\em connection form} $\OneForm(g; \cdot) \equiv \cn_g \in\W^1(T)$,
within each element $T$, given by
\begin{align}
\label{eq:oneform_gortho_frame}
\OneForm_g(X)=g(\GBasis_1,\nabla_X\GBasis_2)= -g(\nabla_X \GBasis_1,\GBasis_2)
\end{align}
for any  $X\in\Xm{M}.$
This section is largely based on \cite{BKG21} (so we will be brief), but we note
that while they approximate the Hodge star of $\cn_g$, we
approximate $\cn_g$ directly (and also note that the orientation
in their work is opposite to ours).


To extend the connection to accommodate the possible discontinuities
of the frame $(e_1, e_2)$ across element interfaces, let
$\agl_g(a, b)$ denote the
counterclockwise angle from $b$ to $a$ measured in the $g$-inner
product,  for any two vectors $a, b \in T_pM$.
This angle is well defined even for points $p$ on a mesh
edge $E$ (excluding the vertices) since we use the glued smooth
structure (see \eqref{eq:Kosovskii-coords}) in which $g$ is continuous
across the edge.
On each interior mesh edge $E$, let $T_\pm, \gn_\pm, \gt_\pm$ be as
in~\eqref{eq:jump-def}, orient the edge $E$ by 
$\gt^E = \gt_+$, and  put $\gn^E = \gn_+$, $e_{\pm, i} = e_i|_{T_\pm}$. 
Let $\Theta^E = \agl_g(e_{+, 1}, e_{-, 1})$;  see Figure~\ref{fig:angle_comp}.
(It is possible to compute this angle without resorting to the coordinates in
\eqref{eq:Kosovskii-coords}, as we will explain later in Appendix~\ref{sec:oneform_algorithm}.)  This is the
angle by which a vector must be rotated while parallel transporting it
across the edge $E$ in the $\gn^E$ direction.  To account for this
rotation, we extend $\cn_g$ as follows:
\begin{definition}
  \label{def:dist-cn}
  Given  $g\in \Regge^+(\T)$ and $g$-orthonormal piecewise smooth frame
  $e_1, e_2 \in \Xm \T$, define $\cn_g \in \Wo_g(\T)'$ by
  \begin{equation}
    \label{eq:29}
    \act{\cn_g, W}_{\Wo_g(\T)}
    =  \int_{\T} \cn_g(W)
      + \sum_{E \in \Eint} \int_{(E, g)} \Theta^E \,g(W, \gn^E)
  \end{equation}
  for all $W \in \Wo_g(\T)$.  
\end{definition}
\begin{figure}[h]
	\centering
	\begin{tikzpicture}[scale=1.5]
	\draw[draw=black] (0,0.5) --(1.5,0)--(1.5,1.5)--cycle;
	\draw[draw=black] (1.5,0)--(1.5,1.5)--(3.1,0.9)--cycle;
	
	\draw[draw=black,->] (1.49,0.75) to (1.04,0.6);
	
	\draw[draw=black,->] (1.51,0.75) to (1.35,0.31);
	
	\draw[draw=teal, thick,->] (1.5,0.7) to (1.5,1.2);
	
	\node (A) at (0.4, 0.1) [] {$T_+$};
	\node (A) at (2.2, 0.1) [] {$T_-$};
	
	
	\node[teal] (A) at (1.3, 1.2) [] {$\gt^E$};
	
	\node (A) at (1.2, 0.8) [] {$\GBasis_{+,1}$};
	\node (A) at (1.76, 0.45) [] {$\GBasis_{-,1}$};
	
	\centerarc[black, thick,-{Latex[length=1.3mm]}](1.5,0.75)(250:197:0.35);
	
	\node (A) at (1.1, 0.33) [] {$\Theta^E$};
	\end{tikzpicture}
	\caption{Angle between frames on different elements.}
	\label{fig:angle_comp}
\end{figure}
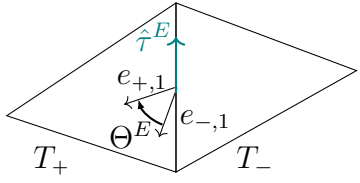

Within each element, the well-known identity
$d^1\OneForm_g = K(g) \vol T$ holds. Equivalently, using the curl in
\eqref{eq:curl_inc_group-1}, $\curl_g (\cn_g|_T) = K(g|_T)$ for each
$T \in \T$. To speak of $\curl_g \cn_g$ for the functional $\cn_g$ in
\eqref{eq:29}, we must extend $\curl_g$. Motivated by
\eqref{eq:curg-rotg-intg-by-parts-simpler}, we define
\begin{equation}
  \label{eq:30}
  \act{\curl_g \mu , u}_{\Vo(\T)} = \act{\mu, \rot_g u}_{\Wo_g(\T)},
  \qquad\text{ for all }
  u \in \Vo(\T), \;
  \mu \in \Wo_g(\T)'.  
\end{equation}
The right hand side is well defined for $\mu \in \Wo_g(\T)'$ since
$\rot_g u \in \Wo_g(\T)$ by \eqref{eq:rotgu-du}.
Next, for
each $V \in \V$ and $E \in \E_V$, let $s_{EV}$ equal $+1$ if $\gt^E$
points towards $V$ and $-1$ otherwise.  Following \cite{BKG21}, we
assume that at each interior vertex $V$, the 
``consistency'' condition
\begin{equation}
  \label{eq:top-cons}
  \sum_{E \in \E_V} s_{EV} \Theta^E(V)
  + \sum_{T \in \T_V} \agl_V^Tg = 2\pi 
\end{equation}
holds. It can be seen from the proof of \cite[Proposition~5.4]{BKG21}
that the left hand side above always equals $2 \pi m$ for some integer
$m$. The condition~\eqref{eq:top-cons} requires that $e_i$ be chosen
so as to achieve $m=1$. (We'll give a recipe for doing this shortly: see~\eqref{eq:28} below.) 

\begin{proposition}
  Let $K_g$ be as in \eqref{eq:Kg}, $\cn_g$ be as \eqref{eq:29} for a $g$-orthonormal frame $e_i$ satisfying~\eqref{eq:top-cons}, and
  $\curl_g\cn_g$ be as given by~\eqref{eq:30}. Then
  $
    \curl_g \cn_g = K_g.
  $
\end{proposition}
\begin{proof}
  This was proved in \cite{BKG21}, so we will merely indicate how to
  apply their result to
  \begin{align*}
    \act{\curl_g \cn_g, u}_{\Vo(\T)}
    & = \act{\cn_g, \rot_g u}_{\Wo_g(\T)} = \int_\T \cn_g(\rot_g u)
     + \sum_{E \in \Eint } \int_{(E, g)} \Theta^E \, g( \rot_g u, \gn^E).
  \end{align*}
  Using
  $\alpha\wedge(\star\beta)=g^{-1}(\alpha,\beta)\vol{T}$
  with  $\alpha = \cn_g$ and $\beta = d^0 u$,
  \begin{align*}
    \int_{(T, g)} \cn_g(\rot_g u) = 
    -\int_Tg^{-1}(\OneForm(g),\star d^0u_h)\vol{T}  =\int_T \OneForm(g)\wedge d^0u_h,
  \end{align*}
  and using \eqref{eq:star-rotation}, 
  \begin{align*}
    \int_{(E, g)} \Theta^E \, g( \rot_g u, \gn^E)
    &  = -\int_{(E, g)}\Theta^E\,d^0u_h(\gt^E).
  \end{align*}
  Now invoking
  \cite[Proposition 5.4]{BKG21}, the result follows.
\end{proof}

A computable representative of the connection form is obtained by
lifting the $\cn_g$ into the BDM finite element space (defined in
\eqref{eq:BDMFEspace}) on the parameter
domain, as follows.

\begin{definition}[Connection 1-form approximation]
  \label{def:distributional_one_form}
  Define $\OneForm_h(g)$ as the unique function in $\Wo_h^k$ such
  that for all $v_h\in\Wo_h^k$
  \begin{align}
    \label{eq:distr_one_form}
    \begin{split}
      \int_\om \Eucl(\OneForm_h(g),v_h)\,\da &= \act{\cn_g, Q_g v_h}_{\Wo_g(\T)}
    \end{split}
  \end{align}
  where $Q_g$ is as Proposition~\ref{prop:cov-curl-coordinates} and
  the right hand side is evaluated using~\eqref{eq:29}.
\end{definition}


In the remainder, we assume that $\cn_h(g)$ is computed using a
specific $g$-orthonormal frame $(\GBasis_1,\GBasis_2)$ satisfying \eqref{eq:top-cons}, that we describe now.
We start with a globally smooth $\delta$-orthonormal
Euclidean basis $(\EuclBasis_1, \EuclBasis_2)$ on the parameter domain
(e.g., the standard unit basis on $\R^2$).  Then, this basis is
continuously transformed to a $g$-orthonormal frame as in the next
lemma.  Let $G(t) = \delta +t(g-\delta)$.
We consider the ordinary differential equation (ODE)
\begin{align}
\label{eq:ODE_oneform}
\dot{u}(t)u(t)^{-1} = -\frac{1}{2}G(t)^{-1} (g-\Eucl),\qquad u(0)=\idop.
\end{align}

\begin{lemma}
  \label{lem:expl_sol_ode_oneform}
  The solution of ODE \eqref{eq:ODE_oneform} is given by
  $u(t) = G(t)^{-1/2}$.  The frame $(u(t) E_1, u(t) E_2)$ is
  $G(t)$-orthonormal, so at $t=1$, it is $g$-orthonormal.
\end{lemma}
\begin{proof}
  Let $g=V\Lambda V^{-1}$ be a diagonalization of $g$ with eigenvalues
  $\lambda_i$ in $\Lambda = \diag{\lambda_i}$ and eigenvectors in the
  orthogonal matrix $V$. Then $G(t) = V ((1-t)\Eucl+t\Lambda) V^{-1}$
  and
  $G(t)^{-1/2} = V \mathrm{diag} \big(1+ t(\lambda_i -1)\big)^{-1/2}
  V^{-1}$. Using these expressions, the statements of the lemma can
  be easily verified.
\end{proof}

We use the $g$-orthonormal frame
\begin{equation}
  \label{eq:28}
  e_i = u(1) E_i = g^{-1/2} E_i  
\end{equation}
for computations. As stated in \cite{BKG21}, 
since the frame $E_i$ obviously satisfies~\eqref{eq:top-cons} with $g = \delta$,
the chosen $e_i$ obtained by the continuous deformation of the metric and frame, 
satisfies~\eqref{eq:top-cons}.
In Appendix~\ref{sec:oneform_algorithm} we present a stable algorithm
for computing the right hand side of
\eqref{eq:distr_one_form} using the discontinuous metric~$g$ in the
computational coordinates.

\section{Error analysis}
\label{sec:num_ana}

In this section, we prove {\it a priori} estimates for the error in
the previously defined approximations.  We restrict ourselves to the
model case (see~\S\ref{ssec:model}) and work on the parameter domain
$\om$, where we shall use the Euclidean dot product $u \cdot v = \delta(u, v)$
and the standard Frobenius inner product $A: B$ between matrices $A, B$.
We assume that 
the
triangulation $\T$ consists of affine-equivalent elements, is
shape-regular, and is quasi-uniform of meshsize
$h:=\max_{T \in \T}\mathrm{diam}(T)$.

\subsection{Convergence results}
\label{subsec:statement_theorem}

All results here concern the canonical interpolant into the Regge
space $\Regge_h^k$ (defined in \eqref{eq:ReggeFE}) of degree
$k \ge 0$.  This well-known interpolant~\cite{li18}, denoted by
$\RegInt[k]:\Czero[\Omega,\Sc]\to\Regge_h^k$, satisfies following
equations
\begin{subequations}
	\label{eq:RegInt}  
	\begin{align}  
          \int_E(\RegInt[k]\sigma)_{\tv\tv}\,q\,\dl
          & = \int_E\sigma_{\tv\tv}\,q\,\dl
          && \text{ for all } q\in \Pol^{k}(E)
             \text{ and edges $E$ of $\d T$},\label{eq:RefInt_edge}
          \\
          \int_T\RegInt[k]\sigma:\rho\,\da
          & = \int_T\sigma:\rho\,\da
          && \text{ for all } \rho\in \Pol^{k-1}(T,\R^{2 \times 2}).
             \label{eq:RefInt_trig}
	\end{align}
\end{subequations}
Note that when $\rho$ is a skew-symmetric matrix, both sides of~\eqref{eq:RefInt_trig} vanish, so~\eqref{eq:RefInt_trig} is nontrivial only for symmetric $\rho$. Throughout, we use standard
Sobolev spaces $\Wsp[\Omega]$ and their norms and seminorms for any $s\geq 0$ and
$p\in [1,\infty].$ When the domain is $\Omega$, we omit it from the
norm notation if there is no chance of confusion.  We also use the
element-wise norms $\|u\|_{\Wsph}^p =\sum_{T\in \T}\|u\|^p_{\Wsp[T]},$
with the usual adaption for $p=\infty$. When $p=2$, we put
$\|\cdot\|_{\Hsh}=\|\cdot\|_{W^{s,2}_h}$.  Let
\[
  \nrm{\sigma}_2 = \| \sigma \|_{L^2} + h \| \sigma \|_{H_h^1}, \qquad
  \nrm{\sigma}_\infty = \| \sigma \|_{L^\infty} + h \| \sigma
\|_{W^{1,\infty}_h}.
\]
Most of our results assume that  $k$ is an integer,
\begin{equation}
  \label{eq:13}
  k\ge 0, \quad g = \RegInt[k] \gex,
  \quad
  \gex \in W^{1,\infty}(\om, \Sc^+), 
  \quad {\gex\,}^{-1} \in L^\infty(\om, \Sc^+).    
\end{equation}
We use $a \lesssim b$ to indicate that there is
an $h$-independent generic constant $C$, depending on $\om$ and the
shape-regularity of the mesh $\T$, such that $a \le C b$.
The $C$ may additionally depend either on
$\{\|\gex\|_{\Winf},\|\gex^{-1}\|_{\Linf},
\cn(\gex), K(\gex)\}$, or on 
$\{\|g\|_{\Winfh},\|g^{-1}\|_{\Linf}\}$, 
depending on whether we assume~\eqref{eq:13} or not, respectively.
We use $(\cdot, \cdot)_D$ to denote the integral of the appropriate
inner product (scalar, dot product, Frobenius product, etc.) of its
arguments over a Euclidean measurable set $D$, e.g.,
$(\sigma_{\tv\tv}, q)_E$ and $(\sigma, \rho)_T$ equal the right hand
sides of~\eqref{eq:RefInt_edge} and \eqref{eq:RefInt_trig},
respectively.

\begin{theorem}[Approximation of covariant curl]
  \label{thm:curlg}
  Suppose $g \in \Regge^+(\T)$, $\sigmaappr\in \Regge_h^k$,
  $v_h\in \Wo_h^k$ (the BDM space in~\eqref{eq:BDMFEspace}),
  $\sigma \in H^1(\om, \Sc) \cap C^0(\om, \Sc)$, and let
  $\curl_{g, h}$ be as defined in \eqref{eq:cov-curl-lift}.  Then,
  there exists an $h_0>0$ such that for all $h < h_0$,
  \begin{align}
    \label{eq:curlg-1}
    (\curl_{g, h}(\sigma- \RegInt[k] \sigma), v_h)_\om
    &\, \lesssim\, \nrm{\smash{\sigma-\RegInt[k] \sigma}}_2\|v_h\|_{\Ltwo},
  \intertext{and if~\eqref{eq:13} holds,}
    \label{eq:curlg-2}
    (\curl_{\gex, h} {\sigmaappr} -  \curl_{g, h} {\sigmaappr}, v_h)_\om
    & \,\lesssim\,
      \nrm{\gex-{\gh}}_\infty\|\sigmaappr\|_{\Honeh}\|v_h\|_{\Ltwo},
    \\ \label{eq:curlg-3}
    (\curl_{\gex, h} {\sigmaappr} -  \curl_{g, h} {\sigmaappr}, v_h)_\om
    & \,\lesssim\,
      \nrm{\gex-{\gh}}_2\|\sigmaappr\|_{W_h^{1,\infty}}\|v_h\|_{\Ltwo}.
  \end{align}
\end{theorem}

Proofs of this and other theorems in this subsection are presented in
later subsections. For now, let us note that on Euclidean manifolds
with  $\gex = \delta$, the expressions of our distributional
covariant curl (either \eqref{eq:curlg-extended-2} or \eqref{eq:cov_distr_curl1}) reduce to
\[
  (\curl_{\delta, h} \sigma, v_h)_\om =  (\curl \mt\sigma, v_h)_\om
  - \sum_{T \in \T}  (\sigma_{\tv\nv},  v_h\cdot \nv)_{\d T}.
\]
It is easy to see from~\eqref{eq:RegInt} (and integrating the right hand side above by parts) that
\begin{equation}
  \label{eq:17}
  (\curl_{\delta, h}(\sigma-\RegInt[k]\sigma),v_h)_\om = 0  
\end{equation}
for all $v_h \in \Wo_h^k$. This equality has the flavor of typical FEEC
identities (also known as commuting diagram properties).  On general
manifolds however, it appears that we must trade this equality  
for the inequality~\eqref{eq:curlg-1}. The remaining
inequalities~\eqref{eq:curlg-2}--\eqref{eq:curlg-3} bound the
nonlinear changes in the covariant operator arising due to the
perturbations in the metric.  Theorem~\ref{thm:curlg} directly implies 
error bounds in $L^2$ norm, while error bounds in stronger norms follow from it:

\begin{corollary}
	\label{cor:curlg}
	Under  the assumptions of Theorem~\ref{thm:curlg}, for all $1\leq l\leq k$,
	\begin{subequations}
	\begin{align*}
		\|\curl_{g, h}(\sigma- \RegInt[k] \sigma)\|_{H^l_h}&\,\lesssim\,h^{-l}\big(\nrm{\smash{\sigma-\RegInt[k] \sigma}}_2+h^{k+1}|\curl_{g, h}(\sigma)|_{H^{k+1}}\big),\\
		\|\curl_{\gex, h} {\sigmaappr} -  \curl_{g, h} {\sigmaappr}\|_{H^l_h}
		& \,\lesssim\,h^{-l}
		\big(\nrm{\gex-{\gh}}_\infty\|\sigmaappr\|_{\Honeh}+h^{k+1}|\curl_{\gex, h} {\sigmaappr}|_{H^{k+1}_h}\big),\\
		\|\curl_{\gex, h} {\sigmaappr} -  \curl_{g, h} {\sigmaappr}\|_{H^l_h}
		& \,\lesssim\,h^{-l}
		\big(\nrm{\gex-{\gh}}_2\|\sigmaappr\|_{W_h^{1,\infty}}+h^{k+1}|\curl_{\gex, h} {\sigmaappr}|_{H^{k+1}_h}\big).
	\end{align*}
	\end{subequations}
\end{corollary}

Similar results can be proved for the incompatibility operator.

\begin{theorem}[Approximation of covariant incompatibility operator]
  \label{thm:incg}
  Suppose $g \in \Regge^+(\T)$,  $\sigmaappr\in \Regge_h^k$,
  $u_h\in \Vo_h^{k+1}$ (the Lagrange space in~\eqref{eq:LagrangeFEspace}), 
  $\sigma \in H^1(\om, \Sc) \cap C^0(\om, \Sc)$,
  and let $\inc_{g, h}$ be as defined in \eqref{eq:incg-h}.
  Then, there exists an $h_0>0$ such that for all $h < h_0$,
  \begin{align}
    \label{eq:incg-1}
    (\inc_{g, h}(\sigma- \RegInt[k] \sigma), u_h)_\om
    &\, \lesssim\, \nrm{\smash{\sigma-\RegInt[k] \sigma}}_2| u_h|_{\Hone},
      \intertext{and if~\eqref{eq:13} holds,}
      \label{eq:incg-2}
    (\inc_{\gex, h} {\sigmaappr} -  \inc_{g, h} {\sigmaappr}, u_h)_\om
    & \,\lesssim\,
      \nrm{\gex-{\gh}}_\infty\|\sigmaappr\|_{\Honeh}|u_h|_{\Hone},
    \\ 
    (\inc_{\gex, h} {\sigmaappr} -  \inc_{g, h} {\sigmaappr}, u_h)_\om
    & \,\lesssim\,
      \nrm{\gex-{\gh}}_2\|\sigmaappr\|_{W_h^{1,\infty}}|u_h|_{\Hone}.
  \end{align}
\end{theorem}

Here again, as in the case of covariant curl, comparison with the
Euclidean case is illuminating.
On Euclidean manifolds,
instead of \eqref{eq:incg-1}, the stronger result
\begin{equation}
  \label{eq:19}
  (\inc_{\delta, h}(\sigma- \RegInt[k] \sigma), u_h)_\om = 0
\end{equation}
holds for the distributional incompatibility (which has element, edge, and
vertex contributions: see Proposition~\ref{prop:cov_distr_inc}).
Indeed,~\eqref{eq:19} follows immediately from~\eqref{eq:17} and
\eqref{eq:def_distr_cov_inc}. The theorem also implies 
error bounds in stronger norms.

\begin{corollary}
	\label{cor:incg}
	Under the assumptions of Theorem~\ref{thm:incg}, for all $0\leq l\leq k$,
		\begin{align*}
			\|\inc_{g, h}(\sigma- \RegInt[k] \sigma)\|_{H^l_h}&\,\lesssim\,h^{-l-1}\big(\nrm{\smash{\sigma-\RegInt[k] \sigma}}_2+h^{k+1}|\inc_{g, h}(\sigma)|_{H^k}\big),\\
			\|\inc_{\gex, h} {\sigmaappr} -  \inc_{g, h} {\sigmaappr}\|_{H^l_h}
			& \,\lesssim\,h^{-l-1}\big(
			\nrm{\gex-{\gh}}_\infty\|\sigmaappr\|_{\Honeh}+h^{k+1}|\inc_{\gex, h} {\sigmaappr}|_ {H^k_h}\big),\\
			\|\inc_{\gex, h} {\sigmaappr} -  \inc_{g, h} {\sigmaappr}\|_{H^l_h}
			& \,\lesssim\,h^{-l-1}\big(
			\nrm{\gex-{\gh}}_2\|\sigmaappr\|_{W_h^{1,\infty}}+h^{k+1}|\inc_{\gex, h} {\sigmaappr}|_ {H^k_h}\big).
		\end{align*}
\end{corollary}

Our remaining results are for the
approximations of connection and curvature. Let $\idop$ denote the
identity operator (on some space that will be %
obvious from context) and let $\Pi^\WW_{k}$ and $\Pi^\VV_{k+1}$ denote
the $L^2$-orthogonal projection into $\Wo_h^k$ and $\Vo^{k+1}_h$,
respectively.

\begin{theorem}[Approximation of Gauss curvature]
  \label{thm:convergence_curvature_hm1-rev}
  Suppose~\eqref{eq:13} holds, $\gex\in W^{k+1,\infty}(\Omega)$,
  $\Gauss(\gex)\in H^k(\Omega)$, and
  $\Gauss_h(\gh)\in \Vo_{h}^{k+1}$ be as in \eqref{eq:1}.  Then,
  there exists an $h_0>0$ such that for all $h < h_0$,
  \begin{align*}
      \|\Gauss_h(\gh)-\Gauss(\gex)\|_{H^{-1}}
      &\,\lesssim\,
      \nrm{\gex - \gh}_\infty
      +
      h \left\| (\idop - \Pi_{k+1}^\VV) \Gauss(\gex)
      \right\|_{\Ltwo} \\
      &\,\lesssim\,
      h^{k+1}(\|\gex\|_{W^{k+1,\infty}}+|\Gauss(\gex)|_{H^{k}}).
  \end{align*} 
\end{theorem}

\begin{corollary}
	\label{cor:convergence_curvature_l2_hl}
	Under  the assumptions of Theorem~\ref{thm:convergence_curvature_hm1-rev}, for all $0\leq l\leq k$,
	\begin{align*}
		|\Gauss_h(\gh)-\Gauss(\gex)|_{H^l_h}&\,\lesssim\,h^{-l}\, \big(h^{-1}\nrm{\gex-\gh}_{\infty} + \|(\idop -     \Pi^\VV_{k+1})\Gauss(\gex)\|_{\Ltwo}+h^{k}|\Gauss(\gex)|_{H^k}\big).
	\end{align*}
\end{corollary}

\begin{theorem}[Approximation of Levi-Civita connection]
  \label{thm:convergence_connectionform_l2-rev}
  Suppose~\eqref{eq:13} holds, 
  $\gex \in H^{k+1}(\Omega)$, $\OneForm(\gex)\in H^{k+1}(\Omega)$, and
  let $\OneForm_h(\gh)\in\Wo_{h}^k$ be as in \eqref{eq:distr_one_form}.
  Then, there
  exists an $h_0>0$ such that for all $h < h_0$,
  \begin{align}\label{eq:oneform_estimate1}
    \|\OneForm_h(\gh)-\OneForm(\gex)\|_{\Ltwo}
    & \,\lesssim\, \nrm{\gex - \gh}_2
      +
      \left\|(\idop- \Pi_{k}^\WW) \cn(\gex)\right\|_{\Ltwo},
  \intertext{and when  $k \ge 1$,}
    \|\OneForm_h(\gh)-\OneForm(\gex)\|_{\Ltwo}
    &\,\lesssim\,
      h^{k+1}(\|\gex\|_{H^{k+1}}+|\OneForm(\gex)|_{H^{k+1}})\label{eq:oneform_estimate2}.
  \end{align}
\end{theorem}

\begin{corollary}
	\label{cor:convergence_connectionform_l2_hl}
	Under  the assumptions of Theorem~\ref{thm:convergence_connectionform_l2-rev}, for all $1\leq l\leq k$,
	\begin{align*}
		\|\OneForm_h(\gh)-\OneForm(\gex)\|_{H^l_h}&\,\lesssim\,h^{-l}\big(\nrm{\gex - \gh}_2
		+
		\left\|(\idop- \Pi_{k}^\WW) \cn(\gex)\right\|_{\Ltwo}+h^{k+1}|\cn(\gex)|_{H^{k+1}}\big).
	\end{align*}
\end{corollary}

Since the curvature $K(\gex)$ has second order derivatives of the
metric $\gex$, at first glance it may seem surprising that
Theorem~\ref{thm:convergence_curvature_hm1-rev} gives
$H^{-1}$-convergence of curvature approximations at the same rate as
$\nrm{\gex-g}_\infty$. Even for the lowest order case $k=0$ (while
using piecewise constant metric approximations), where one might only
expect convergence in the $H^{-2}$-norm, the theorem gives first order
convergence of the curvature in the $H^{-1}$-norm. The convergence
rates of  Theorems~\ref{thm:convergence_connectionform_l2-rev}
and~\ref{thm:convergence_curvature_hm1-rev} are both higher than those
proved in \cite{Gaw20,BKG21}.  As we shall see, the reason behind
these higher rates is a  property of the Regge
interpolant proved in the next subsection.

\subsection{Distributional Christoffel symbols of the first kind}

In a neighborhood where the metric $g$ is smooth, the Christoffel
symbols of the first kind, $\Gamma_{lmn}(g)$, are given by
\eqref{eq:Chris12-coords}. To see what further terms must be supplied
to obtain their distributional version when the metric $g$ is only
$tt$-continuous across an element interface, consider a $\psi^{lmn}$
in the Schwartz test space $\DD(\om)$ of smooth compactly supported
functions. Since $\Gamma_{lmn}(\cdot)$ is a linear first order
differential operator applied to a smooth metric argument, 
its distributional definition is
standard:
\begin{equation}
  \label{eq:14-gamma}
  \act{ \Gamma_{lmn}(g),\psi^{lmn}}_{\DD(\om)}
  =- \frac 1 2
  \int_\om \left( g_{mn} \d_l\psi^{lmn} + g_{nl} \d_m \psi^{lmn}
  - g_{lm} \d_n \psi^{lmn} \right)\,\da.
\end{equation}
Let $\psi^{l\nv n} = \psi^{lqn} \delta_{jq}\nv^q$ and define
$\psi^{\nv mn}, \psi^{\nv\nv n}, \psi^{\nv\nv \tv}$ etc.\ similarly.
For any $T \in \T$ and any smooth $\psi$ on $\bar T$,  let 
\begin{align}
  \label{eq:distr_christoffel_regge2-rev}
  \Gamma_T(g, \psi)
  & := \big(\Gamma_{lmn}(g), \psi^{lmn}\big)_T
    +\frac{1}{2}
    \big(g_{\nv\nv}, \psi^{\nv\nv \nv} \big)_{\d T}
    +  \big(g_{\nv\tv}, \psi^{\nv\nv \tv} \big)_{\d T}.
\end{align}

\begin{prop}
  \label{prop:distr-christ-symb}
  For all $\psi \in \DD(\om)$, the distributional Christoffel symbols
  satisfy
  \[
    \act{ \Gamma_{lmn}(g),\psi^{lmn}}_{\DD(\om)} = \sum_{T\in \T}
    \Gamma_T(g, \psi).
    \]
\end{prop}
\begin{proof}
  Integrating~\eqref{eq:14-gamma}
  by parts, element by element,
  \begin{align*}
    \act{ \Gamma_{lmn}(g),\psi^{lmn}}_{\DD(\om)}
    &=-\sum_{T\in \T}\int_T \frac{1}{2}g_{lm}(\pder{\psi^{nlm}}{n}+\pder{\psi^{lnm}}{n}-\pder{\psi^{lmn}}{n})
      \,\da
    \\
    &=\sum_{T\in \T} \left(\int_T \Gamma_{lmn}(g)\psi^{lmn}\,\da + \int_{\partial T}g_{lm}\frac{1}{2}(\psi^{\nv lm}+\psi^{l\nv m}-\psi^{lm \nv})\,
      \dl \right).
  \end{align*}
  We can split the integrand over $\d T$ into $\nv\nv$, $\tv\nv$,
  $\nv\tv$, and $\tv\tv$ components. When summing over $\d T$ for all
  $T\in \T$, the $tt$-continuity of $g$ implies that the
  $\tv\tv$-terms cancel out. The remaining terms give the boundary
  contribution as
  $\sum_{T \in \T} \frac{1}{2} \int_{\d T} (g_{\nv\nv} \psi^{\nv\nv
    \nv} + 2 g_{\nv\tv} \psi^{\nv\nv \tv} )\, \dl,$ so the result
  follows.
\end{proof}

Proposition~\ref{prop:distr-christ-symb} serves to motivate the introduction of
\[
  \Gamma(g, \Sigma) := \sum_{T \in \T} \Gamma_T(g, \Sigma),
\]
for piecewise smooth $g \in \Regge^+(\T)$ and
$\Sigma \in C^\infty(\T, \R^{2\times 2 \times 2})$. As we proceed to
analyze distributional covariant operators, it is perhaps not a
surprise that this quantity will reappear in our analysis with various 
arguments $\Sigma$, including those in
$\Pol^k(\T, \R^{2\times 2 \times 2}) = \{ \Sigma: \Sigma|_T \in
\Pol^k(T, \R^{2\times 2 \times 2})$ for all $T \in \T\}$. The next
result gives a property of $\Gamma(\cdot, \cdot)$ in connection with
the Regge interpolation error.

\begin{lemma}
  \label{lem:christ-Regge}
  If $k,\gex, g$ are as in~\eqref{eq:13}, then for any
  $\Sigma_h\in \Pol^{k}(\T,\R^{2\times2\times 2})$,
  \begin{align}
    \Gamma(\gex - \gh,\Sigma_h) = 0.
    \label{eq:biortho_regge_christoffel-rev}
  \end{align}
\end{lemma}
\begin{proof}
  We start with \eqref{eq:distr_christoffel_regge2-rev} on $T\in\T$
  and integrate by parts
  \begin{align*}
    \Gamma_T& (\gex-\gh,\Sigma_h)
      = (\Gamma_{ijl}(\gex-\gh), \Sigma^{ijl}_h)_T
      + 
      (\Sigma^{\nv\nv i}_h, (\gex-\gh)_{\nv\tv}\tv_i
      +\frac{1}{2}(\gex-\gh)_{\nv\nv}\nv_i )_{\d T}
    \\
    &=-\frac 1 2 \big(      (\gex-\gh)_{ij},
      \d_l (\Sigma_h^{lij}
      +\Sigma_h^{ilj}
      -\Sigma_h^{ijl})\big)_T
      -\frac 1 2 \big( (\gex-\gh)_{\tv\tv}, 
      \Sigma_h^{\nv\tv\tv}+\Sigma_h^{\tv\nv\tv}-\Sigma_h^{\tv\tv\nv}\big)_{\d T}.
  \end{align*}
  The first and second inner products above vanish
  by~\eqref{eq:RefInt_trig} and~\eqref{eq:RefInt_edge}, respectively.
\end{proof}

\subsection{Basic estimates}

We need a number of preliminary estimates to proceed with the
analysis.  The approximation properties of the Regge elements are well
understood.  By the Bramble-Hilbert lemma, on any $T \in \T$,
\begin{subequations}
\begin{align}
  &\|(\idop-\RegInt[0])\gex\|_{\Lp[T]}\lesssim h |\gex|_{W^{1,p}(T)}
    \label{eq:reg_appr_lo_vol-rev}
\end{align}
for $\gex\in W^{1,p}(T,\Sc)$ and $p\in [1,\infty]$,
for the lowest order case (and certainly for the higher $k \ge 1$).  A similar estimate
holds for the element-wise $L^2(T)$- projection into the space of
constants, which we denote by $\Pi_0$:  
\begin{align}
  \label{eq:l2_appr_lo_vol-rev} 
  &\|(\idop-\Pi_0)f\|_{\Lp[T]}\leq h C|f|_{W^{1,p}(T)}
  \end{align}
\end{subequations}
for $ f\in W^{1,p}(T)$.  Since $\gh = \RegInt[k] \gex$ approaches
$\gex$ as $h \to 0$, we will tacitly assume throughout that $h$ has
become sufficiently small to guarantee that the approximated metric
$\gh$ is positive definite throughout.  Let $E\subset \d T$ be an edge
of $T$. We also need the following well-known estimates that follow
from scaling arguments: for all $u\in \Hone[T]$
\begin{align}
  \|u\|_{\Ltwo[E]}^2\,\lesssim\, 
  h^{-1}\|u\|_{\Ltwo[T]}^2 +h\|\nabla u\|_{\Ltwo[T]}^2
  \label{eq:trace_inequ-rev}
\end{align}
and for all $u\in \Pol^k(T)$,
\begin{align}
  &\|u\|_{\Ltwo[E]}\,
    \lesssim h^{-1/2}\|u\|_{\Ltwo[T]},\qquad \|\nabla u\|_{\Ltwo[T]}\,\lesssim\,
    h^{-1}\|u\|_{\Ltwo[T]}.
    \label{eq:discr_inverse_inequ-rev}
\end{align}

Staying in the setting of~\eqref{eq:13}, the following estimates are a
consequence of \cite[Lemma 4.5 and Lemma 4.6]{Gaw20}:
for $p\in [1,\infty]$,
\begin{gather}
  \|\gh^{-1}-\gex^{-1}\|_{\Lp}
   \,\lesssim \,
    \|\gh-\gex\|_{\Lp},\label{eq:est_inv_by_ten-rev}
  \\
  \|{\gh}\|_{\Winfh}+\|{\gh}^{-1}\|_{\Linf}
   \,\lesssim\, 1,  \label{eq:est_gh_by_g-rev}
  \\
  \|\sqrt{\det {\gh}}\|_{\Linf}+\|\sqrt{\det {\gh^{-1}}}\|_{\Linf}
   \,\lesssim\, 1,
    \label{eq:est_detgh_by_g-rev}
\end{gather}
and for all $x$ in the interior of any element $T\in \T$ and for all
$u \in \R^2,$
\begin{equation*}
  u^\trans u\,\lesssim\, u^\trans {\gh}(x)u
  \,\lesssim\, u^\trans u.
\end{equation*}
Moreover,
\begin{align}
  \|\sqrt{\det \gex}-\sqrt{\det {\gh}}\|_{W_h^{l,p}(\Omega)}
  \,\lesssim\,\|{\gh}-\gex\|_{W_h^{l,p}}, \quad l = 0, 1.
  \label{eq:est_det_by_ten-rev}
\end{align}
Better control of differences of some functions of the  metric
is  possible through the next lemma. Let
\begin{gather}
  \label{eq:14}
  \beta_1(g) = \frac{1}{\sqrt{\det g}}, \qquad
  \beta_2(g) = \frac{g_{\nv\tv}}{g_{\tv\tv}}\beta_1(g), 
  \\ \nonumber
  \eta_1(\gex, g) =
  \frac{2{\gh}_{\nv\tv}({\gex}-{\gh})_{\nv\tv}
      -{\gh}_{\tv\tv}({\gex}-{\gh})_{\nv\nv}}{2\sqrt{\det {\gh}}^3},\qquad\eta_2(\gex, g) =
  \frac{2{\gh}_{\nv\nv}({\gex}-{\gh})_{\nv\tv}
  -{\gh}_{\nv\tv}({\gex}-{\gh})_{\nv\nv}}{2\sqrt{\det {\gh}}^3}.  
\end{gather}

\begin{lemma}
  \label{lem:prep_bnd_terms-rev}  
  In the setting of~\eqref{eq:13}, for sufficiently small $h$,
  there exist smooth functions $f_1,f_2\in\Cinf[\Sc^+, \R]$ such that
  at each point on $\d T,$
  \begin{align*}
    \beta_1(\gex) - \beta_1(\gh)
    &=({\gex}-{\gh})_{\tv\tv}f_1({\gh})
      + \eta_1(\gex, g)
      + \epsilon_1^2,      
    \\
    \beta_2(\gex) - \beta_2(\gh)
    &=({\gex}-{\gh})_{\tv\tv}f_2({\gh})
      + \eta_2(\gex, \gh)
      + \epsilon_2^2,
  \end{align*}
  where $\max(\epsilon_1, \epsilon_2)=O(\|{\gex}-{\gh}\|)$ for some
  Euclidean norm $\| \cdot \|$ at the point.
\end{lemma}
\begin{proof}
  By Taylor expansion and \eqref{eq:est_gh_by_g-rev},
  \begin{align*}
    \beta_1(\gex) - \beta_1(\gh)
    & =
                                -\frac{1}{2\sqrt{\det {\gh}}^3}\cof{{\gh}}:({\gex}-{\gh}) + O(\epsilon^2)
    \\
                              &=- \frac{{\gh}_{\nv\nv}({\gex}-{\gh})_{\tv\tv} - 2{\gh}_{\nv\tv}({\gex}-{\gh})_{\nv\tv} + {\gh}_{\tv\tv}({\gex}-{\gh})_{\nv\nv}}{2\sqrt{\det {\gh}}^3} + O(\epsilon^2)\\
                              &= ({\gex}-{\gh})_{\tv\tv}\frac{-{\gh}_{\nv\nv}}{2\sqrt{\det {\gh}}^3} + \frac{2{\gh}_{\nv\tv}({\gex}-{\gh})_{\nv\tv}-{\gh}_{\tv\tv}({\gex}-{\gh})_{\nv\nv}}{2\sqrt{\det {\gh}}^3}+ O(\epsilon^2),
  \end{align*}
  where $\cof{\cdot}$ denotes the cofactor matrix. Putting
  $f_1({\gh}):=-\frac 1 2 {\gh}_{\nv\nv}/\sqrt{\det
    {\gh}}^3\in\Cinf[\Sc^+]$,  the first identity is proved.

  For the second identity, again starting with Taylor expansion and
  \eqref{eq:est_gh_by_g-rev},
  \begin{align*}
    \beta_2(\gex) - \beta_2(\gh)=
    \frac{-({\gex}-{\gh})_{\tv\tv}}{\sqrt{\det {\gh}}{\gh}_{\tv\tv}^2}+&\frac{2\det {\gh}({\gex}-{\gh})_{\nv\tv}-{\gh}_{\nv\tv}\cof{{\gh}}:({\gex}-{\gh})}{2\sqrt{\det {\gh}}^3{\gh}_{\tv\tv}}     + O(\epsilon^2).
  \end{align*}
  A simple calculation reveals
  \begin{align*}
    2\det {\gh}({\gex}-{\gh})_{\nv\tv}-{\gh}_{\nv\tv}\cof{{\gh}}:({\gex}-{\gh}) 
    &={\gh}_{\tv\tv}(2{\gh}_{\nv\nv}({\gex}-{\gh})_{\nv\tv}-{\gh}_{\nv\tv}({\gex}-{\gh})_{\nv\nv}) \\
    &- {\gh}_{\nv\tv}{\gh}_{\nv\nv}({\gex}-{\gh})_{\tv\tv},
  \end{align*}
  Thus the second identity follows after making an obvious
  choice of  $f_2$.
\end{proof}

\begin{lemma}
  \label{lem:est_bnd_Psi}
  Let $\gh = \RegInt[k] \gex$ for some $k \ge 0$,
  $T\in\T$, $q\in\Pol^k(T),$ and let $E\subset \d T$ be an edge of
  $T$. If $\gex\in \Hone[T,\Sc]\cap C^0(T,\Sc)$ and $\Psi\in\Winf[T],$
  \begin{align}
    &(({\gh}-\gex)_{\tv\tv}, \Psi\, q)_E
      \,\lesssim \, \big(\|{\gh}-{\gex}\|_{\Ltwo[T]}+h|{\gh}-{\gex}|_{\Hone[T]}\big)\|\Psi\|_{\Winf[T]}\|q\|_{\Ltwo[T]}.
      \label{eq:est_bnd_lemma1-rev}
  \end{align}
  If instead, $\gex\in C^0(T,\Sc)$ and
  $\Psi\in\Htwo[T]$, then 
  \begin{align}
    (({\gh}-\gex)_{\tv\tv}, \Psi\, q)_E    
      \,\lesssim\, \|{\gh}-{\gex}\|_{\Linf[T]}
      \big(\|\Psi\|_{\Hone[T]}+h|\Psi|_{H^2(T)}\big)\|q\|_{\Ltwo[T]}.
      \label{eq:est_bnd_lemma2-rev}
  \end{align}
\end{lemma}
\begin{proof}
  Using \eqref{eq:RefInt_edge} with $q\, \Pi_0\Psi\in \Pol^k(E)$,
  \begin{align*}
    (({\gh}-\gex)_{\tv\tv}, q \Psi)_E
    & =
      (({\gh}-\gex)_{\tv\tv}, q (\idop - \Pi_0)\Psi)_E.
  \end{align*}
  Now, by H\"older inequality, the trace inequality
  \eqref{eq:trace_inequ-rev}, triangle inequality, and
  \eqref{eq:l2_appr_lo_vol-rev}, 
  \begin{align*}
    (({\gh}-\gex)_{\tv\tv}, q \Psi)_E
    & \lesssim
      \|({\gh}-{\gex})_{\tv\tv}\|_{\Ltwo[E]}\|q\|_{\Ltwo[E]}\|(\idop-\Pi_0)\Psi\|_{\Linf[E]}
    \\
    &\lesssim
      h^{-1} (\|{\gh}-{\gex}\|^2_{\Ltwo[T]}+h^2|{\gh}-{\gex}|^2_{\Hone[T]})^{1/2}\|q\|_{\Ltwo[T]}\|(\idop-\Pi_0)\Psi\|_{\Linf[T]}
    \\
    &\lesssim \big(\|{\gh}-{\gex}\|_{\Ltwo[T]}+h|{\gh}-{\gex}|_{\Hone[T]}\big)\|q\|_{\Ltwo[T]}\|\Psi\|_{\Winf[T]}.
  \end{align*}
  The proof of
  \eqref{eq:est_bnd_lemma2-rev} is similar.
\end{proof}

\subsection{Analysis of the covariant curl approximation}

This subsection is devoted to proving Theorem~\ref{thm:curlg}.  We
will start with the first estimate of the theorem, which is easier to
prove. The remaining inequalities will be proved using
Lemma~\ref{lem:christ-Regge} afterward.

\begin{lemma}
  \label{lem:first_lemma-rev}
  Suppose $g \in \Regge^+(\T)$,
  $\sigma \in H^1(\om, \Sc) \cap C^0(\om, \Sc)$,
  $\sigmaappr = \RegInt[k] \sigma$, and
  $v_h\in \Wo_h^k.$ Then, 
  \[
    (\curl_{g, h}(\sigma- \sigma_h), v_h)_\om
    \, \lesssim\, \nrm{\smash{\sigma-\sigma_h}}_2\|v_h\|_{\Ltwo}.
  \]
\end{lemma}
\begin{proof}
  Using  \eqref{eq:cov-curl-lift} and \eqref{eq:cov_distr_curl2}, 
  \begin{align}
    \nonumber 
    (\curl_{g, h}(\sigma- \sigma_h), v_h)_\om
    = \sum_{T\in\T}
      \Big[
    & 
      \big((\sigma-\sigmaappr)_{ij},
      \;\mt{\rot \mt {v_h}}^{ij} \,\beta_1(g)\big)_T
    \\  \label{eq:25}
    &
      - \big((\sigma-\sigmaappr)_{ij},\,
      \veps^{jk}(\Gamma_{lk}^l v_h^i-\Gamma_{lk}^iv_h^l)
      \,\beta_1(g)\big)_T
    \\ \nonumber 
    &
      -\big( (\sigma-\sigmaappr)_{\tv\tv}, v_h^i g_{i\tv}
      \,\beta_1(g)/g_{\tv\tv}\big)_{\d T} \Big]
  \end{align}
  where $\beta_1$ is as in~\eqref{eq:14}. The first term on the right
  is zero when $k=0$. When $k\ge 1$, we use \eqref{eq:RefInt_trig} to insert
  a projection to constants:
  \begin{align*}
    \big((\sigma-\sigmaappr)_{ij},
    \;  \mt{\rot \mt {v_h}}^{ij}\,\beta_1(g)\big)_T
    & = \big(
      (\sigma-\sigmaappr)_{ij},
      (\idop - \Pi_0 )(\mt{\rot \mt {v_h}}^{ij}\beta_1(g))\,\big)_{T}
    \\
    & \lesssim\,
      \| \sigma-\sigmaappr \|_{L^2(T)}  h
      \| \rot \mt {v_h}\|_{H^1(T)}
    \\
    & \lesssim\,
      \| \sigma-\sigmaappr \|_{L^2(T)}  
       \|v_h \|_{L^2(T)},
  \end{align*}
  where we also used \eqref{eq:l2_appr_lo_vol-rev} and the inverse
  estimate \eqref{eq:discr_inverse_inequ-rev}. The second term in~\eqref{eq:25} is
  easily bounded by absorbing the maximum of $g$-dependent $\Gamma_{ij}^k$
  into a generic constant:
  \begin{align*}
    \big((\sigma-\sigmaappr)_{ij},\,
      \veps^{jk}(\Gamma_{lk}^l v_h^i-\Gamma_{lk}^iv_h^l)
      \,\beta_1(g)\big)_T
    \,\lesssim\,
    \| \sigma-\sigmaappr \|_{L^2(T)}  \| v_h\|_{L^2(T)}.
  \end{align*}

  Finally, for the last (boundary) term in~\eqref{eq:25}, we use
  \eqref{eq:est_bnd_lemma1-rev} of Lemma~\ref{lem:est_bnd_Psi} for
  each edge $E\subset \d T$, setting
  $\Psi = \beta_1(g) g_{i\tv} /g_{\tv\tv}\in \Winf[T]$, after
  extending the constant tangent vector $\tv$ from $E$ into the
  element $T$. Then
  \begin{align*}
    \big( (\sigma-\sigmaappr)_{\tv\tv}, v_h^i g_{i\tv}
    \,\beta_1(g)/g_{\tv\tv}\big)_{\d T}
    \,\lesssim\,
    (\|\sigma-\sigmaappr\|_{\Ltwo[T]}
    +h|\sigma-\sigmaappr|_{\Honeh[ T]})\|v_h\|_{\Ltwo[T]},
  \end{align*}
  where we have absorbed the norm $\|\Psi\|_{\Winf[T]}$ from the lemma
  into the generic $g$-dependent constant in the inequality. Thus
  \begin{align*}
    (\curl_{g, h}(\sigma- \sigma_h), v_h)_\om
    \lesssim
    \sum_{T\in\T}
    \big(
    \|\sigma-\sigmaappr\|_{\Ltwo[T]}
    +h|\sigma-\sigmaappr|_{\Honeh[ T]}
    \big)
    \|v_h\|_{\Ltwo[T]}
  \end{align*}
  and the result follows by applying Cauchy-Schwarz and Young
  inequalities.
\end{proof}

\begin{lemma}
  \label{lem:second_lemma_prework-rev}
  Suppose~\eqref{eq:13} holds, ${\sh}\in \Regge_h^k$,
  ${\vh}\in \Wo_h^k$, and $\sigma \in H^1(\om, \Sc) \cap C^0(\om, \Sc)$.
  Then for sufficiently small $h$,
  \begin{align}
    \label{eq:est_curl_metric_1_prework-rev}
    (\curl_{\gex, h} {{\sh}} -  \curl_{g, h} {{\sh}}, {\vh})_\om
    & - \Gamma({\gex}-{\gh},\Sigma)
      \,\lesssim\,
      \nrm{{\gex}-{\gh}}_\infty
      \|{\sh}\|_{\Honeh}\|{\vh}\|_{\Ltwo},
    \\
    \label{eq:est_curl_metric_2_prework-rev}
    (\curl_{\gex, h} {{\sh}} -  \curl_{g, h} {{\sh}}, {\vh})_\om
    & - \Gamma({\gex}-{\gh},\Sigma)
      \,\lesssim\,
      \nrm{{\gex}-{\gh}}_2
      \|{\sh}\|_{{W_h^{1,\infty}}}\|{\vh}\|_{\Ltwo}.
  \end{align}
  where
  $\Sigma^{ijl}={\sh}_{mn} \veps^{jm}{\gh}^{nl}{\vh}^i/
  \sqrt{\det \gh}$.
\end{lemma}
\begin{proof}
  By \eqref{eq:cov-curl-lift} and \eqref{eq:cov_distr_curl1}, putting
  $\alpha(g) = \mt{\curl \mt {{\sh}}}_l {\vh}^l -
  {\sh}_{ij}\veps^{mi}\Gamma_{lm}^j(g){\vh}^l$,
  \[
    (\curl_{\gex, h} {{\sh}}
    -  \curl_{g, h} {{\sh}}, {\vh})_\om
    = \sum_{T \in \T}
    \int_T  \frac{\alpha(\gex)}{\sqrt{\det \gex}} -
        \frac{\alpha(\gh)}{\sqrt{\det \gh}} \; \da
        - \int_{\d T} \beta(\gex) - \beta(g) \; \dl,
  \]
  where
  $\beta(g) = {\sh}_{\nv\tv} \beta_1(g) {\vh}^\nv -
  {\sh}_{\tv\tv}\beta_2(g) {\vh}^\nv$ and $\beta_i$ are as in
  \eqref{eq:14}.  The first integrand, which we denote by $A_T$, can
  be simplified to
  \begin{align*}
    A_T
    & = \frac{\alpha(\gex)}{\sqrt{\det \gex}}
      -
      \frac{\alpha(\gh)}{\sqrt{\det \gh}}
    \\
    &
      =  \alpha(\gex)
      \left(
      \frac{1}{\sqrt{\det \gex}} -
      \frac{1}{\sqrt{\det \gh}}
      \right)
      - \frac{{\sh}_{ij} \veps^{mi} }
      {\sqrt{\det g}}
      \big[\Gamma_{lm}^j(\gex) - \Gamma_{lm}^j(\gh)\big] {\vh}^l.
  \end{align*}
  Using $\Gamma_{ij}^m(g) =g^{ml}\Gamma_{ijl}(g)$ and the linearity of
  the Christoffel symbols of the first kind,
  $
    \Gamma_{lm}^j(\gex) - \Gamma_{lm}^j(\gh) =
    {\gh}^{jq}\Gamma_{lmq}(\gex-\gh)
    +
    \big(\gex^{jq} - {\gh}^{jq}\big)\Gamma_{lmq}(\gex).
  $
  Hence
  \[
    A_T = \alpha(\gex) [\beta_1(\gex) - \beta_1(\gh)] -
    \frac{{\sh}_{ij} \veps^{mi} }
    {\sqrt{\det g}} (\gex - \gh)^{jq} \Gamma_{lmq}(\gex) v^l
    + \Sigma^{lmq} \Gamma_{lmq}( \gex - \gh)
  \]
  with $\Sigma^{lmq} = {\sh}_{ij} \veps^{mi} g^{jq} v^l/\sqrt{\det g}$.

  Next, we focus on the boundary integrand $B_T = \beta(\gex ) -\beta(\gh)$.
  By Lemma~\ref{lem:prep_bnd_terms-rev},  
  \begin{align*}
    B_T & = [\beta_1(\gex) - \beta_1(g) ] \sigma_{\nv\tv} v^\nv
          - [\beta_2(\gex) - \beta_2(g) ] \sigma_{\tv\tv}  v^\nv
    \\
        & = (\gex - g)_{\tv\tv}
          [ f_1(g) {\sh}_{\nv\tv} + f_2(g) {\sh}_{\tv\tv} ]
          {\vh}^\nv + (\eta_1(\gex, g)\sigma_{\nv\tv} - \eta_2(\gex, g)\sigma_{\tv\tv}) v^\nv
    \\
        & + (\sigma_{\nv\tv}  \epsilon_1^2 - \sigma_{\tv\tv} \epsilon_2^2) v^\nv, 
  \end{align*}
  with the $\eta_i, f_i,$ and $\epsilon_i$ provided there.
  We claim that
  \begin{equation}
    \label{eq:15}
    (\eta_1(\gex, g)\sigma_{\nv\tv} - \eta_2(\gex, g)\sigma_{\tv\tv}) v^\nv =
    \Sigma^{\nv\nv\tv} (\gex - g)_{\nv\tv}
    + \frac 1 2 \Sigma^{\nv\nv\nv} (\gex - g)_{\nv\nv}.
  \end{equation}
  Indeed, by \eqref{eq:Eucl-normal-comp},
  $\sqrt{\det g}\; \Sigma^{\nv\nv q} = -{\sh}_{ij} \tv^i g^{jq} v^\nv$.
  In this expression, we substitute the orthogonal decomposition
  $ {\sh}_{ij} \tv^i = \sigma_{\tv\tv} \delta_{jm}\tv^m + \sigma_{\tv
    \nv} \delta_{jm} \nv^m$, together with the cofactor expansion of
  $g^{-1}$, to get 
  \[
    \Sigma^{\nv\nv q } = (\det g)^{-3/2} (\sigma_{\tv\nv}
    g_{\tv m} - \sigma_{\tv\tv}g_{\nv m}) v^\nv \veps^{qm}.
  \]
  This yields expressions for both $\Sigma^{\nv\nv\tv}$ and
  $\Sigma^{\nv\nv\nv}$ on the right hand side of~\eqref{eq:15}, which
  can again be simplified using \eqref{eq:Eucl-normal-comp} to
  verify~\eqref{eq:15}.
  Gathering these observations together, we have proved that
  \begin{align}
    \nonumber
    (\curl_{\gex, h}{{\sh}}    -
    & \curl_{g, h} {{\sh}}, {\vh})_\om
      = \sum_{T \in \T} \int_T A_T \, \da + \int_{\d T} B_T \, \dl
    \\ \nonumber 
    = 
     \sum_{T \in \T}
    \Big[
    & 
      \big(\alpha(\gex), \beta_1 (\gex) - \beta_1(g)\big)_T
      +
      \big((\gex - g)^{jq},
      \sigma_{ij} \veps^{mi} v^l \beta_1(g) \Gamma_{lmq}(\gex)\big)_T    
    \\ \label{eq:16}
    &
      + \big( (\gex - g)_{\tv\tv}, 
      [ f_1(g) {\sh}_{\nv\tv} + f_2(g) {\sh}_{\tv\tv} ]
      {\vh}^{\nv}\big)_{\d T}      
    \\ \nonumber 
    & +  \big((\sigma_{\nv\tv}  \epsilon_1^2 - \sigma_{\tv\tv} \epsilon_2^2),
      v^{\nv}\big)_{\d T} + \Gamma_T(\gex - g, \Sigma) 
      \Big],
  \end{align}
  where $\Gamma_T$ is as defined in \eqref{eq:distr_christoffel_regge2-rev}.
  Moving $\Gamma(\gex - g, \Sigma)$ to the left hand side, we proceed
  to estimate the terms on the right one by one.
  
  By H\"older inequality and \eqref{eq:est_det_by_ten-rev}
  \begin{align*}
    \big(\alpha(\gex),
    \beta_1 (\gex) - \beta_1(g)\big)_T
      & \,\lesssim\, \| \beta_1 (\gex) - \beta_1(g) \|_{L^\infty(T)}
      \|{\sh}\|_{\Honeh[T]}\|{\vh}\|_{\Ltwo[T]}
    \\
      & \lesssim \|{\gex}-{\gh}\|_{\Linf[T]}\|{\sh}\|_{\Honeh[T]}\|{\vh}\|_{\Ltwo[T]}.
  \end{align*}
  For the next term, we again use H\"older inequality, followed by
  \eqref{eq:est_inv_by_ten-rev}, and \eqref{eq:est_detgh_by_g-rev}:
  \begin{align*}
    \big((\gex - g)^{jq},
    \sigma_{ij} \veps^{mi} v^l \beta_1(g) \Gamma_{lmq}(\gex)\big)_T
    & \,\lesssim\, \|{\gex}-{\gh}\|_{\Linf[T]}
      \|{\sh}\|_{\Ltwo[T]}\|{\vh}\|_{\Ltwo[T]}.
  \end{align*}
  Next, to bound the boundary term over $\d T$, we extend the constant
  vectors $\nv$ and $\tv$ from an edge $E$ of $\d T$ to $T$, let
  $\Psi=f_1({\gh})\sh_{\nv\tv}+f_2({\gh})\sh_{\tv\tv}\in\Winf[T]$ and
  $q=v^{\nv}\in\Pol^k(T)$ and apply~\eqref{eq:est_bnd_lemma2-rev} to get 
  \begin{align*}
    \big( (\gex - g)_{\tv\tv}, \;
      [ f_1(g) {\sh}_{\nv\tv} + f_2(g) {\sh}_{\tv\tv} ]
      {\vh}^{\nv}\big)_{E}
    & \,\lesssim\,
      \|{\gex}-{\gh}\|_{\Linf[T]}\|v_h\|_{\Ltwo[T]}
      \big(\|\Psi\|_{\Honeh[T]}+h|\Psi|_{H^2_h(T)}\big)
    \\
    &\,\lesssim\,
      \|{\gex}-{\gh}\|_{\Linf[T]}\|v_h\|_{\Ltwo[T]}
      \big(\|\sh\|_{\Honeh[T]}+h|\sh|_{H^2_h(T)}\big)
    \\
    &\,\lesssim\,\|{\gex}-{\gh}\|_{\Linf[T]}\|v_h\|_{\Ltwo[T]}
      \|\sh\|_{\Honeh[T]},
  \end{align*}
  where we used \eqref{eq:est_gh_by_g-rev} and
  \eqref{eq:discr_inverse_inequ-rev}. Finally, for the
  $\epsilon_i$-terms in~\eqref{eq:16}, noting that
  $\epsilon_i \lesssim \|\gex - \gh\|_{L^\infty(\d T)}^2 \lesssim h\|\gex
  - \gh\|_{L^\infty(\d T)}$, by a trace inequality,
  \begin{align*}
    \big((\sigma_{\nv\tv}  \epsilon_1^2 - \sigma_{\tv\tv} \epsilon_2^2),
    v^{\nv})\big)_{\d T}
    & \,\lesssim\,
      h\|{\gex}-{\gh}\|_{\Linf[\d T]}\|\sigma_{h}\|_{\Ltwo[\d T]}\|v_h\|_{\Ltwo[\d T]}
    \\
    &\,\lesssim\,
      \|{\gex}-{\gh}\|_{\Linf[T]}\|\sigma_{h}\|_{\Ltwo[T]}\|v_h\|_{\Ltwo[T]}.
  \end{align*}
  Using these estimates in~\eqref{eq:16}, we finish the proof
  of~\eqref{eq:est_curl_metric_1_prework-rev}.

  The proof of \eqref{eq:est_curl_metric_2_prework-rev} is similar.
\end{proof}

\begin{lemma}
  \label{lem:christoffel_magic-rev}
  Adopt the assumptions of Lemma~\ref{lem:second_lemma_prework-rev}
  and let $\Sigma$ be as defined there.  Then for sufficiently small
  $h$,
  \begin{align}
    \Gamma(\gex-{\gh},\Sigma)
    & \,\lesssim\, \nrm{ \gex - \gh }_\infty
      \|\sigmaappr\|_{\Honeh}\|v\|_{\Ltwo},
      \label{eq:christoffel_magic-rev}
    \\
    \Gamma(\gex-{\gh},\Sigma)
    & \,\lesssim\, \nrm{ \gex - \gh }_2
      \|\sigmaappr\|_{W_h^{1,\infty}}\|v\|_{\Ltwo}.
      \label{eq:christoffel_magic-2}
  \end{align}
\end{lemma}
\begin{proof}
  First observe that in the case $k=0$, the function $\Sigma$ is
  constant on each $T \in \T$, so by Lemma~\ref{lem:christ-Regge},
  $\Gamma(\gex-{\gh},\Sigma) = 0$. In the $k \ge 1$ case, define
  ${\gh}_{0}:=\RegInt[0]{\gex}$, ${\sh}^{0}:=\RegInt[0]{\sh}$, and
  $\Sigma_0^{ijl}={\sh}^0_{mn}
  \beta_1({\gh}_0)\veps^{jm}{\gh}_0^{nl}{v}^i.$
  Splitting
  \begin{align*}
    \Sigma^{ijl} - \Sigma_0^{ijl}
    & =
    {\sh}_{mn}
    \beta_1({\gh})\veps^{jm}{\gh}^{nl}{v}^i - 
      {\sh}^0_{mn}
      \beta_1({\gh}_0)\veps^{jm}{\gh}_0^{nl}{v}^i
    \\
    & =
      ({\sh}_{mn} - {\sh}^0_{mn})
      \beta_1({\gh}){\gh}^{nl}\veps^{jm}{v}^i 
      + 
      \big(\beta_1({\gh}){\gh}^{nl} - 
      \beta_1({\gh}_0){\gh}_0^{nl}\big){\sh}^0_{mn} \veps^{jm} v^i
  \end{align*}
  it is easy to see from H\"older inequality,
  \eqref{eq:est_inv_by_ten-rev}, and \eqref{eq:reg_appr_lo_vol-rev}
  that
  \begin{equation}
    \label{eq:18}
    \| \Sigma - \Sigma_0 \|_{L^1(T)} +
    h \| \Sigma - \Sigma_0 \|_{L^1(\d T)}
    \,\lesssim \,
    h \| \sigma \|_{H^1(T)} \| v \|_{L^2(T)}.
  \end{equation}
  Since $\Sigma_0 \in \Pol^0(\T,\R^{2\times2\times2})$,
  Lemma~\ref{lem:christ-Regge} implies
  $\Gamma(\gex - g, \Sigma) = \Gamma(\gex - g, \Sigma - \Sigma_0),$ 
  \begin{align*}
    \Gamma(\gex-{\gh},\Sigma)
     = \sum_{T \in \T}
      \Big[
    &
      \big(\Gamma_{lmn}(\gex-{\gh}), (\Sigma - \Sigma_0)^{lmn}\big)_T \,+ 
    \\
    & \frac{1}{2}
      \big( (\gex - \gh)_{\nv\nv}, (\Sigma - \Sigma_0)^{\nv\nv \nv} \big)_{\d T}
      +
      \big( (\gex - \gh)_{\nv\tv}, (\Sigma - \Sigma_0)^{\nv\nv \tv} \big)_{\d T}
      \Big]
    \\
    \,\lesssim\,\sum_{T \in \T}
    \Big[
    & 
      \| \gex - \gh\|_{W^{1, \infty}(T)} h \| \sigma \|_{H^1(T)} \| v \|_{L^2(T)} \,+ 
    \\
    &  \| \gex - \gh\|_{L^\infty(\d T)}  \| \sigma \|_{H^1(T)} \| v \|_{L^2(T)}
    \Big],
  \end{align*}
  where we have also used
  $\| \Gamma_{lmn}(\gex-{\gh}) \|_{L^\infty(T)}\lesssim \|\gex -
  {\gh}\|_{\Winf[T]}$, H\"older inequality, and \eqref{eq:18}.  By
  Cauchy-Schwarz inequality,~\eqref{eq:christoffel_magic-rev}
  follows. The proof of~\eqref{eq:christoffel_magic-2} is similar.
\end{proof}

\begin{proof}[Proof of Theorem~\ref{thm:curlg}]
  Error estimate~\eqref{eq:curlg-1} directly follows from
  Lemma~\ref{lem:first_lemma-rev}. Using the estimates of
  Lemma~\ref{lem:christoffel_magic-rev} to bound the $\Gamma$-terms in
  the inequalities of Lemma~\ref{lem:second_lemma_prework-rev}, the
  remaining error estimates of the theorem follow.
\end{proof}
\begin{proof}[Proof of Corollary~\ref{cor:curlg}]
The proof follows along the lines of \cite[pp.~1818]{Gaw20}, where the error is compared to the weaker $\Ltwo$-norm using the Scott-Zhang interpolant, inverse estimates, and the triangle inequality.
\end{proof}

\subsection{Analysis of the covariant incompatibility}

The error analysis here can now be given by a simple argument
using Theorem~\ref{thm:curlg}.

\begin{proof}[Proof of Theorem~\ref{thm:incg}]
  To prove \eqref{eq:incg-1}, we note that by the definitions
  \eqref{eq:def_distr_cov_inc} and \eqref{eq:incg-h}, 
  \[
    (\inc_{g, h}(\sigma- \RegInt[k] \sigma), u_h)_\om
    = (\curl_{g, h} (\sigma- \RegInt[k] \sigma),\rot {u_h})_\om
  \]
  for all $u_h \in \Vo^{k+1}$. Since
  $\| \rot u_h \|_{L^2} = |u_h|_{H^1}$, the estimate
  \eqref{eq:incg-1} follows from \eqref{eq:curlg-1}. The remaining
  estimates similarly follow from~\eqref{eq:curlg-2}
  and \eqref{eq:curlg-3}.
\end{proof}
\begin{proof}[Proof of Corollary~\ref{cor:incg}]
	This proof follows along the lines in \cite[pp.~1818]{Gaw20}.
\end{proof}

We conclude this short subsection with a few remarks on our analysis so far.
To compare our
analysis with \cite{Gaw20}, the easily spotted difference is that
we use our operator $\inc_g$ instead of the operator 
$\div_g\div_g$ in \cite{Gaw20}. These two operators are
closely related in two dimensions (see
Appendix~\ref{sec:rel_cov_inc_vov_divdiv}).  A more substantial
difference is that while~\cite{Gaw20} separately estimates certain
element terms and inter-element jumps (by applying a triangle
inequality first), we do not. Instead, we kept
such terms together, gathered terms of good convergence rates, and
identified the remainder as a collection of terms that look like those
arising from the formula for distributional Christoffel symbols of
first kind. The next key insight was
\eqref{eq:biortho_regge_christoffel-rev}, which zeroed out the latter
collection.  Also notable from our analysis so far is the idea of
splitting the error terms into a high-order ones and ones that might
be sub-optimal in general, but vanishes in specific cases. Such an
idea was also used in \cite{Walker21}, where the convergence of a
surface $\mathrm{div \; div}$ operator on an approximated triangulation
is proven (in an extrinsic manner) to converge.

\subsection{Analysis of the curvature approximation}

Now we turn to the proof of
Theorem~\ref{thm:convergence_curvature_hm1-rev}.  The key extra
ingredient here is a technique pioneered in~\cite{Gaw20} to represent
the curvature approximation using an integral of its first variation,
as stated in the next lemma. For any $g \in \Regge^+(\T)$ and
$\sigma \in \Regge(\T),$ using $\inc_g$ of 
Definition~\ref{def:incg-extended}, let
\begin{align*}
  b_h(g,\sigma,u) & = -\act{ \inc_g\sigma,u}_{\Vo(\T)},
  &
    G(t)
  & = \Eucl+t(g-\Eucl),
  & 
    \Gex(t) & = \Eucl+t(\gex-\Eucl).
\end{align*}

\begin{lemma}
  \label{lem:gauss_curv_integral_repr}
  Let $g\in\Regge^+(\T)$, $\Gauss_h(g)\in \Vo_h^{k+1}$ be as in
  Definition~\ref{def:lift_distr_curv_mf} for some $k \ge 0$. Also let
  $\gex$ be a smooth metric and let $\Gauss(\gex)$ denote its smooth
  Gauss curvature.
  Then 
  \begin{align}
    \label{eq:23}
    \int_\T \Gauss_h(g) \,u_h &= \frac{1}{2}\int_0^1b_h(G(t),g-\Eucl,u_h)\,dt,\quad&& \forall u_h\in \Vo_{h}^{k+1},\\\label{eq:24}
    \int_\T \Gauss(\gex) \,u &= \frac{1}{2}\int_0^1b_h(\Gex(t),\gex-\Eucl,u)\,dt,\quad&& \forall u\in \Vo(\T).
  \end{align}
\end{lemma}
\begin{proof}
  Since $K_h(\delta) = 0$, 
  \begin{align*}
    \int_\T \Gauss_h(g) \,u_h
    & = \int_0^1 \frac {d}{d t} \int_\T K_h(G(t)) \,u_h.
  \end{align*}
  Now, expand the inner integral using \eqref{eq:1} and differentiate
  each term in the direction $\sigma = G'(t) = g-\delta$, using each
  of the three identities of Lemma~\ref{lem:variation}. Comparing the
  result, term by term, to the expression in
  Proposition~\ref{prop:cov_distr_inc}, equation~\eqref{eq:23} is
  proved. The proof of~\eqref{eq:24} is similar, after noting that the
  global smoothness of $\gex-\delta$ implies that the jump terms in
  $b_h(\Gex(t), \gex - \delta, u)$ vanish (by the last statement of 
  Proposition~\ref{prop:cov_distr_inc}).
\end{proof}

To state a simple lemma before the error analysis, for any
$u,v, f\in\Ltwo[\Omega]$, let
\[
  ( u,v)_{g} =\int_{\Omega}u\, v \,\sqrt{\det g}\;\da,\quad
  \| u \|_g = (u, u)_g^{1/2}, 
\]
and let
$P^g_{k+1}:\Ltwo[\Omega]\to \Vo_{h}^{k+1}$ denote the
$(\cdot,\cdot)_g$-orthogonal projector into $\Vo_{h}^{k+1}$.

\begin{lemma}
  \label{lem:est_mixed_term-rev}
  For any  $u\in\Honez[\Omega]$ and any $\gh \in \Regge^k_h$, 
  \begin{align*}
    (\Gauss_h(\gh)-\Gauss(\gex),u- P^g_{k+1}u)_{\gex}
    & \,\lesssim\,  h \|u\|_{\Hone}
      \|(\idop -     \Pi^\VV_{k+1}) \Gauss(\gex)\|_{\Ltwo},
  \end{align*}
\end{lemma}
\begin{proof}
  Since
  \begin{align*}
    (\Gauss_h(\gh)-\Gauss(\gex),u- P^{\gex}_{k+1}u)_{\gex}
    & =
      (\Gauss(\gex),u- P^{\gex}_{k+1}u)_{\gex}
    \\
    & =
      \big((\idop - P^{\gex}_{k+1})\Gauss(\gex),
      (\idop - P^{\gex}_{k+1})u\big)_{\gex}
    \\
    & \lesssim
      h\|u\|_{\Hone[\Omega]}
      \inf\limits_{v_h\in\Vo_h^{k+1}}\|\Gauss(\gex)-v_h\|_{\Ltwo[\Omega]},
  \end{align*}
  where we have used the equivalence of $L^2(\om)$-norm and
  $\|\cdot \|_{\gex}$-norm.
\end{proof}

\begin{proof}[Proof of Theorem~\ref{thm:convergence_curvature_hm1-rev}] 
	The general structure of the proof follows \cite{Gaw20}.
	Let $u\in\Honez[\Omega]$ and let 
        $u_h \in \Vo_h^{k+1}$. Then
	\begin{align*}
          (\Gauss_h(\gh)-\Gauss(\gex),u)_{\gex}
          & = (\Gauss_h(\gh)-\Gauss(\gex),u - u_h + u_h)_{\gex}
            = s_1 + s_2 + s_3
	\end{align*}
	where
        $ s_1 = (\Gauss_h(\gh),u_h)_{\gh} - (\Gauss(\gex),u_h)_{\gex},
        $
        $ s_2 = (\Gauss_h(\gh),u_h)_{\gex} -( \Gauss_h(\gh),u_h)_{\gh}$, and         
        $ s_3 = ( \Gauss_h(\gh)-\Gauss(\gex),u-u_h)_{\gex}.$
        We proceed to estimate each $s_i$.
          
By Lemma~\ref{lem:gauss_curv_integral_repr} and \eqref{eq:incg-h} we rewrite $s_1$ as follows:
\begin{align*}
  s_1 & = \frac{1}{2}\int_0^1 b_h(G(t),\gh-\Eucl,u_h)-b_h(\Gex(t),\gex-\Eucl,u_h)\,dt\\
	&= \frac{1}{2}\int_0^1 b_h(G(t),\gh-\Eucl,u_h)-b_h(\Gex(t),\gh-\Eucl,u_h)\,dt +\frac{1}{2}\int_0^1 b_h(\Gex(t),\gh-\gex,u_h)\,dt
	\\&= \frac{1}{2}(\inc_{\Gex(t),h}(\gh-\Eucl)-\inc_{G(t),h}(\gh-\Eucl),u_h)_{\Omega} -\frac{1}{2}(\inc_{\Gex(t),h}(\gh-\gex),u_h)_{\Omega} . 
\end{align*}
Now we can use Theorem~\ref{thm:incg}.
Applying \eqref{eq:incg-2} with  $g=\Gex(t)$ and $\sigmaappr=\gh-\Eucl\in \Regge_h^k$ and applying
\eqref{eq:incg-1} with  $g=\Gex(t)$ and $\sigma=\gex$,
	\begin{align}
	\label{eq:est_2nd_term}
          s_1 = 
           (\Gauss_h(\gh),u_h)_{\gh} - (\Gauss(\gex),u_h)_{\gex}
                                             \,\lesssim\, \nrm{\gex-{\gh}}_\infty \|\nabla u_h\|_{\Ltwo[]}.
	\end{align}
        for any $u_h \in \Vo_h^{k+1}$, 
        where we have also used $\nrm{\cdot }_2 \lesssim \nrm{\cdot }_\infty$ and  \eqref{eq:est_gh_by_g-rev}.

        For the next term $s_2$, we use H\"older inequality and
        \eqref{eq:est_det_by_ten-rev} to get
	\begin{align*}
          s_2
          & = (\Gauss_h(\gh),u_h)_{\gex} - (\Gauss_h(\gh),u_h)_{\gh}
            = \big( \Gauss_h(\gh) u_h, \sqrt{\det \gex}-\sqrt{\det\gh}\big)_\om
          \\
          &\leq \|\Gauss_h(\gh)\|_{\Ltwo}\|u_h\|_{\Ltwo}\|\sqrt{\det(\gex)}-\sqrt{\det(\gh)}\|_{\Linf}
          \\
          & \lesssim\,\|\gex-\gh\|_{\Linf}\|\Gauss_h(\gh)\|_{\Ltwo}
            \| u_h\|_{\Ltwo}.                
	\end{align*}
Next, setting $u_h = K_h(g)$ in~\eqref{eq:est_2nd_term}, we
        obtain
        $\|\Gauss_h(\gh)\|_{\Ltwo}^2 \lesssim h^{-1} \nrm{ \gex -
          \gh}_\infty \| K_h(g)\|_{L^2} + \| K(\gex) \|_{L^2} \|
        K_h(g)\|_{L^2} $ using the inverse
        inequality~\eqref{eq:discr_inverse_inequ-rev}. Since standard
        estimates imply 
        $h^{-1} \nrm{ \gex - \gh}_\infty \lesssim 1$, we conclude that
        $\| K_h(g) \|_{L^2} \lesssim 1$, so 
	\begin{align}
          \label{eq:26}
          s_2\,\lesssim\,\|\gex-\gh\|_{\Linf}          \| u_h\|_{\Ltwo}.
	\end{align}


        Finally, to estimate $s_3$,  
        let us now fix $u_h=P^{\gex}_{k+1}u\in \Vo_{h}^{k+1}.$
        Then by  Lemma~\ref{lem:est_mixed_term-rev}:
	\begin{align}
          \label{eq:27}
          s_3
          \,\lesssim\,  h \|u\|_{\Hone}
          \left\|(\idop -     \Pi^\VV_{k+1}) \Gauss(\gex)\right\|_{\Ltwo}.
	\end{align}
        For this  choice of $u_h$, we also have 
        $| u_h - \Pi^\VV_{k+1} u |_{H^1} \lesssim h^{-1} \| u_h -
        \Pi^\VV_{k+1} u \|_{L^2} \lesssim | u |_{H^1}, $ so by the
        stability \cite{CrouzThome87}
        of the $L^2$ projection $\Pi^\VV_{k+1}$ in $H^1$, we conclude
        that
        $\| u_h \|_{H^1} \lesssim \| u \|_{H^1}$. This 
        allowing us to replace $u_h$ by $u$
        in the bounds for $s_1$ and $s_2$.
	All together, \eqref{eq:est_2nd_term}, \eqref{eq:26}, and \eqref{eq:27},
        imply
	\begin{align*}
	&(\Gauss_h(\gh)-\Gauss(\gex),u)_{\gex} \,\lesssim\, \big(\nrm{\gex-{\gh}}_\infty+h 
   \left\|(\idop -     \Pi^\VV_{k+1}) \Gauss(\gex)\right\|_{\Ltwo}\big) 
  \|u\|_{\Hone[]}.
	\end{align*}
	Thus, since $\| u / \sqrt{\det \gex} \|_{H^1} \lesssim \| u \|_{H^1}$,
	\begin{align*}
          \|\Gauss_h(\gh)-\Gauss(\gex)\|_{H^{-1}}
          & = \sup\limits_{u\in \Honez}\frac{(\Gauss_h(\gh)-\Gauss(\gex),u)_{\om}}{\|u\|_{\Hone}}
            = \sup\limits_{u\in \Honez}\frac{(\Gauss_h(\gh)-\Gauss(\gex),
            u/\sqrt{\det \gex})_{\gex}}{\|u \|_{\Hone}}
          \\
          & 
            \,\lesssim\,
            \nrm{\gex-{\gh}}_\infty+h 
	\|(\idop -     \Pi^\VV_{k+1}) \Gauss(\gex)\|_{\Ltwo},
	\end{align*}
        which proves the first estimate of the theorem. The second follows from  standard interpolation error estimates.
\end{proof}
\begin{proof}[Proof of Corollary~\ref{cor:convergence_curvature_l2_hl}]
  One can prove this following \cite[pp.~1818]{Gaw20} using the
  improved estimate of
  Theorem~\ref{thm:convergence_curvature_hm1-rev}] in place of the
  estimate used there.
\end{proof}



\subsection{Analysis of the connection approximation}

The integral representation of the connection 1-form can be given
analogously to the curvature case once we know the variation of the
connection with respect to the metric.  Recall that we compute the
connection using the canonical $g$-orthonormal frame
$\GBasis_i(t)=G(t)^{-1/2}\EuclBasis_i$ (see \eqref{eq:28}) obtained using the flow
\eqref{eq:ODE_oneform} evaluated at $t=1$.  We only discuss the
variation of $\OneForm_g$ with respect to $g$ along this flow (i.e.,
unlike Lemma~\ref{lem:variation}, the following is valid only for a
specific direction $\sigma$).

\begin{lemma}
  \label{lem:evolution_one_form}
  Let $\sigma = d G/ dt = g - \delta$. Then for all $X \in \Xm M$,
	\begin{subequations}
	\begin{align}
          \frac{d}{d t}\OneForm_{G(t)}(X)
          &= -\frac{1}{2}
            (\curl_{G(t)}\sigma)(X),
            \label{eq:evolution_one_form_vol}
          \\
          \frac{d}{d t} \Theta^E
          &=\frac{1}{2}\jmp{\sigma_{\gn\gt}}.
            \label{eq:evolution_one_form_bnd}
	\end{align}
	\end{subequations}
\end{lemma}
\begin{proof}
  Both identities follow from \cite{BKG21}, e.g., with $e_i(t) = G(t)^{-1/2} E_i$, 
  \begin{align*}
    2\frac{d}{d t}\OneForm_{G(t)}(X)
    & = -(\nabla_{\GBasis_1}\sigma) (\GBasis_2, X)
      + (\nabla_{\GBasis_2}\sigma)( e_1, X)
    &&  \text{ by \cite[Eq.~(7)]{BKG21}},
    \\
    & = -(\nabla_{\GBasis_1}\sigma) (X, \GBasis_2)
      + (\nabla_{\GBasis_2}\sigma)( X, e_1)
    &&  \text{ by symmetry of $\sigma$},
    \\
    & = - d^1 \sigma_X(e_1, e_2)
    &&  \text{ by \eqref{eq:d1sigma},}
  \end{align*}
  which equals $-(\curl_{G(t)} \sigma)(X)$, thus proving~\eqref{eq:evolution_one_form_vol}.  For a proof of
  \eqref{eq:evolution_one_form_bnd}, see \cite[Proposition
  2.20]{BKG21}.
\end{proof}

Let
$c_h(g,\sigma,v) = -\act{ \curl_g{\sigma},Q_gv}_{\Wo_g(\T)}$ for $v \in \Wo(\om)$, 
where  $Q_g$ is as in \eqref{eq:def_Q_op} and the distributional covariant curl is as defined in~\eqref{eq:cov_distr_curl}.

\begin{lemma}
	\label{lem:conn_oneform_integral_repr}
	Let $g\in\Regge^+(\T)$, $k\in\N_0$. There holds for $\OneForm_h(g)\in \Wo_h^k$ from Definition~\ref{def:distributional_one_form}, the exact metric $\gex$ and its connection 1-form $\OneForm(\gex)$
	\begin{align}
	\label{eq:disc_lifting_oneform_prob}
          \int_{\T}\OneForm_h(g)\, Q_gv_h
          & = \frac{1}{2}\int_0^1c_h(G(t),g-\Eucl,v_h)\,dt,\,\,
          && \forall v_h\in \Wo_h^k,\\
	\label{eq:integral_id_ex_connectionform}
          \int_{\T}\OneForm(\gex)\, Q_gv
          & = \frac{1}{2}\int_0^1c_h(\Gex(t),\gex-\Eucl,v)\,dt,\,\,
          && \forall v\in \Wo(\om).
	\end{align}
\end{lemma}
\begin{proof}
  The proof is similar to proof of
  Lemma~\ref{lem:gauss_curv_integral_repr}:
  \eqref{eq:disc_lifting_oneform_prob} follows from
  Lemma~\ref{lem:evolution_one_form} and the fundamental theorem of
  calculus, and for~\eqref{eq:integral_id_ex_connectionform}, we note
  that the jump terms in $c_h(\Gex(t),\gex-\Eucl,v)$ vanish due to the
  global smoothness of $\gex - \delta.$
\end{proof}

\begin{proof}[Proof of Theorem~\ref{thm:convergence_connectionform_l2-rev}]
	For $v\in\Ltwo[\Omega,\R^2]$ we add and subtract the $\Ltwo$-orthogonal projector $v_h=\Pi_{k}^\WW v\in \Wo_h^k$ and write
	\begin{align*}
	(\OneForm_h(\gh)-\OneForm(\gex),v)_{\Omega} =   s_1+s_2
	\end{align*}
        where $s_1= (\OneForm_h(\gh)-\OneForm(\gex),v -v_h)_{\Omega}$ and
        $s_2 = (\OneForm_h(\gh)-\OneForm(\gex),v_h)_{\Omega}$.
        Due to the properties of the $\Ltwo$-orthogonal projector we obtain
	\begin{align*}
	s_1&=(\Pi_{k}^\WW\OneForm(\gex)-\OneForm(\gex),v-v_h)_{\Omega}\,\lesssim\,\|v\|_{\Ltwo[]}\left\|(\idop-\Pi_{k}^\WW)\OneForm(\gex)\right\|_{\Ltwo[]}.
	\end{align*}
	For $s_2$, we use Lemma~\ref{lem:conn_oneform_integral_repr} and \eqref{eq:cov-curl-lift} to get
	\begin{align*}
	s_2&=\frac{1}{2}\int_0^1 c_h(G(t),\gh-\Eucl,v_h)-c_h(\Gex(t),\gh-\Eucl,v_h)+ c_h(\Gex(t),\gh-\gex,v_h)\,dt\\
	&= \frac{1}{2}(\curl_{\Gex(t), h}(\gh-\Eucl)-\curl_{G(t), h}(\gh-\Eucl), v_h)_\om-\frac{1}{2}(\curl_{\Gex(t), h}(\gh-\gex), v_h)_\om.
	\end{align*}
	 Estimating using \eqref{eq:curlg-1} and \eqref{eq:curlg-3} of  Theorem~\ref{thm:curlg} and \eqref{eq:est_gh_by_g-rev},
	\begin{align*}
	(\OneForm_h(\gh)-\OneForm(\gex),v_h)_{\Omega}\,\lesssim\,\nrm{\gex-\gh}_{2} \| v_h\|_{\Ltwo[]}.
	\end{align*}
	Thus, we obtain
	\begin{align*}
	\|\OneForm_h(\gh)-\OneForm(\gex)\|_{\Ltwo[]}&=\sup\limits_{v\in\Ltwo[\Omega,\R^2]}\frac{(\OneForm_h(\gh)-\OneForm(\gex),v+ v_h -v_h)_{\Omega}}{\|v\|_{\Ltwo[]}}\\
	&\,\lesssim\, \nrm{\gex-\gh}_{2}  + \left\|(\idop-\Pi_{k}^\WW)\OneForm(\gex)\right\|_{\Ltwo[]}.
	\end{align*}
        proving \eqref{eq:oneform_estimate1}. Inequality
	\eqref{eq:oneform_estimate2} follows by interpolation error estimates.
\end{proof}

\begin{proof}[Proof of Corollary~\ref{cor:convergence_connectionform_l2_hl}]
This is analogous to proof of Corollary~\ref{cor:curlg}.
\end{proof}

\begin{remark}[Case $k=0$ for connection approximation]
	\label{rem:conv_one_form_lo}
	In the case $k=0$, the space $\Wo_h^0$ consists of piecewise constant vector fields with normal continuity. Thus every function in $\Wo_h^0$ is exactly divergence-free. If the exact connection 1-form $\OneForm_{\gex}$         is not divergence-free,
        then it cannot generally be approximated by functions in $\Wo_h^0$. Thus, no convergence for the connection approximation in the lowest order case $k=0$ should be expected. This is confirmed by numerical experiments in the next section.
\end{remark}

\section{Numerical examples}
\label{sec:num_examples}

In this section we confirm, by numerical examples, that the theoretical convergence rates from Theorem~\ref{thm:convergence_curvature_hm1-rev} and Theorem~\ref{thm:convergence_connectionform_l2-rev} are sharp.
All experiments were performed in the open source finite element software NGSolve\footnote{www.ngsolve.org} \cite{Sch97,Sch14}, where the Regge elements are available.

\subsection{Curvature approximation}

We consider the numerical example proposed in \cite{Gaw20}, where on the square $\Omega=(-1,1)\times (-1,1)$ the smooth Riemannian metric tensor
\begin{align*}
\gex(x,y):= \mat{1+(\pder{f}{x})^2 & \pder{f}{x}\pder{f}{y} \\ \pder{f}{x}\pder{f}{y} & 1+ (\pder{f}{y})^2}
\end{align*}
with $f(x,y):= \frac{1}{2}(x^2+y^2)-\frac{1}{12}(x^4+y^4)$ is defined. This metric corresponds to the surface induced by the embedding $\big(x,y\big)\mapsto \big(x,y,f(x,y)\big)$ and its exact Gauss curvature is given by
\begin{align*}
\Gauss(\gex)(x,y) = \frac{81(1-x^2)(1-y^2)}{(9+x^2(x^2-3)^2+y^2(y^2-3)^2)^2}.
\end{align*}
The embedded surface and Gauss curvature are depicted in Figure~\ref{fig:embedd_surf_gauss_curv}.

\begin{figure}
	\centering
	\includegraphics[width=0.4\textwidth]{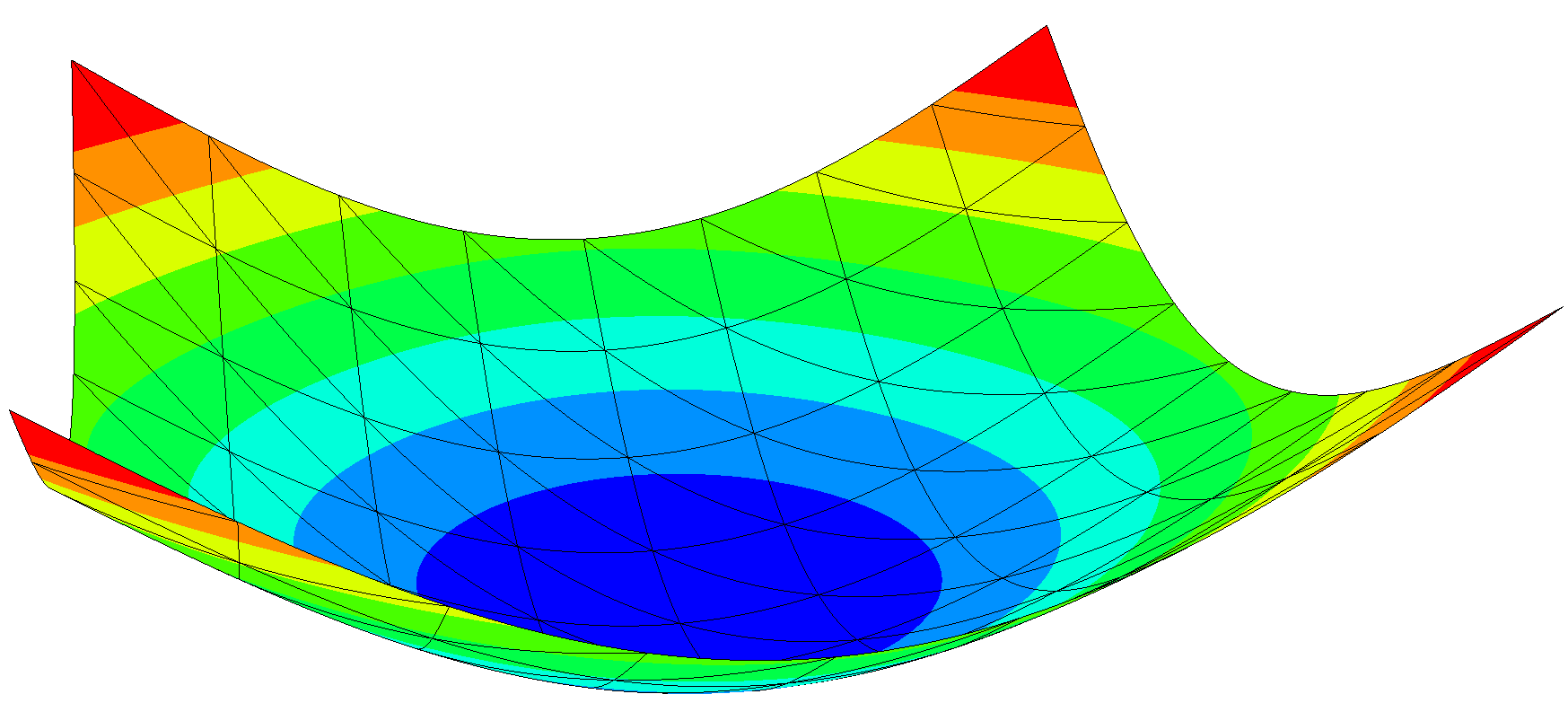}\hspace*{1cm}
	\includegraphics[width=0.3\textwidth]{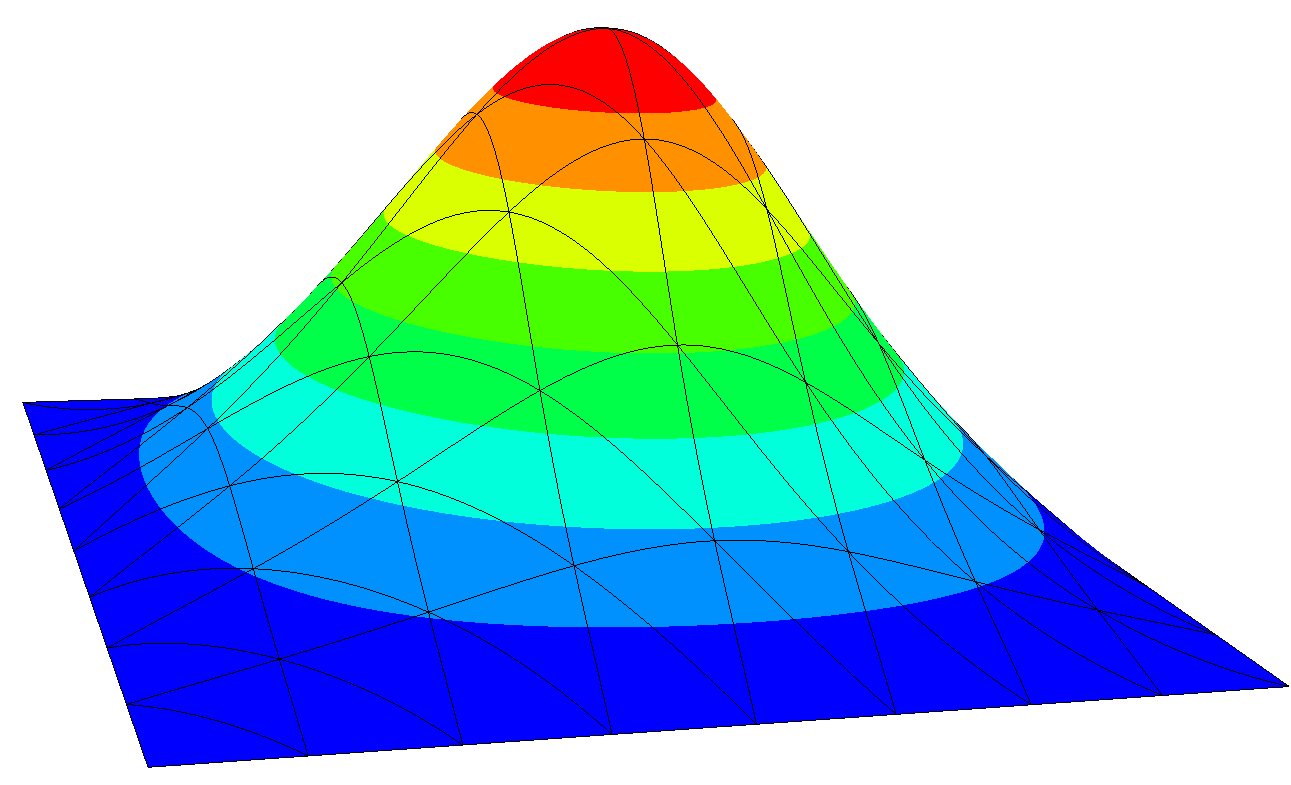}
	\caption{Left: Embedded surface, color indicates to $z$-component. Right: Exact Gauss curvature as graph over the domain $\Omega$.}
	\label{fig:embedd_surf_gauss_curv}
\end{figure}

To test also the case of non-homogeneous Dirichlet and Neumann boundary conditions we consider only one quarter $\Omega=(0,1)\times(0,1)$ and define the right and bottom boundaries as Dirichlet and the remaining parts as Neumann boundary. To avoid possible super-convergence properties due to a structured grid, we perturb all internal points of the triangular mesh by a uniform distribution in the range $[-\frac{h}{4},\frac{h}{4}]$, $h$ denoting the maximal mesh-size of the originally generate mesh. The geodesic curvature on the left boundary is exactly zero, whereas on the top boundary we compute
\begin{align*}
\kappa_g(\gex)|_{\Gamma_{\mathrm{left}}} = 0,\qquad\kappa_g(\gex)|_{\Gamma_{\mathrm{top}}} = \frac{-27(x^2 - 1)y(y^2 - 3)}{(x^2(x^2 - 3)^2 + 9)^{3/2}\sqrt{x^2(x^2 - 3)^2 + y^2(y^2 - 3)^2 + 9}}.
\end{align*}
The vertex expressions $K_V$ at the vertices of the Neumann boundary can directly be computed by measuring the angle $\arccos( \frac{\gex(\tv_1, \tv_2)}{\|\tv_1\|_{\gex}\|\tv_2\|_{\gex}})$.

\begin{figure}
	\centering
	\includegraphics[width=0.28\textwidth]{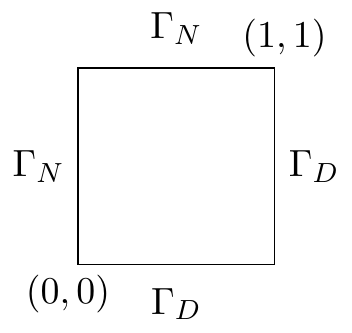}\hfill
	\includegraphics[width=0.28\textwidth]{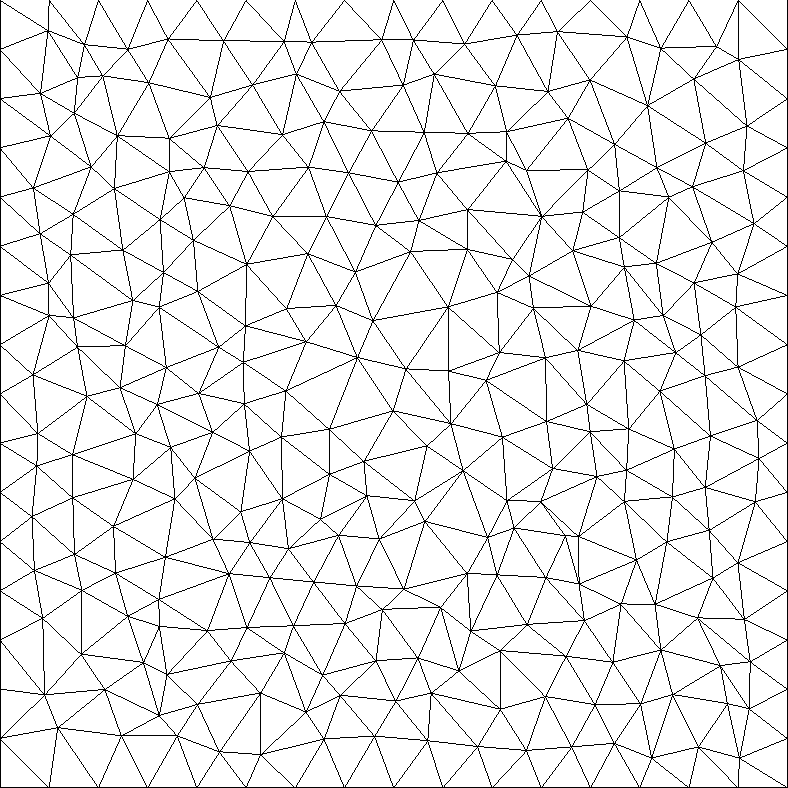}\hfill
	\includegraphics[width=0.28\textwidth]{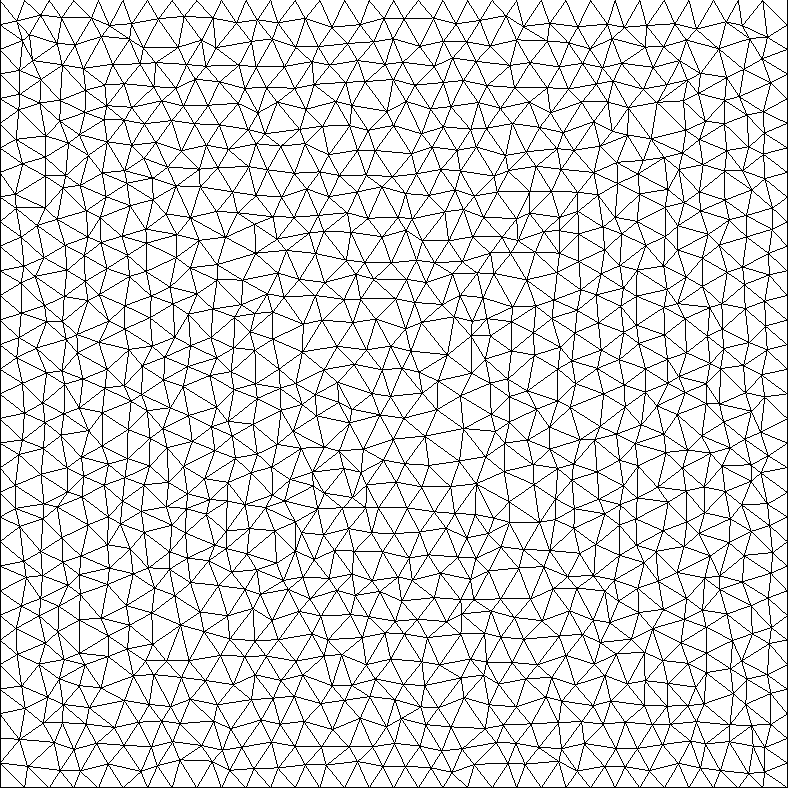}
	\caption{Left: Domain with Dirichlet and Neumann boundaries. Middle and right: perturbed unstructured triangular mesh grids.}
	\label{fig:domain_perturbed_grids}
\end{figure}

To illustrate our theorems, we must use
$\gh=\RegInt[k]\gex$.  In implementing the Regge interpolant, the
moments on the edges have to coincide exactly: see \eqref{eq:RegInt}.
To this end,  we use a high enough integration rule for interpolating $\gex$ for  minimizing the numerical integration errors. 

We compute and report the curvature error in the  $\Ltwo$-norm, namely  $\|\Gauss(\gex)-\Gauss_h(\gh)\|_{\Ltwo[]}$. We also report the  $\Hmone$-norm of the error. It can be computed by  solving for  $w\in\Honez[]$ such that 
$-\Delta w= \Gauss(\gex)-\Gauss_h(\gh)$ and observing that 
\begin{align*}
\|\Gauss(\gex)-\Gauss_h(\gh)\|_{\Hmone[]} = \|w\|_{\Hone[]}.
\end{align*}
Of course the right hand side can generally be computed only
approximately.  To avoid extraneous errors, we approximate $w$ using
Lagrange finite elements of two degrees more, i.e.,
$w_h\in \VV_h^{k+3}$ when  $\Gauss_h(\gh)\in \VV_h^{k+1}$.

\begin{figure}
	\centering

\resizebox{0.32\textwidth}{!}{
\begin{tikzpicture}
\begin{loglogaxis}[
legend style={at={(0,0)}, anchor=south west},
xlabel={ndof},
ylabel={error},
ymajorgrids=true,
grid style=dotted,
]
\addlegendentry{$\|\Gauss(\gex)-\Gauss_h(\gh)\|_{L^2}$}
\addplot[color=red, mark=*] coordinates {
	( 4 , 0.09326784925256802 )
	( 8 , 0.06345322388830371 )
	( 26 , 0.02595867508861117 )
	( 89 , 0.01703732023484835 )
	( 338 , 0.012654767367178849 )
	( 1246 , 0.012565878317247378 )
	( 4879 , 0.011545048826398747 )
};
\addlegendentry{$\|\Gauss(\gex)-\Gauss_h(\gh)\|_{H^{-1}}$}
\addplot[color=blue, mark=o] coordinates {
	( 4 , 0.012489863849611824 )
	( 8 , 0.006365298755469182 )
	( 26 , 0.0014626927103760425 )
	( 89 , 0.0005708093190676403 )
	( 338 , 0.00023759455999740632 )
	( 1246 , 0.00011550881569255937 )
	( 4879 , 5.3239285713173135e-05 )
};


\def\scal{0.01}
\addplot[dashed, color=black] coordinates {
	( 4 , \scal*0.5 )
	( 8 , \scal*0.3535533905932738 )
	( 26 , \scal*0.19611613513818404 )
	( 89 , \scal*0.105999788000636 )
	( 338 , \scal*0.05439282932204212 )
	( 1246 , \scal*0.028329634983503677 )
	( 4879 , \scal*0.014316425279852692 )
};

\end{loglogaxis}
\node (A) at (5, 1.85) [] {$\mathcal{O}(h)$};
\end{tikzpicture}}
\resizebox{0.32\textwidth}{!}{
\begin{tikzpicture}
\begin{loglogaxis}[
legend style={at={(0,0)}, anchor=south west},
xlabel={ndof},
ylabel={error},
ymajorgrids=true,
grid style=dotted,
]
\addlegendentry{$\|\Gauss(\gex)-\Gauss_h(\gh)\|_{L^2}$}
\addplot[color=red, mark=*] coordinates {
	( 9 , 0.03794599735078921 )
	( 21 , 0.013423836922395502 )
	( 85 , 0.006923288837855359 )
	( 321 , 0.002704214400767389 )
	( 1285 , 0.0012252009477189173 )
	( 4853 , 0.0005526100376963728 )
	( 19257 , 0.0002640804715243997 )
};
\addlegendentry{$\|\Gauss(\gex)-\Gauss_h(\gh)\|_{H^{-1}}$}
\addplot[color=blue, mark=o] coordinates {
	( 9 , 0.00378103414321965 )
	( 21 , 0.0011065158571022778 )
	( 85 , 0.0002327366154983381 )
	( 321 , 4.4838104279259916e-05 )
	( 1285 , 9.976071810983794e-06 )
	( 4853 , 2.3747422899575113e-06 )
	( 19257 , 5.672363282613678e-07 )
};


\def\scal{0.01}
\addplot[dashed, color=black] coordinates {
	( 9 , \scal*0.3333333333333333 )
	( 21 , \scal*0.2182178902359924 )
	( 85 , \scal*0.10846522890932808 )
	( 321 , \scal*0.055814557218594754 )
	( 1285 , \scal*0.027896417632583534 )
	( 4853 , \scal*0.01435472425324029 )
	( 19257 , \scal*0.00720618960436169 )
};

\def\scal{0.008}
\addplot[dashed, color=black] coordinates {
	( 9 , \scal*0.1111111111111111 )
	( 21 , \scal*0.047619047619047616 )
	( 85 , \scal*0.011764705882352941 )
	( 321 , \scal*0.003115264797507788 )
	( 1285 , \scal*0.0007782101167315176 )
	( 4853 , \scal*0.000206058108386565 )
	( 19257 , \scal*5.192916861401049e-05 )
};

\end{loglogaxis}
\node (A) at (5, 2.6) [] {$\mathcal{O}(h)$};
\node (A) at (5.7, 1.6) [] {$\mathcal{O}(h^2)$};
\end{tikzpicture}}
\resizebox{0.32\textwidth}{!}{
\begin{tikzpicture}
\begin{loglogaxis}[
legend style={at={(0,0)}, anchor=south west},
xlabel={ndof},
ylabel={error},
ymajorgrids=true,
grid style=dotted,
]
\addlegendentry{$\|\Gauss(\gex)-\Gauss_h(\gh)\|_{L^2}$}
\addplot[color=red, mark=*] coordinates {
	( 16 , 0.03283537542914546 )
	( 40 , 0.01279005887266998 )
	( 178 , 0.0017451911516211448 )
	( 697 , 0.0003817721974622066 )
	( 2842 , 9.413032794245751e-05 )
	( 10822 , 2.5557308283167545e-05 )
	( 43135 , 6.270117247445405e-06 )
};
\addlegendentry{$\|\Gauss(\gex)-\Gauss_h(\gh)\|_{H^{-1}}$}
\addplot[color=blue, mark=o] coordinates {
	( 16 , 0.002404947712051608 )
	( 40 , 0.00064688676729271 )
	( 178 , 3.099353061355925e-05 )
	( 697 , 3.7108450894523983e-06 )
	( 2842 , 4.336048447733595e-07 )
	( 10822 , 6.091280561801814e-08 )
	( 43135 , 7.360410691406379e-09 )
};


\def\scal{0.15}
\addplot[dashed, color=black] coordinates {
	( 16 , \scal*0.0625 )
	( 40 , \scal*0.025 )
	( 178 , \scal*0.0056179775280898875 )
	( 697 , \scal*0.0014347202295552368 )
	( 2842 , \scal*0.0003518648838845883 )
	( 10822 , \scal*9.240436148586214e-05 )
	( 43135 , \scal*2.318303002202388e-05 )
};

\def\scal{0.04}
\addplot[dashed, color=black] coordinates {
	( 16 , \scal*0.015625 )
	( 40 , \scal*0.003952847075210474 )
	( 178 , \scal*0.0004210852185370008 )
	( 697 , \scal*5.434390380764492e-05 )
	( 2842 , \scal*6.600303250212253e-06 )
	( 10822 , \scal*8.882571368875902e-07 )
	( 43135 , \scal*1.1162341275778122e-07 )
};

%
\end{loglogaxis}
\node (A) at (5, 2.55) [] {$\mathcal{O}(h^2)$};
\node (A) at (5.8, 1.6) [] {$\mathcal{O}(h^3)$};
\end{tikzpicture}}

	\caption{Convergence of Gauss curvature with respect to number of degrees of freedom (ndof) in different norms for Regge elements $\gh\in\Regge_h^k$ of order $k=0,1,2$.}
	\label{fig:conv_plot_curvature}
\end{figure}
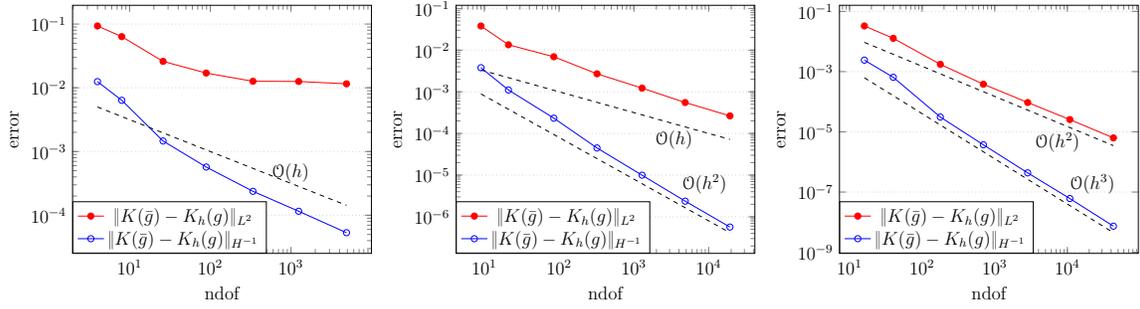

We start by approximating $\gex$ by the lowest order piecewise constant Regge elements $\gh\in\Regge_h^0$. As shown in Figure~\ref{fig:conv_plot_curvature} (left), we do not obtain convergence in the $\Ltwo$-norm, but do obtain linear convergence in the weaker $\Hmone$-norm, in agreement with Theorem~\ref{thm:convergence_curvature_hm1-rev}. When increasing the approximation order of Regge elements to linear and quadratic polynomials we observe the appropriate increase of the convergence rates: see  Figure~\ref{fig:conv_plot_curvature} (middle and right), confirming that the results stated in Theorem~\ref{thm:convergence_curvature_hm1-rev} and Corollary~\ref{cor:convergence_curvature_l2_hl} are practically sharp. In Figure~\ref{fig:plot_curvature} snap-shots of the approximated Gauss curvature are shown.

\begin{figure}
	\centering
	\includegraphics[width=0.28\textwidth]{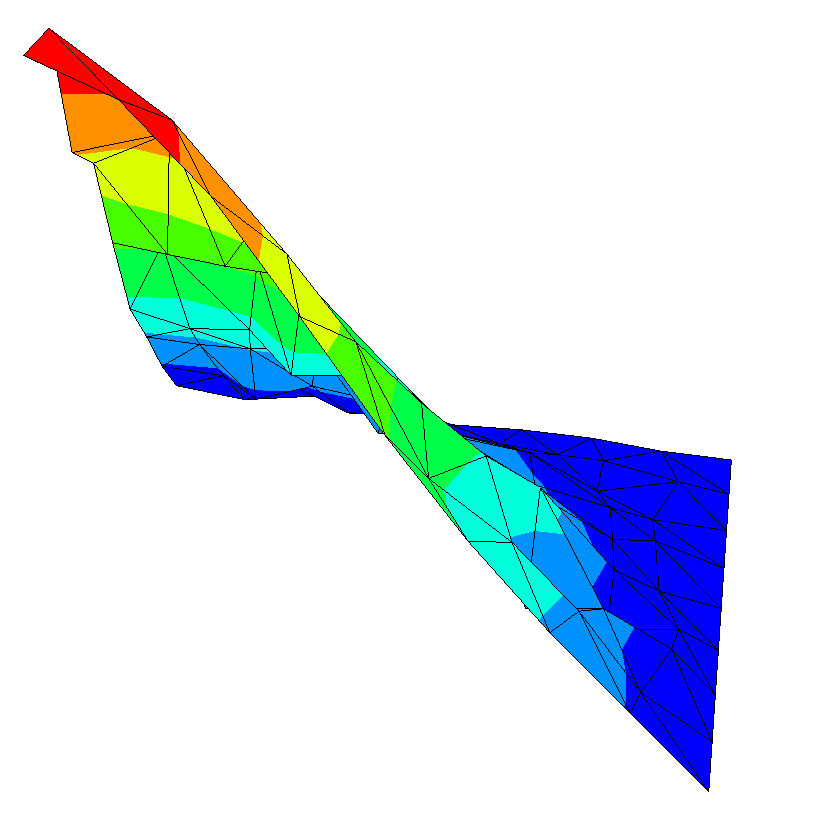}
	\includegraphics[width=0.28\textwidth]{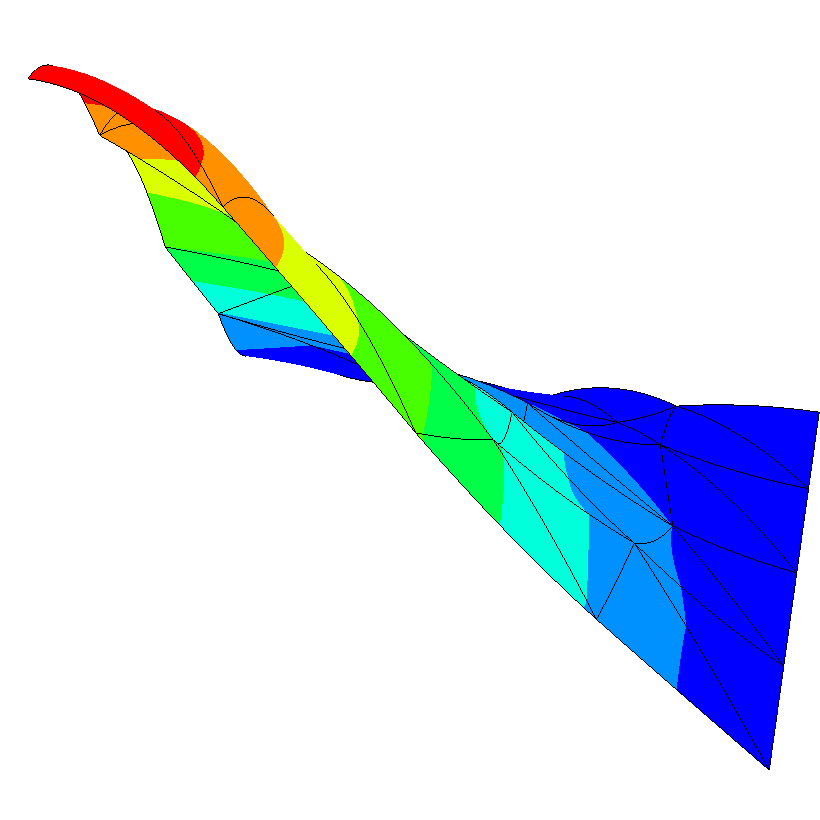}
	\includegraphics[width=0.28\textwidth]{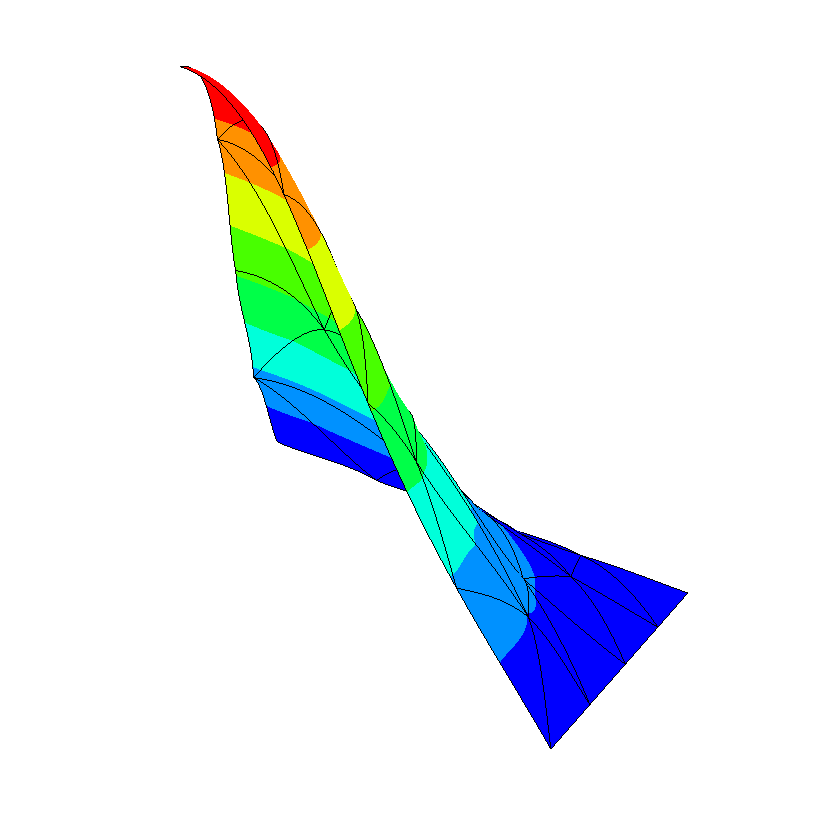}
	\caption{From left to right: Approximated Gauss curvature with Regge elements of order $k=0,1,2$.}
	\label{fig:plot_curvature}
\end{figure}

\subsection{Connection 1-form approximation}
\label{subsec:numex_oneform}
For testing the convergence of the connection 1-form we consider the same metric tensor $\gex$, same domain, and the same type of boundary conditions as before. The exact connection 1-form $\OneForm$ depicted in Figure~\ref{fig:ex_oneform} is given by 
\begin{align*}
&\OneForm(\gex) = -\frac{1}{2}\mat{ (\GBasis_2)^\trans \gex\nabex_{\EuclBasis_1}\GBasis_1-(\GBasis_1)^\trans \gex\nabex_{\EuclBasis_1}\GBasis_2\\(\GBasis_2)^\trans \gex\nabex_{\EuclBasis_2}\GBasis_1-(\GBasis_1)^\trans \gex\nabex_{\EuclBasis_2}\GBasis_2}=\mat{-\frac{3y(y^2-3) (x^2-1) (9 + 3 \sqrt{A} + y^2 (y^2-3)^2 + x^2 (x^2-3)^2)}{\sqrt{A}((A+9)\sqrt{A} + 6A)}\\
\frac{3x(x^2-3) (y^2-1) (9 + 3 \sqrt{A} + x^2 (x^2-3)^2 + y^2 (y^2-3)^2)}{\sqrt{A}((A+9)\sqrt{A} + 6A)}},	
\end{align*}
where
 $A=y^6-6y^4+9+x^6-6x^4+9x^2$, 
$\EuclBasis_i\in\R^2$ are the Euclidean basis vectors, $\GBasis_i=\gex^{-\frac{1}{2}}\EuclBasis_i$  in accordance with \eqref{eq:28}, and the covariant derivative $(\nabex_XY)^i=((\nabex X)Y)^i+\bar\Gamma_{jk}^{i}X^jY^k$ with  $\bar\Gamma_{ij}^k$ denoting the Christoffel symbol of second kind with respect to $\gex$. 

\begin{figure}
	\centering
	\includegraphics[width=0.28\textwidth]{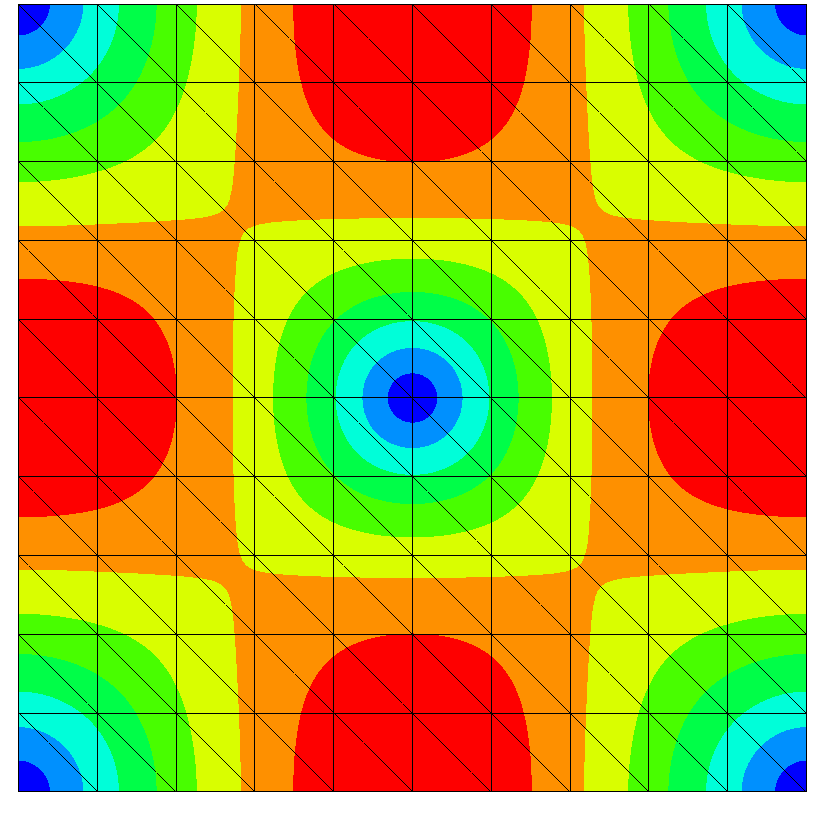}
	\includegraphics[width=0.28\textwidth]{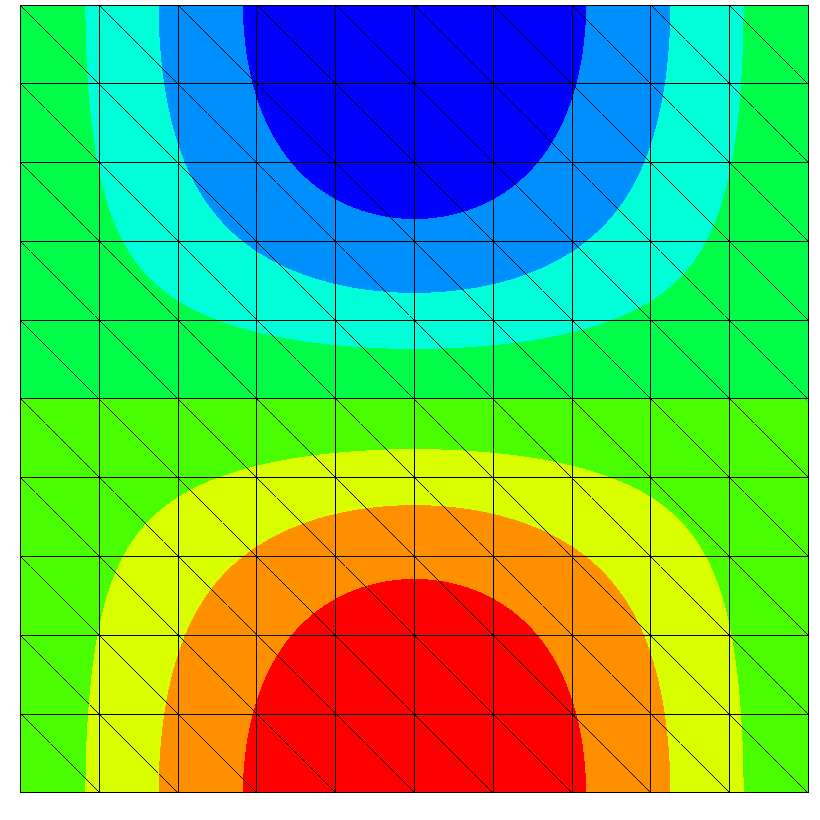}
	\includegraphics[width=0.28\textwidth]{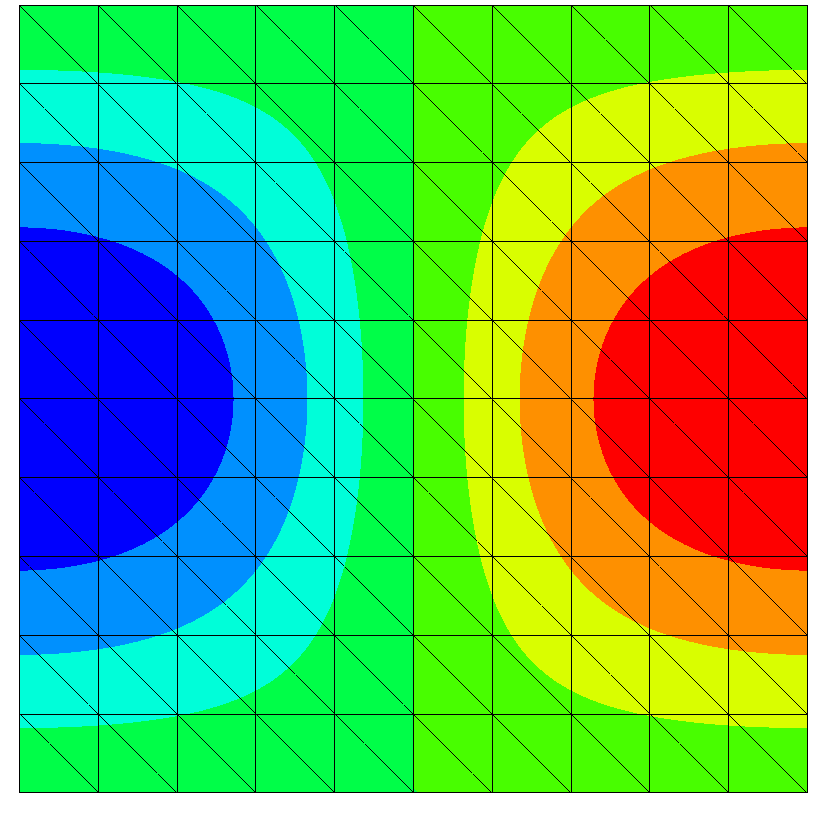}
	\caption{Exact connection 1-form $\OneForm$. Left: $\|\OneForm\|_{2}$, middle: $x$-component, right: $y$-component.}
	\label{fig:ex_oneform}
\end{figure}


For the results shown in Figure~\ref{fig:conv_plot_oneform} we use $\BDM$ and also Raviart--Thomas \cite{RT77}$\RT$ elements. The optimal $\Ltwo$-convergence rates stated in Theorem~\ref{thm:convergence_connectionform_l2-rev} are confirmed for $k>0$ to be sharp when using $\BDM$ elements for $k>0$. If we increase the test-space, however, to $\RT$ elements, which additionally include specific polynomials of one order higher than $\BDM$, one order of convergence is lost, compare also Figure~\ref{fig:conv_plot_oneform}. Note, that to construct the finite element space $\WW_h^0 = \BDM^0,$ we consider the lowest order Raviart-Thomas elements $\RT^0$ and lock the linear part by enforcing that $\div(\RT^0)=0$.  As depicted in Figure~\ref{fig:conv_plot_oneform} (left) the discrete solution converges linearly at the beginning, however, after some refinements the error stagnates. This is in accordance with the explanation provided in   Remark~\ref{rem:conv_one_form_lo}. 
Solution snap-shots are displayed in Figure~\ref{fig:plot_oneform}.

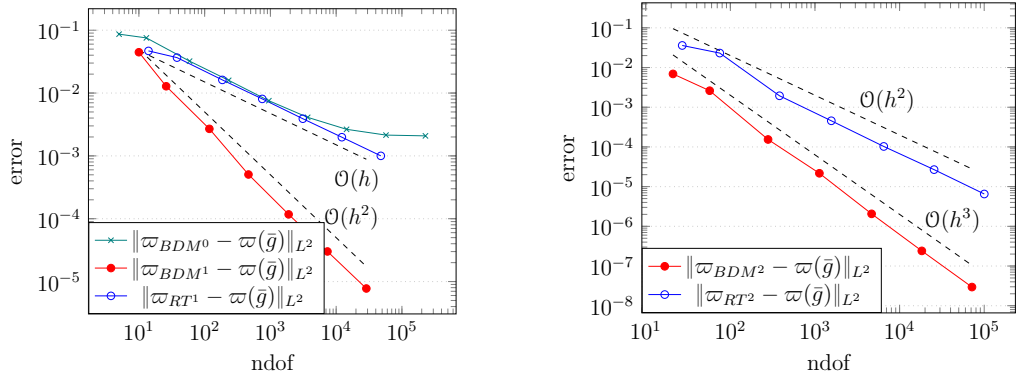
\begin{figure}
	\centering
\resizebox{0.4\textwidth}{!}{
	\begin{tikzpicture}
	\begin{loglogaxis}[
	legend style={at={(0,0)}, anchor=south west},
	xlabel={ndof},
	ylabel={error},
	ymajorgrids=true,
	grid style=dotted,
	]
	
	\addlegendentry{$\|\OneForm_{BDM^0}-\OneForm(\gex)\|_{L^2}$}
	\addplot[color=teal, mark=x] coordinates {
		( 5 , 0.08646296750973805 )
		( 13 , 0.07545373020294766 )
		( 59 , 0.03218940817938298 )
		( 232 , 0.01593162473683139 )
		( 947 , 0.007547356427023641 )
		( 3688 , 0.004105228213658372 )
		( 14387 , 0.002646875412634397 )
		( 57274 , 0.002142258788715815 )
		( 227951 , 0.0020824589657696606 )
	};
	
	\addlegendentry{$\|\OneForm_{BDM^1}-\OneForm(\gex)\|_{L^2}$}
	\addplot[color=red, mark=*] coordinates {
		( 10 , 0.044606773669639437 )
		( 26 , 0.012816980696466945 )
		( 118 , 0.002693393345664789 )
		( 464 , 0.0005055668489533827 )
		( 1894 , 0.00011725318781077317 )
		( 7376 , 3.0130332195958852e-05 )
		( 28774 , 7.756723463741678e-06 )
	};
	\addlegendentry{$\|\OneForm_{RT^1}-\OneForm(\gex)\|_{L^2}$}
	\addplot[color=blue, mark=o] coordinates {
		( 14 , 0.046839647803022265 )
		( 38 , 0.036628804084490575 )
		( 186 , 0.016183746808327442 )
		( 752 , 0.008066274736456732 )
		( 3114 , 0.0039026644612921626 )
		( 12208 , 0.0019826037382274656 )
		( 47786 , 0.001000911631927249 )
	};

	\def\scal{0.15}
	\addplot[dashed, color=black] coordinates {
		( 10 , \scal*0.31622776601683794 )
		( 26 , \scal*0.19611613513818404 )
		( 118 , \scal*0.09205746178983233 )
		( 446 , \scal*0.047351372381037836 )
		( 1888 , \scal*0.023014365447458083 )
		( 7208 , \scal*0.011778571185788638 )
		( 28744 , \scal*0.00589829375244162 )
	};
	
	\def\scal{0.5}
	\addplot[dashed, color=black] coordinates {
		( 10 , \scal*0.1 )
		( 26 , \scal*0.038461538461538464 )
		( 118 , \scal*0.00847457627118644 )
		( 446 , \scal*0.002242152466367713 )
		( 1888 , \scal*0.0005296610169491525 )
		( 7208 , \scal*0.00013873473917869035 )
		( 28744 , \scal*3.4789869190091844e-05 )
	};
	
	\end{loglogaxis}
	\node (A) at (5, 2.5) [] {$\mathcal{O}(h)$};
	\node (A) at (4.9, 1.8) [] {$\mathcal{O}(h^2)$};
	\end{tikzpicture}
}
\hspace*{1cm}\resizebox{0.4\textwidth}{!}{
	\begin{tikzpicture}
	\begin{loglogaxis}[
	legend style={at={(0,0)}, anchor=south west},
	xlabel={ndof},
	ylabel={error},
	ymajorgrids=true,
	grid style=dotted,
	]
	\addlegendentry{$\|\OneForm_{BDM^2}-\OneForm(\gex)\|_{L^2}$}
	\addplot[color=red, mark=*] coordinates {
		( 21 , 0.006903352743413215 )
		( 57 , 0.0026052969807374314 )
		( 279 , 0.00015402121395687588 )
		( 1128 , 2.1679903464834585e-05 )
		( 4671 , 2.061990331602166e-06 )
		( 18312 , 2.4057008976445295e-07 )
		( 71679 , 2.9516610395181677e-08 )
	};
	\addlegendentry{$\|\OneForm_{RT^2}-\OneForm(\gex)\|_{L^2}$}
	\addplot[color=blue, mark=o] coordinates {
		( 27 , 0.035942528357242855 )
		( 75 , 0.023114214504283158 )
		( 381 , 0.001941782709932394 )
		( 1560 , 0.00045532775531465164 )
		( 6501 , 0.00010277000201441132 )
		( 25560 , 2.6814147549149552e-05 )
		( 100197 , 6.521960675976014e-06 )
	};

	\def\scal{2}
	\addplot[dashed, color=black] coordinates {
		( 21 , \scal*0.047619047619047616 )
		( 57 , \scal*0.017543859649122806 )
		( 279 , \scal*0.0035842293906810036 )
		( 1083 , \scal*0.0009233610341643582 )
		( 4656 , \scal*0.0002147766323024055 )
		( 17892 , \scal*5.58909009613235e-05 )
		( 71604 , \scal*1.3965700240210044e-05 )
		
	};
	
	\def\scal{2}
	\addplot[dashed, color=black] coordinates {
		( 21 , \scal*0.010391328106475828 )
		( 57 , \scal*0.0023237409773070945 )
		( 279 , \scal*0.000214582200748835 )
		( 1083 , \scal*2.8058039033368603e-05 )
		( 4656 , \scal*3.1476072903809557e-06 )
		( 17892 , \scal*4.1784159029757865e-07 )
		( 71604 , \scal*5.219081444831702e-08 )
	};
	
	\end{loglogaxis}
	\node (A) at (4.5, 3.9) [] {$\mathcal{O}(h^2)$};
	\node (A) at (5.7, 1.7) [] {$\mathcal{O}(h^3)$};
	\end{tikzpicture}}
	\caption{$\Ltwo$-convergence of connection 1-form error with BDM and RT elements. Left: $\BDM^0$, $\BDM^1$, and $\RT^1$. Right: $\BDM^2$ and $\RT^2$.}
	\label{fig:conv_plot_oneform}
\end{figure}

\begin{figure}
	\centering
	\includegraphics[width=0.28\textwidth]{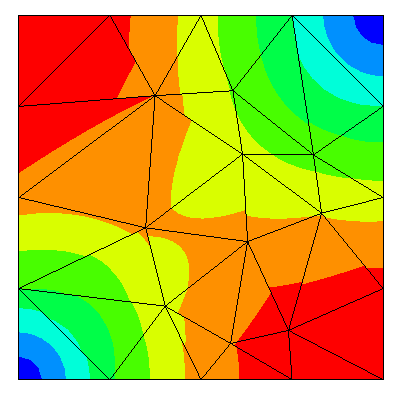}
	\includegraphics[width=0.28\textwidth]{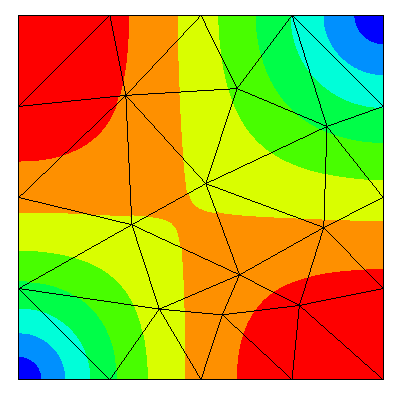}
	\includegraphics[width=0.28\textwidth]{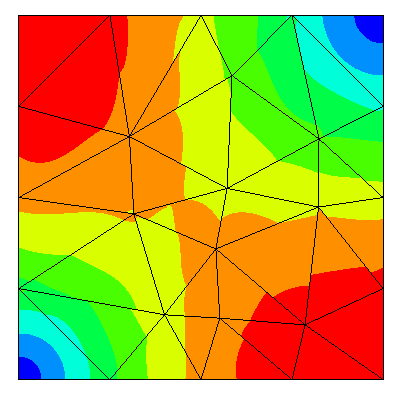}
	\caption{Norm of approximated connection 1-form. Left: $\BDM^1$, middle: $\BDM^2$, right: $\RT^2$.}
	\label{fig:plot_oneform}
\end{figure}


  
\appendix

\section{Rationale for the $g$-normal continuity}
\label{sec:rationale-g-normal}

This section briefly presents a justification for  the usage of $g$-normal
continuous vector fields in Definition~\ref{def:curlg-extended}.
We show that there are smooth functions $\vphi$ in $\Xmo \Moo$ approaching a $g$-normal continuous 
$W \in \Wo_g(\T)$ in such a way that the right hand side
of~\eqref{eq:dist-curlg-firstdef} converges to that
of~\eqref{eq:curlg-extended-2}.
For any mesh vertex $V \in \V$, let
$B_\veps(V) = \{ q \in M: d_g(q, V) \le \veps\}.$ Then put
$D_\veps = \cup_{V \in \V} B_\veps(V)$ and
$\Moo_\veps = M \setminus D_\veps$.  Let
$U_i, \Poo_i: U_i \to \R^2$ denote a chart of the glued smooth
structure. In the parameter domain $\Poo_i(U_i)$, using the Euclidean
divergence operator, define
$W^p(\div, \Poo_i(U_i)) = \{ w \in L^p( \Poo_i(U_i)),\; \div (w) \in
L^p( \Poo_i(U_i))\}$ for any $1\le p < \infty$ with its natural
Euclidean norm.  This norm and the duality pairings defined
in~\eqref{eq:dist-curlg-firstdef} and~\eqref{eq:curlg-extended-2}
feature in the next result.

\begin{prop}
  \label{prop:dens}
  Let $\sigma \in \Regge(\T)$ and $W \in \Wo_g(\T)$. For any given
  $\veps_1>0$, there exists a $p>2$, an $\veps_2>0$, finitely many
  charts $\{(U_i, \Poo_i): i \in I\}$ covering $\Moo_{\veps_2}$ in the
  glued smooth structure, a partition of unity $\psi_i$ subordinate to
  $U_i$, and a smooth $\vphi \in \Xmo \Moo$ such that
  $\vphi = \sum_{i \in I} \vphi_i$ with support of $\vphi_i$  contained
  in $U_i$, satisfies 
  \begin{gather}
    \label{eq:5A}
    \big|\act{\curl_g\sigma, \vphi}_{\Xmo \Moo}
    -
    \act{\curl_g\sigma, W}_{\Wo_g(\T)}\big|
    \le  \veps_1, \quad \text { and } 
    \\ \label{eq:3A}
    \| (\Poo_i)_* ( \vphi_i - \psi_i W ) \|_{W^p(\div, \Poo_i(U_i \cap \Moo_{\veps_2}))}
    \le \veps_1, \quad \text{ for all } i \in I.
  \end{gather}
\end{prop}
\begin{proof}
  As a first step, we zero out the vector field $W$ near vertices. For any
  $\veps>0$, let $0 \le \chi_\veps \le 1$ be a smooth cutoff function
  satisfying $\chi_\veps \equiv 1$ in $\Moo_{\veps}$ and
  $\chi_\veps \equiv 0$ in $D_{\veps/2}$, and let
  \[
    r_\veps(W) = \sum_{T \in \T} \int_{(T \cap D_\veps, g)} g(W, W)^{1/2} +
    \int_{(\d T \cap D_\veps, g)} g(W, W)^{1/2}.
  \]
  Since $\sigma$ is piecewise smooth, there is a constant $C_\sigma$
  depending only on $\sigma$ (independent of $\veps$) such that
  $ | \act{\curl_g\sigma, \; W - \chi_\veps W}_{\Wo_g(\T)} | \le
  C_\sigma \, r_\veps(W).$ Since $W$ is piecewise smooth, $r_\veps(W)$
  approaches zero as $\veps \to 0$. Hence for the given $\veps_1$,
  there is a $\veps_2>0$ such that
  \begin{equation}
    \label{eq:20A}
    \left| \act{\curl_g\sigma, \; W - \chi_\veps W}_{\Wo_g(\T)} \right| \le
    C_\sigma \, r_\veps(W) \le \frac {\veps_1}{2}
    \qquad \text{ for all } \veps \le \veps_2.
  \end{equation}

  Next, to approximate $\chi_{\veps_2} W$ by a smooth vector field, we
  use the precompactness of $\Moo_{\veps_2/2}$ to extract a finite
  subcover from the maximal atlas of $\Moo$. Denoting the resulting
  finitely many charts by $U_i, \Poo_i$, let  $\psi_i$ be  a partition of
  unity subordinate to the cover $U_i$. We focus on a $U_i$
  that intersects an edge $E$ (since the other cases are easier) and
  use the accompanying notation in~\eqref{eq:jump-def}.  By the
  previous discussion of the coordinate
  construction~\eqref{eq:Kosovskii-coords}, $\doo^1 = \gn_+$ and
  $\doo^2 = \gt_+$ along $E$, so the expansion $W = W^i \doo_i$ is
  $g$-orthonormal and 
  \begin{equation}
    \label{eq:30A}
    g(W|_{T_\pm}, \gn_\pm) = \pm W^1|_{T_\pm}  \text { on } E\cap U_i.
  \end{equation}
  The $g$-normal continuity $\jump{g(W, \gn)} = 0$ implies that
  $W_i = \psi_i \chi_{\veps_2} W$ pushed forward to the parameter
  domain, namely $(\Poo_i)_*W_i$, has continuous normal component
  across  $\Poo_i(E)$, a subset of the axis
  $\Yoo = \{ (\xoo^1, \xoo^2) \in \R^2: \xoo^1=0\}.$ Hence
  $(\Poo_i)_*W_i$ is in $\mathring{W}^p(\div, \Poo_i(U_i)),$ the
  subspace of ${W}^p(\div, \Poo_i(U_i))$ with zero normal traces on
  the boundary of $ \Poo_i(U_i)$.  By a well known density result in
  the Euclidean domain, there exists an infinitely smooth compactly
  supported vector field $\breve{\vphi}_i$ on $\Poo_i(U_i)$ that is
  arbitrarily close to $(\Poo_i)_*W_i$, so letting
  $\vphi_i = (\Poo_i^{-1})_*\breve{\vphi}_i$, we have
  \begin{equation}
    \label{eq:50A}
    \| (\Poo_i)_* ( \vphi_i - W_i ) \|_{W^p(\div, \Poo_i(U_i))}
    =
    \| \breve{\vphi}_i - (\Poo_i)_*  W_i  \|_{W^p(\div, \Poo_i(U_i))}
    \le
    \veps_1.    
  \end{equation}
  Moreover, since $\chi_{\veps_{2}} \equiv 1$ in
  $U_i \cap \Moo_{\veps_2}$, the functions $W_i$ and $\psi_i W$
  coincide there, so~\eqref{eq:3A} follows. Constructing such
  $\vphi_i$ on every $U_i$, put $\vphi = \sum_{i\in I} \vphi_i$.

  To  prove~\eqref{eq:5A}, in view of~\eqref{eq:20A}, it suffices to show that
  \begin{equation}
    \label{eq:40A}
    \left| \act{\curl_g\sigma, \; \chi_{\veps_2} W}_{\Wo_g(\T)}
      - \act{\curl_g\sigma, \vphi}_{\Xmo \Moo} \right| \le
    \frac {\veps_1}{2}.
  \end{equation}
  Obviously, the element contributions in the difference above,
  $ \int_\T (\curl_g \sigma )(\chi_{\veps_2} W - \vphi), $ can be made
  arbitrarily small by~\eqref{eq:50A}, revising the choice of
  $\vphi_i$ if needed. For the element boundary terms,
  \begin{equation}
    \label{eq:dT-trace}
    \int_{\d\T} g( \chi_{\veps_2} W - \vphi, \gn)\, \sigma(\gn, \gt)
    = \sum_{i \in I}
    \int_{\d\T} g( W_i - \vphi_i, \gn)\, \sigma(\gn, \gt),    
  \end{equation}
  we focus, as before, on a neighborhood $U_i$ intersecting an edge
  $E = \d T_- \cap \d T_+$.  On $\d T_+$, by~\eqref{eq:30A},
  $g( W_i - \vphi_i, \gn) = W_i^1 - \vphi_i^1$ yields the normal
  component of the pushforward $(\Poo_i)_*(W_i - \vphi_i)$ on
  $\Yoo$-axis in the parameter domain.  Since the latter converges to
  zero in $W^p(\div, \Poo_i(U_i))$ by~\eqref{eq:50A}, its normal trace
  converges to zero in $W^{-1/p, p}(\Poo_i(E \cap U_i))$.  Choose
  $p>2$ and $q$ such that $1/p + 1/q=1$.  When $\sigma(\gn, \gt)$ is mapped to
  $\Poo_i(E \cap U_i)$ and extended to $\Yoo$-axis by zero, is in
  $W^{1-1/q, q}(\Yoo)$ since $1-1/q < 1/2$. Hence the contribution
  from $U_i \cap E$ to the right hand side of~\eqref{eq:dT-trace}
  vanishes. Repeating this argument on other charts, \eqref{eq:40A} is
  proved.
\end{proof}

\section{Angle computation for connection approximation}
\label{sec:oneform_algorithm}

In this section we discuss and present a stable angle computation used
in the connection 1-form approximation. Since $g_{ij}$ is in general
discontinuous across an edge $E$ (in the computational coordinates
$x^i$), we cannot use it to directly compute the angle between frame
vectors on two different triangles. Instead, we compute angles the
frame makes with an intermediate vector element by element and then use it
to compute $\Theta^E$ as explained below.


Before going into details we make the following  observation: if the metric $g$ approximates a smooth metric $\gex$, we
expect that the frame $e_i$ fixed by~\eqref{eq:28} will be such
that their restrictions to adjacent elements,
$(\GBasis_{+,1},\GBasis_{+,2})$ and $(\GBasis_{-,1},\GBasis_{-,2})$,
will differ by a small angle, say less than $\pi$. This is the case in Figure~\ref{fig:inconsistent_jump} (left), where the angle difference $\Theta^E$ is negative and $|\Theta^E| < \pi$. 

On each interior mesh edge $E$, let $T_\pm, \gn_\pm, \gt_\pm$ be as
in~\eqref{eq:jump-def}, orient the edge $E$ by 
$\gt^E = \gt_+$, and  put $\gn^E = \gn_+$, $e_{\pm, i} = e_i|_{T_\pm}$. 
Let
$\Theta^E_\pm = \agl_g(e_{\pm, 1}, \pm \gn_\pm)$. Clearly, $\Theta^E_\pm$ can be computed using $g|_{T_\pm}$, specifically using its components $g_{ij}$  in the computational coordinates $x^i$ on either triangle.
Then  $\Theta^E=\Theta^E_+-\Theta^E_-$ in most cases (and certainly in the case illustrated in left drawing of Figure~\ref{fig:inconsistent_jump}).
In some cases however, such as that in the middle drawing of  Figure~\ref{fig:inconsistent_jump},
although $\Theta^E$ is negative and  $|\Theta^E| < \pi,$ the number 
$\Theta^E_+-\Theta^E_-$ is positive and larger than $\pi$. Thus setting
$\Theta^E = \Theta^E_+-\Theta^E_-$ 
would be incorrect and would lead to a bad numerical approximation of the connection 1-form. To cure this problem we change the choice of the starting angle on the fly. First, we compute
$\Theta^E_\pm = \agl_g(e_{\pm, 1}, \pm \gn_\pm)$
on every edge
as a pre-processing step. On each edge, set
\begin{equation}
\label{eq:31}
s_{E}  = \begin{cases}
+1, & \text{ if } |\Theta^E_+-\Theta^E_-| < \pi,\\
-1, & \text{ otherwise}.
\end{cases}  
\end{equation}
Then set $\tilde \Theta^E_\pm = \agl_g(e_{\pm, 1}, \pm s_E \gn_{\pm})$
and $\Theta^E = \tilde\Theta^E_+ - \tilde\Theta^E_-$. In other words,
we  compute after 
reversing the sign of the artifically introduced
$g$-normal vector $\gn_{\pm}$ on both the adjacent elements of an edge
if the modulus of the pre-computed angle is larger than $\pi$. This is illustrated in Figure~\ref{fig:inconsistent_jump} (right), where the sign change of the normal vector is depicted in red. The following computational formula is easy to prove.

\begin{figure}[h]
	\centering
%
%
%
%
%
%
%
%
%
	\begin{tikzpicture}[scale=1.5]
\draw[draw=black] (0,0.5) --(1.5,0)--(1.5,1.5)--cycle;
\draw[draw=black] (1.5,0)--(1.5,1.5)--(3.1,0.9)--cycle;

\draw[draw=blue,->] (1.2,0.65) to (0.7,0.65);
\draw[draw=black,->] (1.2,0.65) to (0.75,0.45);

\draw[draw=blue,->] (2.1,0.75) to (1.6,0.75);
\draw[draw=black,->] (2.1,0.75) to (1.75,0.4);

\draw[draw=teal, thick,->] (1.5,0.7) to (1.5,1.2);

\node (A) at (0.4, 0.1) [] {$T_+$};
\node (A) at (2.2, 0.1) [] {$T_-$};

\node[blue] (A) at (1.2, 0.8) [] {$\gn_+$};
\node[blue] (A) at (1.8, 0.95) [] {$-\gn_-$};

\node[teal] (A) at (1.3, 1.2) [] {$\gt^E$};

\node (A) at (1.2, 0.4) [] {$\GBasis_{+,1}$};
\node (A) at (2.3, 0.6) [] {$\GBasis_{-,1}$};

\centerarc[black, thick,-{Latex[length=1.3mm]}](1.2,0.65)(180:180+26:0.35);
\centerarc[black, thick,-{Latex[length=1.3mm]}](2.1,0.75)(180:180+46:0.35);

\node (A) at (0.5, 0.6) [] {$\Theta_+^E$};
\node (A) at (1.7, 0.2) [] {$\Theta_-^E$};
\end{tikzpicture}
	\begin{tikzpicture}[scale=1.5]
	\draw[draw=black] (0,0.5) --(1.5,0)--(1.5,1.5)--cycle;
	\draw[draw=black] (1.5,0)--(1.5,1.5)--(3.1,0.9)--cycle;
	
	\draw[draw=blue,->] (1.2,0.65) to (0.7,0.65);
	\draw[draw=black,->] (1.2,0.65) to (0.75,0.85);
	
	\draw[draw=blue,->] (2.1,0.75) to (1.6,0.75);
	\draw[draw=black,->] (2.1,0.75) to (1.75,0.4);

	\draw[draw=teal, thick,->] (1.5,0.7) to (1.5,1.2);
	
	\centerarc[black, thick,-{Latex[length=1.3mm]}](1.2,0.65)(180:180+335:0.35);
	\centerarc[black, thick,-{Latex[length=1.3mm]}](2.1,0.75)(180:180+45:0.35);
	
	
	
	\node (A) at (1.2, 0.78) [] {$\Theta_+^E$};
	\node (A) at (1.75, 0.2) [] {$\Theta_-^E$};
	
	\node (A) at (0.4, 0.1) [] {$T_+$};
	\node (A) at (2.2, 0.1) [] {$T_-$};
	
	\node[blue] (A) at (0.7, 0.45) [] {$\gn_+$};
	\node[blue] (A) at (1.8, 0.9) [] {$-\gn_-$};
	
	
	\node (A) at (0.6, 1.1) [] {$\GBasis_{+,1}$};
	\node (A) at (2.25, 0.58) [] {$\GBasis_{-,1}$};
	
	\node[teal] (A) at (1.75, 1.2) [] {$\gt^E$};
	\end{tikzpicture}
	\begin{tikzpicture}[scale=1.5]
\draw[draw=black] (0,0.5) --(1.5,0)--(1.5,1.5)--cycle;
\draw[draw=black] (1.5,0)--(1.5,1.5)--(3.1,0.9)--cycle;

\draw[draw=black,->] (1.2,0.65) to (0.75,0.85);
\draw[draw=red,->] (1.2,0.65) to (1.7,0.65);

\draw[draw=black,->] (2.1,0.75) to (1.75,0.4);
\draw[draw=red,->] (2.1,0.75) to (2.6,0.75);

\draw[draw=teal, thick,->] (1.5,0.7) to (1.5,1.2);


\centerarc[red, thick,-{Latex[length=1.3mm]}](1.2,0.65)(0:335-180:0.25);

\centerarc[red, thick,-{Latex[length=1.3mm]}](2.1,0.75)(0:405-180:0.25);

\node (A) at (1.2, 1.05) [] {\textcolor{red}{$\tilde\Theta_+^E$}};
\node (A) at (2.2, 1.15) [] {\textcolor{red}{$\tilde\Theta_-^E$}};

\node (A) at (0.4, 0.1) [] {$T_+$};
\node (A) at (2.2, 0.1) [] {$T_-$};


\node (A) at (1.25, 0.5) [] {\textcolor{red}{$-\gn_+$}};
\node (A) at (2.3, 0.6) [] {\textcolor{red}{$\gn_-$}};

\node (A) at (0.6, 1.1) [] {$\GBasis_{+,1}$};
\node (A) at (1.75, 0.32) [] {$\GBasis_{-,1}$};

\node[teal] (A) at (1.75, 1.2) [] {$\gt^E$};
\end{tikzpicture}
	\caption{Computation of angle difference. Left: $|\Theta^E_+-\Theta^E_-|<\pi$. Middle: $|\Theta^E_+-\Theta^E_-|>\pi$. Right: $g$-normal vector sign is swapped  according to~\eqref{eq:31}.}
	\label{fig:inconsistent_jump}
\end{figure}
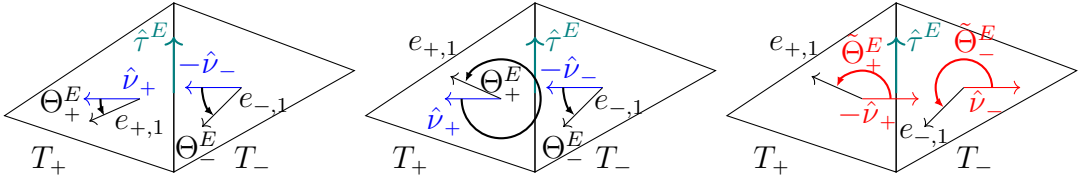

\begin{proposition}
	\label{prop:modelOF}
	Let
	$ \mt{\cn}_i = \frac{1}{2}g_{jk}(\d_i\GBasis_2^j+\Gamma_{li}^j\GBasis_2^l\GBasis_1^k-\d_i\GBasis_1^j-\Gamma_{li}^j\GBasis_1^l\GBasis_2^k)
	$ where $e_i$ is chosen as in~\eqref{eq:28}. Then 
	the connection 1-form $\OneForm_h(g)$ of
	Definition~\ref{def:distributional_one_form} satisfies
	\begin{equation*}
	\begin{aligned}
	\int_\om
	\delta(\OneForm_h(g),  v)\, \da
	= 
	\sum_{T \in \T}
	\bigg(& 
	\int_T
	\mt{\cn}_i v^i
	\; \da-
	\int_{\d T }
	\sphericalangle_g(\GBasis_1,s_E\,g(\gt_E,\gt)\gn)\,
	v^\nv\,\dl\bigg),
	\end{aligned}
	\end{equation*}
	for all
	$v\in \Wo_h^k$.
\end{proposition}
\begin{proof}
	By noting that $\pm \gn^E_\pm = g(\gt^E,\gt_{\pm})\gn^E_\pm$ we obtain
	\begin{align*}
	\sum_{E\in\Eint} \int_{E,g}{\Theta}^Eg(Q_gv,\gn^E)=\sum_{E\in\Eint} \int_{E}(\tilde{\Theta}^E_+-\tilde{\Theta}^E_-)v^{\nv^E}\,\dl = \sum_{ T\in \T}\int_{\d T}\sphericalangle_g(\GBasis_1,s_E\,g(\gt_E,\gt)\gn)v^{\nv}\,\dl.
	\end{align*}
	By symmetrization of \eqref{eq:oneform_gortho_frame} we have $\OneForm(g;e_i)= \frac{1}{2}g_{jk}(\d_i\GBasis_2^j+\Gamma_{li}^j\GBasis_2^l\GBasis_1^k-\d_i\GBasis_1^j-\Gamma_{li}^j\GBasis_1^l\GBasis_2^k)$ and the claim follows.
\end{proof}

\section{Relation between distributional covariant $\mathrm{inc}$ and $\div\div$}
\label{sec:rel_cov_inc_vov_divdiv}
In this section we show that the distributional covariant inc from Proposition~\ref{prop:cov_distr_inc} and the covariant distributional divdiv operator of a rotated sigma used in \cite{BKG21} coincide in the sense 
\begin{align}
\label{eq:identity_distr_cov_divdiv_curlcurl}
\act{\div_g\div_gS_g \sigma, u}_{\Vo(\T)}=-\act{\inc_g\sigma, u}_{\Vo(\T)}\qquad \forall u\in\Vo(\T),
\end{align}
where the covariant divergence is defined below. This is in common with the Euclidean identity $\div\div S\sigma = \div\div g-\Delta \tr{\sigma} = -\inc\sigma$, $\tr{\cdot}$ denoting the trace of a matrix. The distributional covariant divdiv reads
\begin{align}
\label{eq:def_distr_cov_divdiv}
\act{\div_g\div_gS_g \sigma, u}_{\Vo(\T)} =
\int_{\T} u \;\div_g\div_gS_g \sigma
&+ \int_{\d \T} u\; ( (\div_gS_g\sigma)^\flat(\gn) + (d^0\sigma_{\gn \gt})(\gt))\nonumber\\
&+ \sum_{T \in \T}\sum_{V \in \V_T}\jump[V]{\sigma_{\gn \gt}}^Tu(V),
\end{align}
where $S_g\sigma = \sigma - \mathrm{tr}_g(\sigma)g$ with $\mathrm{tr}_g(\sigma)=\sigma_{ij}g^{ij}$. (Note that the authors in \cite{BKG21} used $(\gn,\gt)$ as positively oriented frame, whereas we use $(\gt,\gn)$ such that the signs in the boundary and vertex terms differ. Also, a different orientation in the vertex jump is used.)

The covariant divergence is defined as the $L^2$-adjoint of the covariant gradient. For $f\in\W^0(\Omega)$ its covariant gradient is given by the equation
\begin{align*}
g(\grad_g f,v) = d^0 f(v) \qquad \forall v\in\Xm{\Omega}
\end{align*}
and $\div_g:\Xm{\Omega}\to\W^0(\Omega)$ by
\begin{align*}
\int_{\Omega}g(\grad_g f,v) = \int_{\Omega}f\,\div_g v, \quad \forall v\in\Wo_g(\Omega),f\in\W^0(\Omega),
\end{align*}
in coordinates
\begin{align*}
\div_g v &= \frac{1}{\sqrt{\det g}}\d_i(\sqrt{\det g}v^i),\quad v\in\Xm{\Omega}.
\end{align*}
The usual extension to tensor fields, see e.g. \cite{Gaw20}, $\div_g:\TT_0^2(\Omega)\to \Xm{\Omega}$ reads in coordinates
\begin{align*}
\div_g\sigma &= (\partial_j\sigma^{ij}+\Gamma_{lj}^i\sigma^{lj}+\Gamma_{jl}^j\sigma^{il}) \d _i,\qquad \sigma\in\TT^2_0(\Omega),
\end{align*}
where $\sigma^{ij} = g^{ik}\sigma_{kl}g^{lj}$ for $\sigma\in \TT_2^0(\Omega)$.

\begin{lemma}
	       \label{lem:identities_div_curl}
	       There holds for $g\in \Regge^+(\T)$ and $\sigma\in \Regge[\T]$
	       \begin{enumerate}
		               \item $\star (\div_gS_g\sigma)^\flat = -\curl_g\sigma$ and  $(\div_gS_g\sigma)^\flat(\gn) = (\curl_g\sigma)(\gt)$ on $\d\T$,
		               \item $\div_g\div_gS_g\sigma = \div_g\div_g\sigma-\Delta_g\mathrm{tr}_g(\sigma) = -\inc_g\sigma$ on $\T$,
		       \end{enumerate}
	       where $\Delta_gf:=\div_g \grad_g f$, $f\in \W^0(\Omega)$ denotes the Laplace-Beltrami operator.
	\end{lemma}
\begin{proof}
	       In the first identity only first order derivatives of $g$ are involved. Thus, in normal coordinates, $\xtt^i$ (see $\S$\ref{subsec:connection_curv_inc}), $(\div_gS_g\sigma)^\flat$ becomes the Euclidean version $\div(\sigma-\tr{\sigma}\Eucl)$, which reads in components
	       \begin{align*}
	       (\div_gS_g\sigma)^\flat = \mat{\dtt_2\sigmatt_{12}-\dtt_1\sigmatt_{22}\\ \dtt_1\sigmatt_{12}-\dtt_2\sigmatt_{11}}.
	       \end{align*}
	       The Hodge star operator performs a counterclockwise 90-degree rotation, $\startt=-\varepsilon^{ij}$, so that 
	       \begin{align*}
	       \star(\div_gS_g\sigma)^\flat = -\mat{\dtt_1\sigmatt_{12}-\dtt_2\sigmatt_{11}\\ \dtt_1\sigmatt_{22}-\dtt_2\sigmatt_{12}} = -\curl \sigmatt,
	       \end{align*}
	       which coincides with $-\curl_g\sigma$ in normal coordinates. The identity $(\div_gS_g\sigma)^\flat(\gn) = (\curl_g\sigma)(\gt)$ follows now by \eqref{eq:star-rotation}.

	       The identity $\div_g\div_gS_g\sigma = \div_g\div_g\sigma -\Delta_g\mathrm{tr}_g(\sigma)$ is well known~\cite{Gaw20}, so we focus on proving its relationship with $\inc$, by means of normal coordinates. The Laplace-Beltrami operator becomes
	       \begin{align*}
	       \mt{\Delta_g\mathrm{tr}_g(\sigma)} = \mt{\div_g\nabla_g\mathrm{tr}_g(\sigma)} =\mt{\div_g(g^{ij}\d_j \mathrm{tr}_g(\sigma))} = \dtt_i^2\tr{\sigmatt \gtt^{-1}} = \Delta \tr{\sigmatt}-\tr{\sigmatt\Delta \gtt}
	       \end{align*}
	       and the divdiv part becomes 
	       \begin{align*}
	       \mt{\div_g\div_g\sigma}= \div\div\mt\sigmatt-2\dtt^2_{ij}\gtt_{ik}\sigmatt_{kj} +\dtt_i\Gamtt_{lji}\sigmatt_{lj}-\dtt_i\Gamtt_{jlj}\sigmatt_{il}.
	       \end{align*}
	       Note that we abused notation and summed over repeated indices all of which are subscripts (forgivable while using normal coordinates).	       
	       Furthermore, by inserting the definition of Christoffel symbols of the  first kind, a lengthy but elementary computation gives
	       \begin{align*}
	       -2\dtt^2_{ij}\gtt_{ik}\sigmatt_{kj} +\dtt_i\Gamtt_{lji}\sigmatt_{lj}&-\dtt_i\Gamtt_{jlj}\sigmatt_{il} = \sigmatt_{11}\big(\frac{1}{2}\dtt_1^2\gtt_{22}-\dtt_1^2\gtt_{11}-\frac{1}{2}\dtt_2^2\gtt_{11}-\dtt_1\dtt_2\gtt_{12}\big)\\
	       &-2\sigmatt_{12}\big(\dtt_1^2\gtt_{12}+\dtt_2^2\gtt_{12}\big)
	       -\sigmatt_{22}\big(\dtt_2^2\gtt_{22}+\frac{1}{2}\dtt_1^2\gtt_{22}+\dtt_1\dtt_2\gtt_{12}-\frac{1}{2}\dtt_2^2\gtt_{11}\big).
	       \end{align*}
	       Combining these,  
	       \begin{align*}
	       \mt{\div_g\div_gS_g\sigma} = -\inc\mt\sigmatt +\frac{1}{2}\mathrm{tr} \mt\sigmatt\inc \mt\gtt,
	       \end{align*}
	       which finishes the proof by comparing with \eqref{eq:cov_inc_normal_coord}.
	\end{proof}

Using the results of Lemma~\ref{lem:identities_div_curl} and comparing the terms of \eqref{eq:def_distr_cov_divdiv} with Proposition~\ref{prop:cov_distr_inc} shows that identity \eqref{eq:identity_distr_cov_divdiv_curlcurl} holds.

\section*{Acknowledgments}

%

This work was supported in part by the Austrian Science Fund (FWF)
project F65 and NSF grant DMS-1912779.

\bibliographystyle{acm}
\bibliography{cites}

\begin{thebibliography}{10}

\bibitem{AG16}
{\sc Amstutz, S., and Van~Goethem, N.}
\newblock Analysis of the incompatibility operator and application in intrinsic
  elasticity with dislocations.
\newblock {\em SIAM Journal on Mathematical Analysis 48}, 1 (2016), 320--348.

\bibitem{AG20}
{\sc Amstutz, S., and Van~Goethem, N.}
\newblock Existence and asymptotic results for an intrinsic model of
  small-strain incompatible elasticity.
\newblock {\em Discrete \& Continuous Dynamical Systems - B 25}, 10 (2020),
  3769--3805.

\bibitem{Arnol18}
{\sc Arnold, D.~N.}
\newblock {\em Finite element exterior calculus}.
\newblock SIAM, Philadelphia, 2018.

\bibitem{ArnolAwanoWinth08}
{\sc Arnold, D.~N., Awanou, G., and Winther, R.}
\newblock Finite elements for symmetric tensors in three dimensions.
\newblock {\em Math. Comp. 77}, 263 (2008), 1229--1251.

\bibitem{arnold10}
{\sc Arnold, D.~N., Falk, R., and Winther, R.}
\newblock Finite element exterior calculus: from {H}odge theory to numerical
  stability.
\newblock {\em Bulletin of the American Mathematical Society 47}, 2 (2010),
  281--354.

\bibitem{arnold06}
{\sc Arnold, D.~N., Falk, R.~S., and Winther, R.}
\newblock Finite element exterior calculus, homological techniques, and
  applications.
\newblock {\em Acta numerica 15\/} (2006), 1--155.

\bibitem{AH21}
{\sc Arnold, D.~N., and Hu, K.}
\newblock Complexes from complexes.
\newblock {\em Foundations of Computational Mathematics 21}, 6 (2021),
  1739--1774.

\bibitem{AW20}
{\sc Arnold, D.~N., and Walker, S.~W.}
\newblock The {H}ellan--{H}errmann--{J}ohnson method with curved elements.
\newblock {\em SIAM Journal on Numerical Analysis 58}, 5 (2020), 2829--2855.

\bibitem{BOW18}
{\sc Barrett, J.~W., Oriti, D., and Williams, R.~M.}
\newblock Tullio {R}egge's legacy: {R}egge calculus and discrete gravity.
\newblock {\em arXiv preprint arXiv:1812.06193\/} (2018).

\bibitem{BKG21}
{\sc Berchenko-Kogan, Y., and Gawlik, E.~S.}
\newblock Finite element approximation of the {L}evi-{C}ivita connection and
  its curvature in two dimensions, 2021.

\bibitem{BBF13}
{\sc Boffi, D., Brezzi, F., and Fortin, M.}
\newblock {\em Mixed finite element methods and applications}, 1~ed., vol.~44.
\newblock Springer-Verlag Berlin Heidelberg, Berlin, Heidelberg, 2013.

\bibitem{BCM03}
{\sc Borrelli, V., Cazals, F., and Morvan, J.-M.}
\newblock On the angular defect of triangulations and the pointwise
  approximation of curvatures.
\newblock {\em Computer Aided Geometric Design 20}, 6 (2003), 319--341.

\bibitem{BDM85}
{\sc Brezzi, F., Douglas, J., and Marini, L.~D.}
\newblock Two families of mixed finite elements for second order elliptic
  problems.
\newblock {\em Numerische Mathematik 47}, 2 (1985), 217--235.

\bibitem{Carmo1992}
{\sc Carmo, M. P.~d.}
\newblock {\em Riemannian Geometry}, 1~ed.
\newblock Birkh\"auser Basel, Mannheim, 1992.

\bibitem{Cheeger84}
{\sc Cheeger, J., M{\"u}ller, W., and Schrader, R.}
\newblock On the curvature of piecewise flat spaces.
\newblock {\em Communications in Mathematical Physics 92}, 3 (1984), 405--454.

\bibitem{ChrisGopalGuzma21}
{\sc Christiansen, S., Gopalakrishnan, J., Guzm{\'{a}}n, J., and Hu, K.}
\newblock A discrete elasticity complex on three-dimensional {Alfeld} splits.
\newblock {\em Preprint\/} (2021).

\bibitem{christiansen04}
{\sc Christiansen, S.~H.}
\newblock A characterization of second-order differential operators on finite
  element spaces.
\newblock {\em Mathematical Models and Methods in Applied Sciences 14}, 12
  (2004), 1881--1892.

\bibitem{christiansen11}
{\sc Christiansen, S.~H.}
\newblock On the linearization of {R}egge calculus.
\newblock {\em Numerische Mathematik 119}, 4 (2011), 613--640.

\bibitem{christiansen15}
{\sc Christiansen, S.~H.}
\newblock Exact formulas for the approximation of connections and curvature,
  2015.

\bibitem{ClarkDray87}
{\sc Clarke, C. J.~S., and Dray, T.}
\newblock Junction conditions for null hypersurfaces.
\newblock {\em Classical and Quantum Gravity 4}, 2 (1987), 265--275.

\bibitem{Com89}
{\sc Comodi, M.~I.}
\newblock The {H}ellan--{H}errmann--{J}ohnson method: Some new error estimates
  and postprocessing.
\newblock {\em Mathematics of Computation 52}, 185 (1989), 17--29.

\bibitem{CrouzThome87}
{\sc Crouzeix, M., and Thom{\'{e}}e, V.}
\newblock {The Stability in $L_p$ and $W_p^1$ of the $L_2$-projection onto
  Finite Element Function Spaces}.
\newblock {\em Math. Comp 48}, 178 (April 1987), 521--532.

\bibitem{FischMarsd75}
{\sc Fischer, A.~E., and Marsden, J.~E.}
\newblock Deformations of the scalar curvature.
\newblock {\em Duke Math. J.\/} (1975).

\bibitem{Fritz13}
{\sc Fritz, H.}
\newblock Isoparametric finite element approximation of {R}icci curvature.
\newblock {\em IMA Journal of Numerical Analysis 33}, 4 (2013), 1265--1290.

\bibitem{Fritz15}
{\sc Fritz, H.}
\newblock Numerical {R}icci--{D}e{T}urck flow.
\newblock {\em Numerische Mathematik 131}, 2 (2015), 241--271.

\bibitem{Gawlik19}
{\sc Gawlik, E.~S.}
\newblock Finite element methods for geometric evolution equations.
\newblock In {\em Geometric Science of Information\/} (Cham, 2019), F.~Nielsen
  and F.~Barbaresco, Eds., Springer International Publishing, pp.~532--540.

\bibitem{Gaw20}
{\sc Gawlik, E.~S.}
\newblock High-order approximation of {G}aussian curvature with {R}egge finite
  elements.
\newblock {\em SIAM Journal on Numerical Analysis 58}, 3 (2020), 1801--1821.

\bibitem{Hauret13}
{\sc Hauret, P., and Hecht, F.}
\newblock A discrete differential sequence for elasticity based upon continuous
  displacements.
\newblock {\em SIAM Journal on Scientific Computing 35}, 1 (2013), B291--B314.

\bibitem{Israe66}
{\sc Israel, W.}
\newblock Singular hypersurfaces and thin shells in general relativity.
\newblock {\em Il Nuovo Cimento B Series 44}, 1 (1966), 1--14.

\bibitem{Kosov02}
{\sc Kosovski{\u \i}, N.~N.}
\newblock Gluing of {R}iemannian manifolds of curvature {$\geq \kappa$}.
\newblock {\em Algebra i Analiz 14}, 3 (2002), 140--157.

\bibitem{Kosov04}
{\sc Kosovski{\u \i}, N.~N.}
\newblock Gluing with branching of {R}iemannian manifolds of curvature {$\leq
  \kappa$}.
\newblock {\em Algebra i Analiz 16}, 4 (2004), 132--145.

\bibitem{Lee97}
{\sc Lee, J.~M.}
\newblock {\em Riemannian manifolds: an introduction to curvature}, 1~ed.
\newblock Springer, New York, NY, New York, 1997.

\bibitem{Lee12b}
{\sc Lee, J.~M.}
\newblock {\em Introduction to Smooth Manifolds}, 2~ed.
\newblock Springer, 2012.

\bibitem{li18}
{\sc Li, L.}
\newblock {\em Regge Finite Elements with Applications in Solid Mechanics and
  Relativity}.
\newblock PhD thesis, University of Minnesota, 2018.

\bibitem{LX07}
{\sc Liu, D., and Xu, G.}
\newblock Angle deficit approximation of {G}aussian curvature and its
  convergence over quadrilateral meshes.
\newblock {\em Computer-Aided Design 39}, 6 (2007), 506--517.

\bibitem{Neun21}
{\sc Neunteufel, M.}
\newblock {\em Mixed Finite Element Methods for Nonlinear Continuum Mechanics
  and Shells}.
\newblock PhD thesis, TU Wien, 2021.

\bibitem{NS21}
{\sc Neunteufel, M., and Sch\"oberl, J.}
\newblock Avoiding membrane locking with {R}egge interpolation.
\newblock {\em Computer Methods in Applied Mechanics and Engineering 373\/}
  (2021), 113524.

\bibitem{Peter16}
{\sc Petersen, P.}
\newblock {\em Riemannian Geometry}, third~ed.
\newblock Springer, 2016.

\bibitem{RT77}
{\sc Raviart, P.-A., and Thomas, J.-M.}
\newblock A mixed finite element method for 2-nd order elliptic problems.
\newblock In {\em Mathematical Aspects of Finite Element Methods}, vol.~66.
  Springer, 1977, pp.~292--315.

\bibitem{Regge61}
{\sc Regge, T.}
\newblock General relativity without coordinates.
\newblock {\em Il Nuovo Cimento (1955-1965) 19}, 3 (1961), 558--571.

\bibitem{RW00}
{\sc Regge, T., and Williams, R.~M.}
\newblock Discrete structures in gravity.
\newblock {\em Journal of Mathematical Physics 41}, 6 (2000), 3964--3984.

\bibitem{Sch97}
{\sc Sch{\"o}berl, J.}
\newblock {NETGEN} an advancing front 2{D}/3{D}-mesh generator based on
  abstract rules.
\newblock {\em Computing and Visualization in Science 1}, 1 (1997), 41--52.

\bibitem{Sch14}
{\sc Sch{\"o}berl, J.}
\newblock C++ 11 implementation of finite elements in {NGS}olve.
\newblock {\em Institute for Analysis and Scientific Computing, Vienna
  University of Technology\/} (2014).

\bibitem{Sorkin75}
{\sc Sorkin, R.}
\newblock Time-evolution problem in {R}egge calculus.
\newblock {\em Phys. Rev. D 12\/} (1975), 385--396.

\bibitem{Stric20}
{\sc Strichartz, R.~S.}
\newblock Defining curvature as a measure via gauss--bonnet on certain singular
  surfaces.
\newblock {\em The Journal of Geometric Analysis 30}, 1 (2020), 153--160.

\bibitem{Str2020}
{\sc Strichartz, R.~S.}
\newblock Defining {{Curvature}} as a {{Measure}} via
  {{Gauss}}\textendash{{Bonnet}} on {{Certain Singular Surfaces}}.
\newblock {\em The Journal of Geometric Analysis 30}, 1 (2020), 153--160.

\bibitem{Sullivan08}
{\sc Sullivan, J.~M.}
\newblock {\em Curvatures of Smooth and Discrete Surfaces}.
\newblock Birkh{\"a}user Basel, Basel, 2008, pp.~175--188.

\bibitem{Tu17}
{\sc Tu, L.~W.}
\newblock {\em Differential Geometry: Connections, Curvature, Characteristic
  Classes}.
\newblock Springer, 2017.

\bibitem{Walker21}
{\sc Walker, S.~W.}
\newblock {The {K}irchhoff plate equation on surfaces: the surface
  {H}ellan–{H}errmann–{J}ohnson method}.
\newblock {\em IMA Journal of Numerical Analysis\/} (2021).

\bibitem{Warde06}
{\sc Wardetzky, M.}
\newblock {\em Discrete differential operators on polyhedral
  surfaces---convergence and approximation}.
\newblock PhD thesis, Freie Universit{\"a}t Berlin, 2006.

\bibitem{whitney57}
{\sc Whitney, H.}
\newblock {\em Geometric integration theory}.
\newblock Princeton University Press, Princeton, N. J, 1957.

\bibitem{williams92}
{\sc Williams, R.~M., and Tuckey, P.~A.}
\newblock Regge calculus: a brief review and bibliography.
\newblock {\em Classical and Quantum Gravity 9}, 5 (1992), 1409--1422.

\bibitem{Xu06}
{\sc Xu, G.}
\newblock Convergence analysis of a discretization scheme for {G}aussian
  curvature over triangular surfaces.
\newblock {\em Computer Aided Geometric Design 23}, 2 (2006), 193--207.

\bibitem{XX09}
{\sc Xu, Z., and Xu, G.}
\newblock Discrete schemes for {G}aussian curvature and their convergence.
\newblock {\em Computers \& Mathematics with Applications 57}, 7 (2009),
  1187--1195.

\end{thebibliography}

\end{document}